\pdfoutput=1
\documentclass{amsart}

\usepackage{mathtools}
\usepackage{latexsym}
\usepackage{amssymb}
\usepackage{braket}
\usepackage{mathrsfs}
\usepackage{ifthen}
\usepackage{here}
\usepackage{todonotes}
\usepackage{tikz}
\usetikzlibrary{patterns,decorations.pathreplacing,calligraphy}
\usepackage{comment} 
\usepackage{mleftright}
\usepackage{amsmath, amsthm,verbatim,amsfonts, graphicx, enumerate}
\usepackage{thmtools,thm-restate}

\usepackage[dvipsnames]{xcolor}
\usepackage[pagebackref,hypertexnames=false,colorlinks=true,
            linkcolor=NavyBlue,
            citecolor=NavyBlue,
            urlcolor=NavyBlue]{hyperref} 
\usepackage{cleveref}
 \usepackage[all]{xy}
 \usepackage{amscd}
 \usepackage[alphabetic,backrefs,msc-links]{amsrefs}
 \usepackage{color}
 \usepackage{enumitem}
\newlist{steps}{enumerate}{1}
\setlist[steps, 1]{label = Step \arabic*:}
\usepackage[abs]{overpic}
\usepackage{tikz-cd}
\usepackage{pinlabel}

\newcounter{casenum}

\usepackage{geometry}
\makeatletter
\DeclareRobustCommand\widecheck[1]{{\mathpalette\@widecheck{#1}}}
\def\@widecheck#1#2{%
   \setbox\z@\hbox{\m@th$#1#2$}%
   \setbox\tw@\hbox{\m@th$#1%
      {%
         \vrule\@width\z@\@height\ht\z@
         \vrule\@height\z@\@width\wd\z@}$}%
   \dp\tw@-\ht\z@
   \@tempdima\ht\z@ \advance\@tempdima2\ht\tw@ \divide\@tempdima\thr@@
   \setbox\tw@\hbox{%
      \raise\@tempdima\hbox{\scalebox{1}[-1]{\lower\@tempdima\box\tw@}}}%
   {\ooalign{\box\tw@ \cr \box\z@}}}
\makeatother

\theoremstyle{plain}
\newtheorem{thm}{Theorem}[section]
\crefname{thm}{Theorem}{Theorems}
\Crefname{thm}{Theorem}{Theorems}
\newtheorem{prop}[thm]{Proposition}
\crefname{prop}{Proposition}{Propositions}
\Crefname{prop}{Proposition}{Propositions}
\newtheorem{lem}[thm]{Lemma}
\crefname{lem}{Lemma}{Lemmas}
\Crefname{lem}{Lemma}{Lemmas}
\newtheorem{cor}[thm]{Corollary}
\crefname{cor}{Corollary}{Corollaries}
\Crefname{cor}{Corollary}{Corollaries}

\crefname{claim}{Claim}{Claims}
\Crefname{claim}{Claim}{Claims}

\crefname{property}{Property}{Properties}
\Crefname{property}{Property}{Properties}

\crefname{problem}{Problem}{Problems}
\Crefname{problem}{Problem}{Problems}

\crefname{conjecture}{Conjecture}{Conjecture}
\Crefname{conjecture}{Conjecture}{Conjecture}

\theoremstyle{definition}
\newtheorem{defn}[thm]{Definition}
\crefname{defn}{Definition}{Definitions}
\Crefname{defn}{Definition}{Definitions}

\crefname{notation}{Notation}{Notations}
\Crefname{notation}{Notation}{Notations}

\crefname{convention}{Convention}{Conventions}
\Crefname{convention}{Convention}{Conventions}

\crefname{cond}{Condition}{Conditions}
\Crefname{cond}{Condition}{Conditions}

\crefname{assum}{Assumption}{Assumptions}
\Crefname{assum}{Assumption}{Assumptions}

\crefname{conj}{Conjecture}{Conjectures}
\Crefname{conj}{Conjecture}{Conjectures}

\crefname{claim1}{Claim}{Claims}
\Crefname{claim1}{Claim}{Claims}
\Crefname{ques}{Question}{Question}
\newtheorem{ques}[thm]{Question}
\crefname{que}{Question}{Question}
\Crefname{que}{Question}{Question}

\theoremstyle{remark}
\newtheorem{rem}[thm]{Remark}
\crefname{rem}{Remark}{Remarks}
\Crefname{rem}{Remark}{Remarks}
\newtheorem{ex}[thm]{Example}
\crefname{ex}{Example}{Examples}
\Crefname{ex}{Example}{Examples}

\crefname{section}{Section}{Sections}
\Crefname{section}{Section}{Sections}
\crefname{subsection}{Subsection}{Subsections}
\Crefname{subsection}{Subsection}{Subsections}
\crefname{figure}{Figure}{Figures}
\Crefname{figure}{Figure}{Figures}

\newcommand{\spinc}{\mathrm{spin}^c}
\newcommand{\Z}{\mathbb{Z}}

\newcommand{\Q}{\mathbb{Q}}

\newcommand{\id}{\mathrm{id}}
\newcommand{\ind}{\mathop{\mathrm{ind}}\nolimits}

\newcommand{\C}{\mathbb{C}}
\newcommand{\s}{\mathfrak{s}}

\newcommand{\R}{\mathbb R}

\def\ker{\operatorname{Ker}}

\def\det{\operatorname{det}}

\def\dim{\operatorname{dim}}

\def\id{\operatorname{Id}}

\def\ind{\operatorname{ind}}

\newcommand{\mbar}[1]{{\ooalign{\hfil#1\hfil\crcr\raise.167ex\hbox{--}}}}

\def\wt{\widetilde}

\newcommand{\Pin}{\operatorname{Pin}}

     \RequirePackage{rotating}                   
    \def\HMt{%
       \setbox0=\hbox{$\widehat{\mathit{HM}}$}
       \setbox1=\hbox{$\mathit{HM}$}
       \dimen0=1.1\ht0
       \advance\dimen0 by 1.17\ht1
       \smash{\mskip2mu\raise\dimen0\rlap{%
          \begin{turn}{180}
              {$\widehat{\phantom{\mathit{HM}}}$}
           \end{turn}} \mskip-2mu    
                \mathit{HM}
                    }{\vphantom{\widehat{\mathit{HM}}}}{}}

\title[$\widetilde{H}$-cobordisms, infinite cyclic covers, and real Seiberg--Witten theory]{$\widetilde{H}$-cobordisms, infinite cyclic covers,\\ and real Seiberg--Witten theory}

\author{Sungkyung Kang}
\address{Department of Pure Mathematics and Mathematical Statistics, University of Cambridge, United Kingdom}
\email{sungkyung.kang@dpmms.cam.ac.uk}

\author{JungHwan Park}
\address{Department of Mathematical Sciences and RIM, Seoul National University, Republic of Korea}
\email{jungpark0817@snu.ac.kr}

\author{Masaki Taniguchi} 
\address{Department of Mathematics, Kyoto University, Japan}
\email{taniguchi.masaki.7m@kyoto-u.ac.jp}

\begin{document}

\begin{abstract}
We study $\wt{H}$-cobordisms of distinguished homology handles, introduced by Kawauchi in 1976 using infinite cyclic covers. Despite the extensive development of gauge-theoretic and Floer-theoretic invariants since Kawauchi's work, none were previously known to distinguish smooth and topological $\wt{H}$-cobordism. In this paper, we construct asymptotic invariants of distinguished homology handles by applying real Seiberg--Witten theory to finite cyclic covers. Using these invariants, we show that the kernel of the natural map from the smooth $\wt{H}$-cobordism group to its topological counterpart contains a subgroup isomorphic to $\Z$. We also define spin versions of these groups and show that the kernel of the corresponding natural map contains a subgroup isomorphic to $\Z^\infty$.
\end{abstract}

\maketitle

%\tableofcontents

% \textcolor{red}{SK: My philosophy for writing section 2 and 3 was that, whenever something can be proven in full generality (like any field instead of characteristic 0), then I would write in that way. The reason is that I expect this paper to be the starting point or a standard reference for a direction towards studying infinite cyclic covers via gauge theory, and I would not want anyone having to write a generalized variant of lemmas here just because we have not written in that way. Please let me know if anything more can still be generalized.}

\section{Introduction}

Since the groundbreaking work of Donaldson \cite{Donaldson:1983} and Freedman \cite{Freedman:1982}, distinguishing the smooth and topological categories has been a central theme in 4-dimensional topology. Gauge theory and Floer theory have since revealed rich structures reflecting this distinction in knot concordance and homology cobordism; see, for example, \cite{Fintushel-Stern:1985, Cochran-Gompf:1988,Furuta:1990,Endo:1995,Fr02,HLR:2012,Hedden-Kirk:2012,CHH:2013,Hom:2015,HKL:2016,OSS:2017,Pinzon-Caicedo:2017,Kim-Park:2018,FPR:2019,Daemi:2020,Cha-Kim:2021,DHST:2021,Cha:2021,dai2023infinite,Chen:2024,Nozaki-Sato-Taniguchi:2024}. It has remained unclear, however, whether this distinction can also be detected by $\wt{H}$-cobordism, introduced by Kawauchi in 1976 \cite{Kawauchi:1976}. This relation gives rise to an abelian group of distinguished homology handles under circle sum. A central difficulty in studying its smooth aspects is that the defining homological condition concerns an infinite cyclic cover, so the usual gauge-theoretic and Floer-theoretic tools for homology cobordism do not apply directly. In this paper, we initiate a gauge-theoretic study of this relation by applying the real Seiberg--Witten Floer theory developed by Konno--Miyazawa--Taniguchi~\cite{KMT21,konno2024involutions} to finite cyclic covers and extracting invariants from their behavior as the covering degree tends to infinity. As a consequence, we give the first known examples of nontrivial smooth $\wt{H}$-cobordism classes which become trivial in the topological category.

More precisely, a \emph{distinguished homology handle} is a closed oriented $3$-manifold with the integral homology of $S^1\times S^2$, equipped with a chosen generator of its first integral cohomology. An \emph{$\wt{H}$-cobordism} between two such handles is a compact oriented cobordism $W$ equipped with an integral cohomology class extending the chosen generators, such that the associated infinite cyclic cover has finite-dimensional rational homology. Equivalently, the defining homological condition is $H_\ast(W;\Q(t))=0$, where the local coefficients are determined by the distinguished class. Thus $\wt{H}$-cobordism can also be viewed as $L^2$-acyclic cobordism over $\Z$; see \cite[Lemma~1.34]{Lueck:2002} and the account of work of Cappell--Davis--Weinberger in \cite{Davis:2015}. Not every $\wt{H}$-cobordism is an integral homology cobordism. However, by \Cref{lem: homology cobordism implies Kawauchi}, every distinguished integral homology cobordism is an $\wt{H}$-cobordism. Thus $\wt{H}$-cobordism generalizes integral homology cobordism.

One advantage of this generalization is that distinguished homology handles modulo $\wt{H}$-cobordism form an abelian group under \emph{circle sum}, as proved by Kawauchi \cite[Section~1]{Kawauchi:1976}. This group is called the \emph{$\wt{H}$-cobordism group} and is denoted by $\Omega(S^1\times S^2)$. Roughly speaking, the circle sum is obtained by removing tubular neighborhoods of oriented loops on which the distinguished classes evaluate to $1$ and gluing the resulting boundary tori compatibly with these classes; see \Cref{sec: real spin and spin c structures} for the precise construction.

Moreover, Kawauchi showed that $0$-surgery induces a well-defined homomorphism from the smooth knot concordance group $\mathcal{C}$,
\[
\mathcal{S}\colon\mathcal{C}\longrightarrow\Omega(S^1\times S^2);
\qquad [K]\longmapsto[S^3_0(K)],
\]
and that the algebraic concordance class of $K$ is determined by $\mathcal{S}([K])$. More precisely, he constructed a surjective homomorphism $\psi$ fitting into the commutative diagram
\[
\begin{tikzcd}[column sep=2.5em, row sep=1.5em]
\mathcal{C} \arrow[rr,"\mathcal{S}"] \arrow[dr,swap,"\phi"] & & \Omega(S^1\times S^2) \arrow[dl,"\psi"] \\
& \mathcal{AC} &
\end{tikzcd}
\]
where $\mathcal{AC}$ denotes the algebraic concordance group and $\phi$ is Levine's natural surjective homomorphism defined using Seifert matrices \cite{Levine:1969}. In particular, $\Omega(S^1\times S^2)$ has infinite rank. Furthermore, the topological counterparts naturally fit into the following commutative diagram \cite[Section~2]{Kawauchi:1976}:
\[
\begin{tikzcd}[column sep=2.5em, row sep=1.5em]
\;\;\mathcal{C}\;\; \arrow[rr,"\mathcal{S}"] \arrow[d,swap,"\iota_{\mathcal{C}}"] & & \Omega(S^1\times S^2) \arrow[d,"\iota"] \\
\;\;\mathcal{C}^{\mathrm{top}}\;\; \arrow[rr,"\mathcal{S}^{\mathrm{top}}"] \arrow[dr,swap,"\phi'"] & & \Omega^{\mathrm{top}}(S^1\times S^2) \arrow[dl,"\psi'"] \\
& \;\mathcal{AC}\;&
\end{tikzcd}
\]
Here $\iota_{\mathcal{C}}$ and $\iota$ are the natural forgetful maps, while $\phi'$ and $\psi'$ are the topological counterparts of $\phi$ and $\psi$, respectively. Thus $\phi=\phi'\circ\iota_{\mathcal{C}}$ and $\psi=\psi'\circ\iota$. The kernels of $\mathcal{S}$ and $\mathcal{S}^{\mathrm{top}}$ each contain a subgroup isomorphic to $\Z^\infty$ \cite{Litherland:1984,lee2023kernel}, while the kernel of $\phi'$ contains a subgroup isomorphic to $\Z^\infty\oplus\Z_2^\infty$ \cite{Casson-Gordon:1978,Jiang:1981,Livingston:1999}. Moreover, Cha announced results implying that the kernel of $\psi'$ contains a subgroup isomorphic to $\Z^\infty$ \cite{Cha:2016}.

The main motivation for this paper is to study the $\wt{H}$-cobordism analogue of the kernel of the forgetful homomorphism
\[
\iota_{\mathcal{C}}\colon\mathcal{C}\longrightarrow\mathcal{C}^{\mathrm{top}}.
\]
The kernel of $\iota_{\mathcal{C}}$ has attracted significant attention, as it illustrates how the distinction between the topological and smooth categories in dimension 4 can be detected through knot theory. In particular, its nontriviality implies the existence of an exotic $\mathbb{R}^4$, that is, a smooth manifold homeomorphic but not diffeomorphic to the standard $\mathbb{R}^4$; see, for instance, \cite[Exercise~9.4.23]{Gompf-Stipsicz:1999}. The kernel has been studied extensively, including in many of the works cited in the first paragraph, and we do not attempt to review the full scope of known results here. It is worth emphasizing, however, that the kernel of $\iota_{\mathcal{C}}$ contains a direct summand of infinite rank~\cite{Hom:2015,OSS:2017,DHST:2021}, as well as a subgroup isomorphic to $\Z_2^\infty$~\cite{HKL:2016}.

In contrast, it has remained unknown whether the kernel of $\iota$ is even nontrivial. The rich structure of $\ker\iota_{\mathcal{C}}$ does not by itself resolve this question, since the zero-surgery homomorphism $\mathcal{S}$ can annihilate classes that distinguish smooth from topological knot concordance. A related result was established by Sato under a fundamental-group restriction \cite[Theorem~4]{sato:1998}. Specifically, he exhibited homology handles that bound compact topological $4$-manifolds homotopy equivalent to $S^1$, but admit no compact smooth filling with the same homotopy type. Our main result shows that the distinction between the two categories persists for $\wt{H}$-cobordism.\footnote{In particular, smooth and topological $L^2$-acyclic cobordism over $\Z$ differ even when restricted to distinguished homology handles.}

% In contrast, it has remained unknown whether the kernel of $\iota$ is even nontrivial. The rich structure of $\ker\iota_{\mathcal{C}}$ does not by itself resolve this question, since the zero-surgery homomorphism $\mathcal{S}$ can annihilate classes that distinguish smooth from topological knot concordance. Our main result shows that the distinction between the two categories persists for $\wt{H}$-cobordism.\footnote{In particular, smooth and topological $L^2$-acyclic cobordism over $\Z$ differ even when restricted to distinguished homology handles.}

\begin{restatable}{thm}{mainkerneltheorem}\label{thm:main}
The kernel of
\[
\iota\colon\Omega(S^1\times S^2)\longrightarrow\Omega^{\mathrm{top}}(S^1\times S^2)
\]
contains infinitely many nontrivial elements. In fact, it contains a subgroup isomorphic to $\Z$.
\end{restatable}

The infinite cyclic subgroup constructed in the proof of \Cref{thm:main} is generated by $[(N,\phi)]$, where
\[
N \coloneqq \Sigma(2,3,5)\#(S^1\times S^2)
\]
and $\phi\in H^1(N;\Z)$ is the distinguished generator coming from the $S^1$ factor. Let $M$ be the infinite cyclic cover of $N$ determined by $\phi$, equipped with its deck $\Z$-action. We next give two applications of this example.

There has been considerable recent interest in the existence of equivariant fillings of homology $3$-spheres and, more generally, rational homology $3$-spheres, even when suitable nonequivariant smooth fillings exist; see, e.g., \cite{Dai-Hedden-Mallick:2023,Dai-Kang-Mallick-Park-Stoffregen:2024,KMT21,Alfieri-Dai-Mallick-Taniguchi:2023,Kang-Park-Taniguchi:2026}. The following corollary gives an analogue of this phenomenon for infinite cyclic actions on noncompact $3$-manifolds. We call a connected oriented smooth $3$-manifold $M$ a \emph{$\Z$-periodic homology $\R\times S^2$} if it has the integral homology of $\R\times S^2$ and is equipped with a free, properly discontinuous, cocompact, orientation-preserving smooth $\Z$-action. Throughout this discussion, a \emph{free $\Z$-equivariant filling} of $M$ means a connected oriented $4$-manifold $W$ with $\partial W=M$, equipped with a free, properly discontinuous, cocompact $\Z$-action extending the given action on $M$.

\begin{restatable}{cor}{nonextendablecor}\label{cor: nonextendable cyclic action}
There exists a $\Z$-periodic homology $\R\times S^2$ manifold $M$ that admits a smooth contractible filling and satisfies the following:
\begin{enumerate}
\item $M$ admits a free $\Z$-equivariant topological filling that is contractible;
\item $M$ admits no free $\Z$-equivariant smooth filling with finite-dimensional total rational homology. In particular, $M$ admits no free $\Z$-equivariant smooth filling with the rational homology of $\R\times B^3$.
\end{enumerate}
\end{restatable}

This nonexistence result admits a quantitative refinement: the second Betti numbers of finite cyclic quotients of any free $\Z$-equivariant smooth filling of $M$ grow at least linearly, whereas they can all vanish for a topological filling. The linear lower bound is sharp; see \Cref{cor: cyclic cover Betti growth} in \Cref{sec: additive invariants}.

More generally, we extend Kawauchi's notion of $\wt{H}$-cobordism to arbitrary fields. For a field $F$, we define \emph{$\wt{H}_F$-cobordism} by requiring the associated infinite cyclic cover to have finite-dimensional homology with coefficients in $F$; $\wt{H}$-cobordisms then correspond to the case $F=\Q$. In both the smooth and topological categories, distinguished homology handles modulo this relation form an abelian group under circle sum; see \Cref{sec: real spin and spin c structures}. We denote these groups by $\Omega_F(S^1\times S^2)$ and $\Omega_F^{\mathrm{top}}(S^1\times S^2)$, respectively. The distinction between the smooth and topological categories persists over every field.

\begin{restatable}{cor}{Fkerneltheorem}\label{cor:main}
For any field $F$, the kernel of
\[
\iota_F\colon\Omega_F(S^1\times S^2)\longrightarrow\Omega_F^{\mathrm{top}}(S^1\times S^2)
\]
contains a subgroup isomorphic to $\Z$.
\end{restatable}

Over $\mathbb{F}_2$, we obtain a stronger result. In this case, we use the notation
\[
\Omega_{\mathrm{spin}}(S^1\times S^2)\coloneqq\Omega_{\mathbb{F}_2}(S^1\times S^2),
\qquad
\Omega_{\mathrm{spin}}^{\mathrm{top}}(S^1\times S^2)\coloneqq\Omega_{\mathbb{F}_2}^{\mathrm{top}}(S^1\times S^2).
\]
This terminology is motivated by the fact that sufficiently large $2$-power cyclic covers of smooth $\wt{H}_{\mathbb{F}_2}$-cobordisms are spin and, moreover, admit real spin structures with respect to the deck involutions of the successive double covers; see \Cref{lem: Z2-Kawauchi implies spin covers,lem: Z2-Kawauchi implies real spin}.

\begin{restatable}{thm}{mainspinkerneltheorem}\label{thm: main2}
The kernel of
\[
\iota_{\mathrm{spin}}\colon\Omega_{\mathrm{spin}}(S^1\times S^2)\longrightarrow\Omega_{\mathrm{spin}}^{\mathrm{top}}(S^1\times S^2)
\]
contains a subgroup isomorphic to $\Z^\infty$.
\end{restatable}

To prove these results, we use real Seiberg--Witten theory. Tian--Wang initiated a study of real Seiberg--Witten invariants for 4-manifolds with involutions~\cite{Tian-Wang:2009}. Nakamura developed the $\mathrm{Pin}^{-}(2)$-monopole theory associated with free involutions~\cite{NakamuraPinMinus}. For spin 4-manifolds with odd involutions, Kato introduced the real involutive symmetry of the Seiberg--Witten equations in \cite{Kato:2022}. Building on Kato's construction, Konno--Miyazawa--Taniguchi defined real Seiberg--Witten Floer homotopy types~\cite{KMT21,konno2024involutions}; for a real monopole Floer formulation, see also Li~\cite{Li:2022}.

We develop an asymptotic version of the real Seiberg--Witten Floer homotopy type for homology handles introduced by Miyazawa--Park--Taniguchi~\cite[Sections~4 and~5]{miyazawa2025satellite}, using higher cyclic covers. For a distinguished homology handle $(Y,\phi)$, let $Y^{2^n}\to Y^{2^{n-1}}$ be the successive double covering in the $2$-power cyclic tower determined by $\phi$, and let $\tau_n$ denote its deck involution. Recall that, in the real spin case, the quaternionic symmetry $j$ gives the real Floer homotopy type an action of $\langle j\rangle\cong\Z_4$, viewed as a subgroup of $\mathrm{Pin}(2)$. Applying real Seiberg--Witten theory at each level therefore gives a $\Z_4$-equivariant spectrum and, by restricting to $\langle -1\rangle\cong\Z_2$, a $\Z_2$-equivariant spectrum:
\[
SWF_R^n(Y,\phi),
\qquad
SWF_{R,c}^n(Y,\phi).
\]
Their local equivalence classes lie in $\mathcal{LE}_{\Z_4}$ and $\mathcal{LE}_{\Z_2}$, respectively. These are the abelian groups of local equivalence classes of the corresponding SWF-like equivariant spectra, with group operation induced by the smash product; see \Cref{sec: real SWF for handles}.

%Passing to sufficiently high covers is essential for cobordism invariance. For a general $\wt{H}_F$-cobordism $(W,\phi_W)$, the real Bauer--Furuta map at the first level need not be local, since the twisted first cohomology and the positive part of the twisted intersection form may be nonzero. We show that the finite-dimensionality of $H_\ast(\wt W;F)$ forces the deck involution of $W^{2^n}\to W^{2^{n-1}}$ to act trivially on cohomology over $F$ for all sufficiently large $n$. 
%This is an essential observation to get a real spin/spin$^c$ structure on high covers $W^{2^n}\to W^{2^{n-1}}$ and also the corresponding real Bauer--Furuta invariants to be local. 
%Consequently, if $\ell_{n-1}$ denotes the associated rank-one $\Z$-local system, then
%\[
%H^1(W^{2^{n-1}};\ell_{n-1}\otimes_{\Z}\R)=0
%\qquad\text{ and }\qquad
%H^+(W^{2^{n-1}};\ell_{n-1}\otimes_{\Z}\R)=0, 
%\]
%where $H^+(W^{2^{n-1}};\ell_{n-1}\otimes_{\Z}\R)$ denotes the self-dual part of the intersection form with the local coefficient $\ell_{n-1}\otimes_{\Z}\R$. 
%The high covers also admit real $\spinc$ structures extending the prescribed boundary structures, and real spin structures in the $\mathbb{F}_2$-setting.
%The relative real Bauer--Furuta construction therefore gives local maps in both directions, establishing local equivalence for all sufficiently large $n$.

Passing to sufficiently high covers is essential for cobordism invariance. For a smooth $\wt{H}_{\Q}$-cobordism $(W,\phi_W)$, the finite-dimensionality of $H_\ast(\wt W;\Q)$ forces the deck involution of $W^{2^n}\rightarrow W^{2^{n-1}}$ to act trivially on rational cohomology for all sufficiently large $n$. At each level, the class defining this double cover admits an integral lift and therefore has square zero modulo $2$. Together with a twisted integral lift of the second Stiefel--Whitney class of the quotient, this gives a real $\mathrm{spin}^c$ structure on the cover; see \Cref{prop: relative existence spinc-minus}. For sufficiently large $n$, the trivial action on rational cohomology forces its first Chern class to be torsion. We then modify the real $\mathrm{spin}^c$ structure to remove the odd-order part of this class. The remaining $2$-primary class restricts to zero on the boundary, since the boundary covers have no $2$-torsion in integral second cohomology. The uniqueness of the real lift then identifies the boundary restrictions with the real $\mathrm{spin}^c$ structures induced by the canonical real spin structures; see \Cref{thm: Q-Kawauchi implies real spin c}.

The trivial action on rational cohomology also implies that the anti-invariant first cohomology and the positive part of the anti-invariant intersection form vanish. On the fixed-point sets used to define local maps, the spinor terms vanish, and the relative real Bauer--Furuta map is governed by the linear operator on anti-invariant differential forms. The preceding vanishings make this fixed-point restriction a homotopy equivalence, so the map is local. Moreover, the first Chern class is torsion, and the signature of the cover is zero, so the normalized Dirac-index vanishes. In the $\mathbb{F}_2$-setting, sufficiently high covers admit compatible real spin structures, making these maps $\Z_4$-equivariant rather than merely $\Z_2$-equivariant. Applying the construction to the cobordism and its reverse therefore gives local maps in both directions, establishing local equivalence for all sufficiently large $n$; see \Cref{thm: cyclic covers give local maps,lem: local equivalence lemma}.

Since the values of $n$ for which local equivalence holds may depend on the cobordism, we disregard finitely many terms by defining the maps
\[
SWF_R^\infty\colon\Omega_{\mathrm{spin}}(S^1\times S^2)\longrightarrow\wt\prod_{n>0}\mathcal{LE}_{\Z_4}
\qquad\text{ and }\qquad
SWF_{R,c}^\infty\colon\Omega(S^1\times S^2)\longrightarrow\wt\prod_{n>0}\mathcal{LE}_{\Z_2},
\]
given by
\[
SWF_R^\infty(Y,\phi)=\bigl[(SWF_R^1(Y,\phi),SWF_R^2(Y,\phi),\ldots)\bigr]
\]
and similarly for $SWF_{R,c}^\infty$, where $\wt\prod_{n>0}A\coloneqq(\prod_{n>0}A)/(\bigoplus_{n>0}A)$ for an abelian group $A$. We prove in \Cref{lem: additivity} that these maps are group homomorphisms and use them to establish \Cref{thm:main,thm: main2}.

%Applying the zero-surgery comparison in the proof of \cite[Theorem~4.10]{miyazawa2025satellite} at each level of the branched-cover tower gives the commutative diagram
%\[
%\begin{tikzcd}[column sep=2.5em, row sep=1.5em]
%\mathcal{C} \arrow[rr,"\mathcal{S}_{\mathbb{F}_2}"] \arrow[dr,swap] & & \Omega_{\mathrm{spin}}(S^1\times S^2) \arrow[dl,"SWF_R^\infty"] \\
%& \text{ \phantom{13}} \wt\prod_{n>0}\mathcal{LE}_{\Z_4} &
%\end{tikzcd}
%\]
%Here the diagonal map from $\mathcal{C}$ sends $[K]$ to $[([SWF_R^{n}(K)]_{\operatorname{loc}})_{n>0}]$, where $SWF_R^{n}(K)$ denotes the real Floer homotopy type of the $2^n$-fold cyclic branched cover of $K$, with respect to the deck involution of the successive double covering, defined in \cite[Definition~3.1]{KPT25}. Similarly, restricting the symmetry to $\Z_2$ gives a homomorphism $\mathcal{C}\to\wt\prod_{n>0}\mathcal{LE}_{\Z_2}$ that factors through $\mathcal{S}\colon\mathcal{C}\to\Omega(S^1\times S^2)$. Thus the asymptotic higher-cover knot invariants depend only on the $\wt{H}_{\mathbb{F}_2}$-cobordism class of the zero surgery, and, after restricting the symmetry to $\Z_2$, only on its $\wt{H}_{\Q}$-cobordism class.

We conclude with some natural questions.

\begin{ques}
Does the kernel of $\iota\colon\Omega(S^1\times S^2)\rightarrow\Omega^{\mathrm{top}}(S^1\times S^2)$ contain a subgroup isomorphic to $\Z^\infty$? Does it contain an element that cannot be written as $[Y\#(S^1\times S^2)]$ for any homology $3$-sphere $Y$ (see \Cref{cor: connected sum homomorphism} for a related discussion)?
\end{ques}

\begin{ques}
Does the kernel of $\iota_F\colon\Omega_F(S^1\times S^2)\rightarrow\Omega_F^{\mathrm{top}}(S^1\times S^2)$ contain nonzero torsion elements for some field $F$? Does this hold for every field $F$?
\end{ques}

\begin{ques}
Is the zero-surgery homomorphism $\mathcal{S}_F\colon\mathcal{C}\rightarrow\Omega_F(S^1\times S^2)$ defined in \Cref{cor: zero-surgery homomorphism} surjective for every field $F$?
\end{ques}

% Applying the zero-surgery comparison in the proof of \cite[Theorem~4.10]{miyazawa2025satellite} at each level of the branched-cover tower gives the commutative diagram
% \[
% \begin{tikzcd}[column sep=2.5em, row sep=1.5em]
% \mathcal{C} \arrow[rr,"\mathcal{S}_{\mathbb{F}_2}"] \arrow[dr,swap,"SWF_R^\infty"] & & \Omega_{\mathrm{spin}}(S^1\times S^2) \arrow[dl,"SWF_R^\infty"] \\
% & \text{ \phantom{13}}\wt\prod_{n>0}\mathcal{LE}_{\Z_4} &
% \end{tikzcd}
% \]
% Here the diagonal map sends $[K]$ to $[([SWF_R^{(n)}(K)]_{\operatorname{loc}})_{n>0}]$, where $SWF_R^{(n)}(K)$ denotes the real Floer homotopy type of the $2^n$-fold cyclic branched cover of $K$, with respect to the deck involution of the successive double covering, defined in \cite[Definition~3.1]{KPT25}. Similarly, the real $\spinc$ invariant gives a homomorphism $SWF_{R,c}^\infty\colon\mathcal{C}\to\wt\prod_{n>0}\mathcal{LE}_{\Z_2}$ that factors through $\mathcal{S}\colon\mathcal{C}\to\Omega(S^1\times S^2)$.

\subsection*{AI Disclosure}

ChatGPT 6 Pro was used to spot various errors in previous drafts of this paper. During this process, ChatGPT suggested adding \Cref{prop: tensor product for connected base,cor: equivariant chain tensor product}, pointing out that they were being implicitly used in several places throughout the paper without a statement or a proof, and \Cref{cor: equivariant chain tensor product} does not hold if one works over a ``wrong'' characteristic. ChatGPT also suggested their rough proof sketches, together with various references. The authors then examined the given proof sketches and references carefully and used them to reconstruct the proofs, from which \Cref{subsec: appendix kunneth} was written.

\subsection*{Notation}

We write $f\sim g$ for homotopic morphisms and $X\simeq Y$ for homotopy equivalent objects, understood in the relevant category. Unless otherwise specified, homology spheres and homology handles are understood with integral coefficients. For a $G$-spectrum $X$ and a commutative ring $R$, we use $H_G^\ast(X;R)$ to denote reduced Borel cohomology with coefficients in $R$.

\subsection*{Acknowledgements}

The first author is grateful to Dongsoo Lee for bringing this question to his attention while visiting him in 2022. The first author is partially supported by the Royal Society University Research Fellowship URF\textbackslash R1\textbackslash 251501. The second author is partially supported by the Samsung Science and Technology Foundation (SSTF-BA2102-02) and by National Research Foundation of Korea (NRF) grants funded by the Korean government (MSIT) (RS-2025-00542968 and RS-2026-25506097). The third author was partially supported by JSPS KAKENHI Grant Number 22K13921.

\section{Homology handles, Kawauchi cobordisms, and real spin structures}\label{sec: real spin and spin c structures}

We recall Kawauchi's cobordism group and extend its definition to arbitrary coefficient fields. We then study the homology and intersection forms of finite cyclic covers of $\wt{H}_F$-cobordisms and establish the existence of real spin structures on sufficiently large $2$-power cyclic covers in the $\mathbb{F}_2$-setting. These results provide the topological input for the real Seiberg--Witten constructions below.

\subsection{The $\wt{H}_F$-cobordism group}

We start by reviewing several notations related to Kawauchi's cobordism group. 

\begin{defn}
A \emph{distinguished homology handle} is a closed oriented $3$-manifold $Y$ that is a homology $S^1\times S^2$, together with a choice of generator $\phi\in H^1(Y;\Z)\cong\Z$. Given distinguished homology handles $(Y,\phi)$ and $(Y',\phi')$, a \emph{distinguished cobordism} from $(Y,\phi)$ to $(Y',\phi')$ is a compact connected oriented
cobordism $W$, in the smooth or topological category as appropriate, from $Y$ to $Y'$, together with a class $\phi_W\in H^1(W;\Z)$ such that $\left.\phi_W\right|_Y=\phi$ and $\left.\phi_W\right|_{Y'}=\phi'$.
\end{defn}

\begin{defn}
Given distinguished homology handles $(Y,\phi)$ and $(Y',\phi')$, an orientation-preserving diffeomorphism $f\colon Y\to Y'$ is called a \emph{diffeomorphism of distinguished homology handles} if $\phi=f^\ast\phi'$. If such a diffeomorphism exists, we say that $(Y,\phi)$ and $(Y',\phi')$ are \emph{diffeomorphic}.
\end{defn}

\begin{rem}
We will usually suppress $\phi$, $\phi'$, and $\phi_W$ from the notation unless we need to refer to them explicitly.
\end{rem}
From now on, we use the following notation. Given a distinguished homology handle $(Y,\phi)$, the class $\phi$ corresponds to a homotopy class of maps $Y\to S^1$. Choosing a representative of this homotopy class, we obtain the induced infinite cyclic cover
\[
\wt{Y} \coloneqq Y\times_{S^1}\R\longrightarrow Y,
\]
which admits a natural properly discontinuous $\Z$-action induced by translation on $\R$. Furthermore, for any integer $n>0$, the quotient by the subgroup $n\Z\subset\Z$ gives the $n$-fold cyclic cover
\[
Y^n \coloneqq \wt{Y}/n\Z\longrightarrow Y.
\]
Similarly, given a distinguished cobordism $(W,\phi_W)$, we define $\wt{W}$ and $W^n$ in the same way.

\begin{defn}
Let $F$ be a field, and let $t$ denote the generating deck transformation of an infinite cyclic cover $\wt{W}\to W$. We say that a compact oriented topological manifold $W$, together with the infinite cyclic cover $\wt{W}\to W$, is \emph{$F$-Kawauchi} if
\[
\dim_F H_\ast(\wt{W};F)<\infty.
\]
We say that $W$, together with $\wt{W}\to W$, is \emph{$F$-stable} if $t^{2^k}$ acts trivially on $H_\ast(W^{2^{k+1}};F)$ for all sufficiently large integers $k>0$. Moreover, given a distinguished cobordism $(W,\phi_W)$ between two distinguished homology handles, we say that $(W,\phi_W)$ is an \emph{$\wt{H}_F$-cobordism} if $W$ is $F$-Kawauchi with respect to the infinite cyclic cover induced by $\phi_W$. Note that $\wt{H}_{\Q}$-cobordisms are precisely the $\wt{H}$-cobordisms defined in \cite{Kawauchi:1976}. We nevertheless retain the subscript $\Q$ throughout this article to distinguish these cobordisms from $\wt{H}_{\mathbb{F}_2}$-cobordisms.
\end{defn}

Let $F$ be a field. Define a relation $\sim_F$ on the set of diffeomorphism classes of distinguished homology handles by
\[
(Y,\phi)\sim_F(Y',\phi')
\quad\text{if}\quad
(Y,\phi)\text{ and }(Y',\phi')\text{ are smoothly $\wt{H}_F$-cobordant}.
\]
By the proof of \cite[Lemma~1.2]{Kawauchi:1976}, which works over any field, $\sim_F$ is an equivalence relation. We therefore define
\[
\Omega_F(S^1\times S^2)
 \coloneqq 
\{\text{diffeomorphism classes of distinguished homology handles}\}/\sim_F.
\]
To equip this set with a group operation, we use the \emph{circle sum} introduced by Kawauchi in \cite{Kawauchi:1976}. For $[(Y,\phi)],[(Y',\phi')]\in\Omega_F(S^1\times S^2)$, we denote their circle sum by $[(Y,\phi)\bigcirc(Y',\phi')]$. Briefly, choose oriented simple closed curves $\gamma\subset Y$ and $\gamma'\subset Y'$ such that $\phi([\gamma])=\phi'([\gamma'])=1$, and let $N_\gamma$ and $N_{\gamma'}$ be tubular neighborhoods of these curves. Remove the interiors of these neighborhoods and glue the resulting torus boundaries by an orientation-reversing diffeomorphism compatible with the restrictions of $\phi$ and $\phi'$. The resulting manifold
\[
(Y\smallsetminus\operatorname{int}N_\gamma)
\cup_{\partial N_\gamma\cong\partial N_{\gamma'}}
(Y'\smallsetminus\operatorname{int}N_{\gamma'})
\]
together with the induced distinguished cohomology class is called the \emph{circle sum} of $(Y,\phi)$ and $(Y',\phi')$.

One caveat is that the circle sum is not well-defined up to diffeomorphism, but only up to $\wt{H}_{\mathbb{Q}}$-cobordism, as shown in \cite[Lemma~1.6]{Kawauchi:1976}. The same proof, with coefficients in $F$, applies to the $\wt{H}_F$-cobordism setting, but we will rewrite it here with a more explicit computation for the sake of self-containedness.

\begin{lem}\label{lem: inf cyclic cover of homology handle has fin dim homology}
Let $(Y,\phi)$ be a distinguished homology handle. Then $\dim_F H_\ast(\wt{Y};F)<\infty$ for any field $F$. Equivalently, $Y\times I$ is an $\wt{H}_F$-cobordism.
\end{lem}

\begin{proof}
Since $H_\ast(\wt{Y};F)$ is a finitely generated $F[t,t^{-1}]$-module and $F[t,t^{-1}]$ is a PID, it decomposes into cyclic summands. Following the arguments in \cite[Proof of Assertion 5]{Milnor1968InfiniteCyclicCoverings}, we start by comparing the kernel and cokernel of multiplication by $t-1$ on each cyclic summand. Then, in each degree $k$, we get
\[
\dim_F\ker(H_k(\wt{Y};F)\xrightarrow{t-1}H_k(\wt{Y};F))
\leq
\dim_F\operatorname{coker}(H_k(\wt{Y};F)\xrightarrow{t-1}H_k(\wt{Y};F)),
\]
with equality if and only if $H_k(\wt{Y};F)$ is finite-dimensional over $F$; their difference is exactly the free rank of $H_k(\wt{Y};F)$.

Consider the long exact sequence
\[
\cdots\longrightarrow H_k(\wt{Y};F)
\xrightarrow{t-1}H_k(\wt{Y};F)
\longrightarrow H_k(Y;F)\longrightarrow\cdots
\]
induced by the short exact sequence
\[
0\longrightarrow F[t,t^{-1}]
\xrightarrow{t-1}F[t,t^{-1}]
\longrightarrow F\longrightarrow0
\]
of local systems on $Y$. Since $\wt{Y}$ is a connected noncompact $3$-manifold, we have $H_k(\wt{Y};F)=0$ for $k\geq3$. Thus, in degrees 2 and 3, we obtain
\[
0\longrightarrow H_3(Y;F)
\longrightarrow H_2(\wt{Y};F)
\xrightarrow{t-1}H_2(\wt{Y};F)
\longrightarrow H_2(Y;F).
\]
It follows that the kernel of $t-1$ on $H_2(\wt{Y};F)$ is one-dimensional. By the preceding inequality, its cokernel is at least one-dimensional. Since this cokernel embeds into $H_2(Y;F)\cong F$, it is exactly one-dimensional. Hence $H_2(\wt{Y};F)\to H_2(Y;F)$ is surjective, and the sequence in degrees 1 and 2 becomes
\[
0\longrightarrow H_1(\wt{Y};F)
\xrightarrow{t-1}H_1(\wt{Y};F)
\longrightarrow H_1(Y;F)
\longrightarrow H_0(\wt{Y};F)
\overset{0}{\longrightarrow}H_0(\wt{Y};F).
\]
Here we have used the fact that $t-1$ acts trivially on $H_0(\wt{Y};F)$. Since $H_1(Y;F)\cong H_0(\wt{Y};F)\cong F$, exactness implies that $H_1(Y;F)\to H_0(\wt{Y};F)$ is an isomorphism. Consequently,
\[
H_1(\wt{Y};F)\xrightarrow{t-1}H_1(\wt{Y};F)
\]
is also an isomorphism. It follows that
\[
\dim_F\ker(H_\ast(\wt{Y};F)\xrightarrow{t-1}H_\ast(\wt{Y};F))
=
\dim_F\operatorname{coker}(H_\ast(\wt{Y};F)\xrightarrow{t-1}H_\ast(\wt{Y};F))
=2.
\]
Thus $\dim_F H_\ast(\wt{Y};F)<\infty$, as desired.
\end{proof}

\begin{lem}\label{lem: circle sum is well def over any field}
Let $(Y,\phi)$ and $(Y',\phi')$ be distinguished homology handles, and let $Y\bigcirc Y'$ and $Y\bigcirc' Y'$ be any two circle sums of them. Then there exists a smooth distinguished cobordism $W$ between these circle sums that is an $\wt{H}_F$-cobordism for every field $F$.
\end{lem}

\begin{proof}
Following the construction in \cite[Proof of Proposition~2.6(1)]{lee2023kernel}, we obtain $W$ by attaching one round $1$-handle at each end of $(Y\sqcup Y')\times I$, using the choices defining the respective circle sums. Each round $1$-handle decomposes into a $1$-handle and a $2$-handle, and its $S^1$ direction lifts to $\R$ in the infinite cyclic cover $\wt{W}$. Thus
\[
C_\ast(\wt{W},(\wt{Y}\sqcup\wt{Y}')\times I;\Z)
\simeq
\left[0\longrightarrow\Z[t,t^{-1}]^2
\xrightarrow{t-1}\Z[t,t^{-1}]^2\longrightarrow0\right],
\]
where the two nonzero terms lie in degrees 2 and 1, respectively. Tensoring with any field $F$ gives
\[
C_\ast(\wt{W},(\wt{Y}\sqcup\wt{Y}')\times I;F)
\simeq
\left[0\longrightarrow F[t,t^{-1}]^2
\xrightarrow{t-1}F[t,t^{-1}]^2\longrightarrow0\right].
\]
Consequently,
\[
\dim_F H_\ast(\wt{W},(\wt{Y}\sqcup\wt{Y}')\times I;F)=2.
\]
By \Cref{lem: inf cyclic cover of homology handle has fin dim homology}, we also have
\[
\dim_F H_\ast((\wt{Y}\sqcup\wt{Y}')\times I;F)<\infty.
\]
The long exact sequence of the pair therefore gives $\dim_F H_\ast(\wt{W};F)<\infty$. The lemma follows.
\end{proof}

Hence the circle sum is well-defined on $\Omega_F(S^1\times S^2)$, and it is clearly associative. Similarly, the arguments in \cite[Section~1]{Kawauchi:1976} extend to $\wt{H}_F$-cobordisms and show that the circle sum satisfies the abelian group axioms on $\Omega_F(S^1\times S^2)$ for every field $F$.

\begin{defn}
We call the abelian group $(\Omega_F(S^1\times S^2),\bigcirc)$ the \emph{smooth $\wt{H}_F$-cobordism group}. Replacing smooth $\wt{H}_F$-cobordisms with topological ones defines an equivalence relation $\sim_F^{\mathrm{top}}$ and an abelian group $(\Omega_F^{\mathrm{top}}(S^1\times S^2),\bigcirc)$, which we call the \emph{topological $\wt{H}_F$-cobordism group}.
\end{defn}

\begin{rem}
    It is straightforward, via a standard Mayer--Vietoris argument, to show that a distinguished homology hande $(Y,\phi)$ represents the identity $0\in \Omega_F(S^1 \times S^2)$ if and only if there exists a smooth $F$-Kawauchi cobordism between $(Y,\phi)$ and $\emptyset$. Similarly, $(Y,\phi)$ represents $0\in \Omega_F^{\mathrm{top}}(S^1 \times S^2)$ if and only if there exists a topological $F$-Kawauchi cobordism between $(Y,\phi)$ and $\emptyset$.
\end{rem}

Although some intermediate results will be stated for arbitrary fields, our main applications concern the cases $F=\Q$ and $F=\mathbb{F}_2$.

\begin{rem}
The groups $\Omega_{\Q}(S^1\times S^2)$ and $\Omega_{\Q}^{\mathrm{top}}(S^1\times S^2)$ coincide with the groups studied in \cite{Kawauchi:1976,lee2023kernel}. We therefore omit the subscript $\Q$ and write them as $\Omega(S^1\times S^2)$ and $\Omega^{\mathrm{top}}(S^1\times S^2)$, respectively. We denote $\Omega_{\mathbb{F}_2}(S^1\times S^2)$ and $\Omega_{\mathbb{F}_2}^{\mathrm{top}}(S^1\times S^2)$ by $\Omega_{\mathrm{spin}}(S^1\times S^2)$ and $\Omega_{\mathrm{spin}}^{\mathrm{top}}(S^1\times S^2)$, respectively. The notation ``spin'' is motivated by \Cref{subsec:F2-real-spin}, where we show that sufficiently large $2$-power cyclic covers of smooth $\wt{H}_{\mathbb{F}_2}$-cobordisms admit real spin structures. We emphasize that no choice of spin structure is included in the definition of these groups.
\end{rem}

We next establish some basic relationships between $\Omega_F(S^1\times S^2)$ and the usual homology cobordism groups.

\begin{lem}\label{lem: homology cobordism implies Kawauchi}
Let $(W,\phi_W)$ be a distinguished cobordism between distinguished homology handles and $F$ be a field. If $W$ is an $F$-homology cobordism, then $(W,\phi_W)$ is an $\wt{H}_F$-cobordism.
\end{lem}

\begin{proof}
Denote the incoming and outgoing boundaries of $W$ by $Y$ and $Y'$, respectively, and set $R=F[t,t^{-1}]$. We regard $F$ as an $R$-module via the homomorphism $R\to F$ given by $t\mapsto 1$. Since $W$ is compact, each $H_j(\wt{W},\wt{Y};F)$ is a finitely generated $R$-module. Because $R$ is a PID, we may write
\[
H_j(\wt{W},\wt{Y};F)\cong R^{m_j}\oplus\bigoplus_{i=1}^{k_j}R/(p_{j,i}(t))
\]
for nonzero polynomials $p_{j,i}\in R$. The change-of-rings universal coefficient theorem gives a short exact sequence
\[
0\longrightarrow H_j(\wt{W},\wt{Y};F)\otimes_R F\longrightarrow H_j(W,Y;F)\longrightarrow \operatorname{Tor}_1^R(H_{j-1}(\wt{W},\wt{Y};F),F)\longrightarrow 0.
\]
Since $W$ is an $F$-homology cobordism, $H_j(W,Y;F)=0$. Therefore $H_j(\wt{W},\wt{Y};F)\otimes_R F=0$, which implies that $m_j=0$. Thus $H_j(\wt{W},\wt{Y};F)$ is a finitely generated torsion $R$-module and hence is finite-dimensional over $F$. By the proof of \cite[Lemma~1.2]{Kawauchi:1976}, which works over any coefficient field, $H_\ast(\wt{Y};F)$ is also finite-dimensional over $F$. The long exact sequence
\[
\cdots\longrightarrow H_j(\wt{Y};F)\longrightarrow H_j(\wt{W};F)\longrightarrow H_j(\wt{W},\wt{Y};F)\longrightarrow\cdots,
\]
together with \Cref{lem: inf cyclic cover of homology handle has fin dim homology}, then shows that $H_\ast(\wt{W};F)$ is finite-dimensional over $F$. Hence $W$ is $F$-Kawauchi, and therefore $(W,\phi_W)$ is an $\wt{H}_F$-cobordism.
\end{proof}

By \Cref{lem: homology cobordism implies Kawauchi}, we obtain two natural group homomorphisms.\begin{cor}\label{cor: zero-surgery homomorphism}
For any field $F$, the zero-surgery operation defines a group homomorphism
\[
\mathcal{S}_F\colon\mathcal{C}\longrightarrow\Omega_F(S^1\times S^2);\qquad [K]\longmapsto[(S^3_0(K),\phi_K)],
\]
where $\phi_K\in H^1(S^3_0(K);\Z)\cong\Z$ is the generator that evaluates to $1$ on the positively oriented meridian of $K$.
\end{cor}

\begin{proof}
A smooth concordance between knots $K$ and $K'$ induces a distinguished integral homology cobordism between $(S^3_0(K),\phi_K)$ and $(S^3_0(K'),\phi_{K'})$. By \Cref{lem: homology cobordism implies Kawauchi}, this is an $\wt{H}_F$-cobordism. Hence the assignment is well-defined.

Moreover, as explained in \cite[Proof of Lemma~2.4]{Kawauchi:1976}, suitable choices in the circle-sum construction give a diffeomorphism of distinguished homology handles
\[
(S^3_0(K\#K'),\phi_{K\#K'})\cong(S^3_0(K),\phi_K)\bigcirc(S^3_0(K'),\phi_{K'}).
\]
Since the circle sum is well-defined up to $\wt{H}_F$-cobordism by \Cref{lem: circle sum is well def over any field}, it follows that $\mathcal{S}_F$ respects addition and is therefore a group homomorphism.
\end{proof}

\begin{cor}\label{cor: connected sum homomorphism}
Let $\Theta^3_{\Z}$ denote the smooth homology cobordism group of integral homology $3$-spheres. For any field $F$, connected sum with $S^1\times S^2$ defines a group homomorphism
\[
\Sigma_F\colon\Theta^3_{\Z}\longrightarrow\Omega_F(S^1\times S^2);\qquad [Y]\longmapsto[(Y\#(S^1\times S^2),\phi)],
\]
where $\phi\in H^1(Y\#(S^1\times S^2);\Z)\cong\Z$ is the generator induced by the standard generator of $H^1(S^1\times S^2;\Z)$. Furthermore, the image of $\Sigma_F$ is contained in the kernel of the natural map $\Omega_F(S^1\times S^2)\to\Omega_F^{\mathrm{top}}(S^1\times S^2)$.
\end{cor}

\begin{proof}
A smooth integral homology cobordism from $Y$ to $Y'$ induces a distinguished integral homology cobordism from $Y\#(S^1\times S^2)$ to $Y'\#(S^1\times S^2)$ by taking its connected sum along a properly embedded arc with the product cobordism $(S^1\times S^2)\times[0,1]$. By \Cref{lem: homology cobordism implies Kawauchi}, the resulting cobordism is an $\wt{H}_F$-cobordism. Moreover, using the standard circles in the $S^1\times S^2$ summands to form the circle sum gives
\[
(Y\#(S^1\times S^2),\phi)\bigcirc(Y'\#(S^1\times S^2),\phi)\cong((Y\# Y')\#(S^1\times S^2),\phi).
\]
Thus $\Sigma_F$ is a well-defined group homomorphism.

By \cite[Theorem~1.4$^\prime$]{Freedman:1982}, every integral homology $3$-sphere $Y$ bounds a contractible topological $4$-manifold. Applying the construction above and \Cref{lem: homology cobordism implies Kawauchi} shows that $(Y\#(S^1\times S^2),\phi)$ is topologically $\wt{H}_F$-cobordant to $(S^1\times S^2,\phi)$. Therefore $\Sigma_F([Y])$ maps to zero in $\Omega_F^{\mathrm{top}}(S^1\times S^2)$.
\end{proof}

\begin{rem}
    Given the definition of $\Sigma_F:
    \Theta^3_\Z\rightarrow \Omega_F(S^1 \times S^2)$ for any field $F$, it is clearly not surjective (as its images map to zero in the algebraic concordance group), so one may ask whether they are injective. Suppose that $\operatorname{char}(F)\neq 2$ and consider the homology sphere
    \[
    Y = \Sigma(2,3,7) \simeq S^3_{+1}(4_1),
    \]
    where $4_1$ denotes the figure-eight knot. Since $4_1$ (or, in general, any strongly negative amphichiral knot) bounds an $\mathbb{F}_2$-homology ball~\cite[Theorem 1]{Levine23}, we see that $Y$ also bounds a $\mathbb{F}_2$-homology ball, which we denote by $W$. Then the boundary connected sum $W\natural (S^1 \times S^2 \times I)$ is an $\wt{H}_F$-cobordism from $Y\sharp S^1 \times S^2$ to $S^1 \times S^2$, i.e.
    \[
    \Sigma(Y) = 0 \in \Omega_F(S^1 \times S^2).
    \]
    On the other hand, since the Rochlin invariant of $Y$ is nonzero, we have $[Y]\neq 0$ in $\Theta^3_\Z$. Hence we see that $\Sigma_F$ is not injective whenever $\operatorname{char}(F)\neq 2$. We do not know whether $\Sigma_{\mathbb{F}_2}$ is injective.
\end{rem}

\subsection{Algebraic properties of $\wt{H}_F$-cobordisms}
We now establish some basic properties of $\wt{H}_F$-cobordisms.

\begin{lem}\label{lem: kawauchi depends on char}
Let $F_1$ and $F_2$ be fields of the same characteristic. Then a distinguished cobordism is $F_1$-Kawauchi if and only if it is $F_2$-Kawauchi. Consequently, $\wt{H}_{F_1}$-cobordisms are the same as $\wt{H}_{F_2}$-cobordisms.
\end{lem}

\begin{proof}
Let $k$ denote the common prime subfield of $F_1$ and $F_2$. For $i=1,2$, the extension $k\subset F_i$ is faithfully flat, and hence
\[
H_\ast(\wt{W};F_i)\cong H_\ast(\wt{W};k)\otimes_k F_i.
\]
It follows that
\[
\dim_{F_i}H_\ast(\wt{W};F_i)=\dim_k H_\ast(\wt{W};k).
\]
Therefore, $H_\ast(\wt{W};F_1)$ is finite-dimensional over $F_1$ if and only if $H_\ast(\wt{W};F_2)$ is finite-dimensional over $F_2$. Thus $W$ is $F_1$-Kawauchi if and only if it is $F_2$-Kawauchi.
\end{proof}

\begin{rem}\label{rem: common refinement}
It is worth noting that $\mathcal{S}_F$ and $\Sigma_F$ admit a natural
common refinement. Let $\mathcal{C}_{\Z}$ denote the smooth
$\Z$-concordance group of knots in integral homology
$3$-spheres.\footnote{Our notation differs from the convention used in
\cite{Hom-Levine-Lidman:2022}, \cite{Zhou:2021}, and
\cite{Dai-Hom-Stoffregen-Truong:2024}. In those papers,
$\mathcal{C}_{\Z}$ denotes the group of knots in $S^3$ modulo homology
concordance, while $\widehat{\mathcal{C}}_{\Z}$ denotes the group of
knots in integral homology $3$-spheres that bound integral homology
$4$-balls, modulo homology concordance. Here, by contrast, arbitrary
integral homology $3$-spheres are allowed.}
There are natural homomorphisms
\[
\begin{aligned}
\mathcal{C}&\longrightarrow\mathcal{C}_{\Z};
\qquad [K]\longmapsto[(S^3,K)],\\
\Theta^3_{\Z}&\longrightarrow\mathcal{C}_{\Z};
\qquad [Y]\longmapsto[(Y,U_Y)].
\end{aligned}
\]
where $U_Y\subset Y$ denotes an unknot. Following the proofs of
\Cref{cor: zero-surgery homomorphism,cor: connected sum homomorphism},
we obtain a homomorphism
\[
\mathcal{S}\Sigma_F\colon
\mathcal{C}_{\Z}
\longrightarrow
\Omega_F(S^1\times S^2)
\]
for every field $F$, defined again via zero-surgery. This homomorphism is surjective because every distinguished homology handle can be obtained by zero-surgery on a knot in an integral homology $3$-sphere. This map fits into the following commutative
diagram:
\[
\xymatrix{
\mathcal{C} \ar[rd] \ar[rrrd]^{\mathcal{S}_F} \\
& \mathcal{C}_{\Z} \ar[rr]^-{\mathcal{S}\Sigma_F}
&& \Omega_F(S^1\times S^2) \\
\Theta^3_{\Z} \ar[ru] \ar[rrru]^{\Sigma_F}
}
\]
\end{rem}

\begin{lem}\label{lem: F-kawauchi implies Q-kawauchi}
Let $F$ be a field. If $W$, together with its infinite cyclic cover, is $F$-Kawauchi, then it is $\Q$-Kawauchi. Consequently, every $\wt{H}_F$-cobordism is an $\wt{H}_{\Q}$-cobordism.
\end{lem}

\begin{proof}
If $F$ has characteristic zero, the conclusion follows immediately from \Cref{lem: kawauchi depends on char}. Suppose that $F$ has characteristic $p>0$. By the same lemma, it suffices to consider $F=\mathbb{F}_p$.

Set $R=\Z[t,t^{-1}]$ and $M=H_\ast(\wt{W};\Z)$. Since $W$ is compact, $M$ is a finitely generated $R$-module. The universal coefficient theorem gives an injection
\[
M\otimes_{\Z}\mathbb{F}_p\longrightarrow H_\ast(\wt{W};\mathbb{F}_p)
\]
and an isomorphism
\[
M\otimes_{\Z}\Q\cong H_\ast(\wt{W};\Q).
\]
Thus it suffices to show that if $M\otimes_{\Z}\mathbb{F}_p$ is finite-dimensional over $\mathbb{F}_p$, then $M\otimes_{\Z}\Q$ is finite-dimensional over $\Q$.

Suppose otherwise. Since $M\otimes_{\Z}\Q$ is a finitely generated module over the PID $\Q[t,t^{-1}]$, it contains a free submodule isomorphic to $\Q[t,t^{-1}]$. After multiplying a generator of this submodule by a nonzero integer, we obtain an element of $M$ that generates a submodule isomorphic to $R$. Localizing at the prime ideal $(p)\subset R$, we find that $M_{(p)}$ contains a submodule isomorphic to $R_{(p)}$, as localizations are flat.

On the other hand,
\[
R_{(p)}/pR_{(p)}\cong\mathbb{F}_p(t),
\]
and hence
\[
M_{(p)}/pM_{(p)}\cong M\otimes_R\mathbb{F}_p(t)\cong\bigl(M\otimes_{\Z}\mathbb{F}_p\bigr)\otimes_{\mathbb{F}_p[t,t^{-1}]}\mathbb{F}_p(t)=0.
\]
Here the final equality follows because $M\otimes_{\Z}\mathbb{F}_p$ is finite-dimensional over $\mathbb{F}_p$ and therefore is a torsion $\mathbb{F}_p[t,t^{-1}]$-module. Thus $M_{(p)}=pM_{(p)}$. Since $R_{(p)}$ is a local ring with maximal ideal $pR_{(p)}$ and $M_{(p)}$ is finitely generated over $R_{(p)}$, Nakayama's lemma implies that $M_{(p)}=0$. This contradicts the fact that $M_{(p)}$ contains a submodule isomorphic to $R_{(p)}$. Therefore $M\otimes_{\Z}\Q$ is finite-dimensional over $\Q$, as required.
\end{proof}

\begin{lem}\label{lem: kawauchi lemma for modules}
Let $F$ be a field, let $M$ be a finitely generated torsion
$F[t,t^{-1}]$-module, and let $p$ be a prime. Then, for every
sufficiently large integer $n$, the element $t^{p^{n-1}}$ acts trivially
on
\[
M/(t^{p^n}-1)M.
\]
\end{lem}

\begin{proof}
Set $M_n \coloneqq M/(t^{p^n}-1)M$. First, suppose that
$\operatorname{char}(F)\neq p$. Since $M$ is finitely generated and
torsion, there exists a nonzero element $h\in F[t,t^{-1}]$ that
annihilates $M$. Both $h$ and $t^{p^n}-1$ annihilate $M_n$, so
$d_n \coloneqq \gcd(h,t^{p^n}-1)$ also annihilates $M_n$. Moreover,
\[
\gcd\left(
t^{p^n}-1,
\frac{d}{dt}(t^{p^n}-1)
\right)
=
\gcd(t^{p^n}-1,p^nt^{p^n-1})
=
1.
\]
Thus $t^{p^n}-1$ is separable over $F$, and its roots in
$\overline{F}$ are precisely the $p^n$th roots of unity. Since $h$ has
only finitely many roots, for all sufficiently large $n$, every root
of $h$ whose order is a power of $p$ is a $p^{n-1}$th root of unity.
Since $d_n$ divides the separable polynomial $t^{p^n}-1$, it follows
that $d_n\mid t^{p^{n-1}}-1$. Therefore, $t^{p^{n-1}}-1$ annihilates
$M_n$, and hence $t^{p^{n-1}}$ acts trivially on $M_n$.

Now suppose that $\operatorname{char}(F)=p$. Since $F[t,t^{-1}]$ is a
PID, $M$ admits a primary decomposition
\[
M\cong\bigoplus_f M_f,
\]
where the sum ranges over the finitely many irreducible elements $f$,
up to associates, for which the $f$-primary submodule $M_f$ is
nonzero. If $f$ is not associated to $t-1$, then $f$ and $t-1$ are
coprime, so $t-1$ acts invertibly on $M_f$. Since
$$t^{p^n}-1=(t-1)^{p^n},$$ the element $t^{p^n}-1$ also acts invertibly
on $M_f$, and hence $M_f/(t^{p^n}-1)M_f=0$. It therefore remains to consider the $(t-1)$-primary submodule. Since
this submodule is finitely generated, there exists an integer $N>0$
such that $(t-1)^N$ annihilates it. For every sufficiently large $n$,
we have $p^{n-1}\geq N$, and hence
$$t^{p^{n-1}}-1=(t-1)^{p^{n-1}}$$ annihilates the $(t-1)$-primary
submodule. Thus $t^{p^{n-1}}$ acts trivially on $M_n$, completing the
proof.
\end{proof}

The proof of \Cref{lem: kawauchi lemma for modules} also yields the
following corollary.

\begin{cor}\label{cor: division bound}
Let $F$ be a field, let $p$ be a prime, and let
$0\neq f\in F[t,t^{-1}]$. Then there exists a positive integer $n$
such that $\gcd(f,t^{p^m}-1)$ divides $t^{p^n}-1$ for every $m\geq n$.
\end{cor}

\begin{proof}
Multiplying $f$ by a power of $t$ does not change the statement, so we
may assume that $f\in F[t]$. Suppose first that $\operatorname{char}(F)=p$. Write
$f=(t-1)^r g$, where $t-1$ does not divide $g$, and choose $n>0$ such
that $p^n\geq r$. Since $t^{p^m}-1=(t-1)^{p^m}$, for every $m\geq n$,
the polynomial $\gcd(f,t^{p^m}-1)$ is associated to $(t-1)^r$, which
divides $(t-1)^{p^n}=t^{p^n}-1$.

Now suppose that $\operatorname{char}(F)\neq p$. Let $\Lambda$ be the
finite set of roots of $f$ in $\overline{F}$ whose orders are powers of
$p$. Choose $n>0$ such that every element of $\Lambda$ is a
$p^n$th root of unity. For $m\geq n$, the polynomial
$t^{p^m}-1$ is separable. Hence $\gcd(f,t^{p^m}-1)$ is separable, and
each of its roots belongs to $\Lambda$. It follows that every root of
$\gcd(f,t^{p^m}-1)$ is a root of $t^{p^n}-1$, and therefore
$\gcd(f,t^{p^m}-1)$ divides $t^{p^n}-1$.
\end{proof}

\begin{lem}\label{lem: F-Kawauchi lemma}
Let $F$ be a field, let $p$ be a prime, and let $W$ be an
$F$-Kawauchi cobordism. Then, for every sufficiently large integer
$n$, the deck transformation group of the covering
$W^{p^n}\to W^{p^{n-1}}$ acts trivially on
$H_\ast(W^{p^n};F)$.
\end{lem}

\begin{proof}
Set $R \coloneqq F[t,t^{-1}]$ and $g_n \coloneqq t^{p^n}-1$. The cellular chain
complex of $W^{p^n}$ is naturally identified with
$C_\ast(\wt{W};F)\otimes_R R/(g_n)$. Since $R$ is a PID, the universal
coefficient theorem gives a noncanonical isomorphism of $R$-modules
\[
H_j(W^{p^n};F)
\cong
H_j(\wt{W};F)/g_nH_j(\wt{W};F)
\oplus
\operatorname{Tor}_1^R
\left(H_{j-1}(\wt{W};F),R/(g_n)\right).
\]
Since $H_\ast(\wt{W};F)$ is a finitely generated torsion $R$-module,
\Cref{lem: kawauchi lemma for modules} implies that, for every
sufficiently large $n$, the element $t^{p^{n-1}}-1$ annihilates the
first summand in every degree.

We now consider the $\operatorname{Tor}$ term. For each $j$, write
$H_{j-1}(\wt{W};F)\cong\bigoplus_{i=1}^{k_j}R/(f_{j,i})$, where each
$f_{j,i}$ is nonzero. It follows that
\[
\operatorname{Tor}_1^R
\left(H_{j-1}(\wt{W};F),R/(g_n)\right)
\cong
\bigoplus_{i=1}^{k_j}
R/\left(\gcd(f_{j,i},g_n)\right).
\]
For each $f_{j,i}$, choose an integer $N_{j,i}>0$ as in \Cref{cor: division bound}. Since there are only finitely many such polynomials, we may choose $n$ sufficiently large that $n-1\geq N_{j,i}$ for every $j$ and $i$. Then $\gcd(f_{j,i},g_n)$ divides $g_{N_{j,i}}$, which in turn divides $g_{n-1}=t^{p^{n-1}}-1$. Hence $t^{p^{n-1}}-1$ annihilates every summand of the $\operatorname{Tor}$ term.

It follows that $t^{p^{n-1}}-1$ annihilates $H_\ast(W^{p^n};F)$, and therefore $t^{p^{n-1}}$ acts trivially on $H_\ast(W^{p^n};F)$. Finally, $t^{p^{n-1}}$ generates the deck transformation group of the covering $W^{p^n}\to W^{p^{n-1}}$, completing the proof.
\end{proof}

For a manifold $W$ equipped with an infinite cyclic cover and an integer $n>0$, let
\[
p_{n+1}\colon W^{2^{n+1}}\longrightarrow W^{2^n}
\]
denote the natural double covering, and let $\tau_{n+1}$ denote its nontrivial deck transformation. We denote by $\ell_n$ the rank-one $\Z$-local system on $W^{2^n}$ associated with $p_{n+1}$ via the sign representation of $\Z_2$. Thus $\tau_{n+1}$ is induced by the action of $t^{2^n}$ on $W^{2^{n+1}}$. For a field $F$, we write $\ell_n\otimes_{\Z}F$ for the induced rank-one $F$-local system.

\begin{cor}\label{cor: f-homology action}
Let $F$ be a field, and let $W$ be an $F$-Kawauchi cobordism. Then $W$ is $F$-stable. Moreover, if $\operatorname{char}(F)\neq 2$, then the following statements hold for all sufficiently large integers $n>0$:
\begin{itemize}
\item $H^\ast(W^{2^n};\ell_n\otimes_{\Z}F)=0$;
\item the pullback and transfer maps
\[
p_{n+1}^\ast\colon H^\ast(W^{2^n};F)\longrightarrow H^\ast(W^{2^{n+1}};F),
\qquad
(p_{n+1})_!\colon H^\ast(W^{2^{n+1}};F)\longrightarrow H^\ast(W^{2^n};F)
\]
are isomorphisms.
\end{itemize}
\end{cor}

% \begin{cor}\label{cor: f-homology action}
% Let $F$ be a field, and let $W$ be an $F$-Kawauchi cobordism. Then $W$
% is $F$-stable. Moreover, if $\operatorname{char}(F)\neq 2$, then the
% following statements hold for all sufficiently large integers $n>0$:
% \begin{itemize}
% \item
% $H^\ast(W^{2^n};\ell_n\otimes F)=0$, where $\ell_n$ denotes the local
% system on $W^{2^n}$ associated with the double covering
% $p_{n+1}\colon W^{2^{n+1}}\to W^{2^n}$;
% \item
% the pullback and transfer maps
% \[
% p_{n+1}^\ast\colon H^\ast(W^{2^n};F)
% \longrightarrow H^\ast(W^{2^{n+1}};F),
% \qquad
% (p_{n+1})_!\colon H^\ast(W^{2^{n+1}};F)
% \longrightarrow H^\ast(W^{2^n};F)
% \]
% are isomorphisms.
% \end{itemize}
% \end{cor}

\begin{proof}
By \Cref{lem: F-Kawauchi lemma} with $p=2$, there exists an integer $N>0$ such that, for every $n\geq N$, the deck transformation $\tau_{n+1}$ of $p_{n+1}$ acts trivially on $H_\ast(W^{2^{n+1}};F)$. This is precisely the statement that $W$ is $F$-stable. Fix an integer $n\geq N$. Since $F$ is a field, the universal coefficient theorem implies that $\tau_{n+1}^\ast$ also acts trivially on $H^\ast(W^{2^{n+1}};F)$.

Suppose that $\operatorname{char}(F)\neq 2$. The standard eigenspace decomposition associated with the double covering $p_{n+1}$ gives
\[
H^\ast(W^{2^{n+1}};F)
\cong
H^\ast(W^{2^n};F)\oplus H^\ast(W^{2^n};\ell_n\otimes_{\Z}F),
\]
where the two summands are respectively the $+1$- and $-1$-eigenspaces of $\tau_{n+1}^\ast$. Since $\tau_{n+1}^\ast$ acts trivially, its $-1$-eigenspace vanishes. Therefore
\[
H^\ast(W^{2^n};\ell_n\otimes_{\Z}F)=0.
\]

Under the same decomposition, the pullback map $p_{n+1}^\ast$ identifies $H^\ast(W^{2^n};F)$ with the $+1$-eigenspace. It is therefore an isomorphism. Finally, the transfer satisfies
\[
(p_{n+1})_!\circ p_{n+1}^\ast=2\,\mathrm{id}.
\]
Since $p_{n+1}^\ast$ is an isomorphism and $2$ is invertible in $F$, the transfer $(p_{n+1})_!$ is also an isomorphism.
\end{proof}

\subsection{$\wt{H}_{\mathbb{F}_2}$-cobordisms and real spin structures}
\label{subsec:F2-real-spin}

We briefly recall the notions of real spin and real $\spinc$ structures used throughout the paper; see, for example, \cite{KMT21,miyazawa2025satellite}. Let $X$ be an oriented smooth $3$- or $4$-manifold equipped with an orientation-preserving smooth involution $\tau$. Fix a $\tau$-invariant Riemannian metric on $X$.

\begin{defn}\label{def: real spinc}
Let $\s$ be a $\spinc$ structure on $X$, with spinor bundle $S$ and Clifford multiplication $\rho$. A \emph{real structure} on $\s$ with respect to $\tau$ is an anti-complex linear isometric bundle map $I\colon S\to S$ covering $\tau$ such that $I^2=\id$ and
\[
I(\rho(v)\Phi)=\rho(d\tau(v))I(\Phi)
\]
for every tangent vector $v$ and spinor $\Phi$. In dimension $4$, we also require $I$ to preserve the decomposition $S=S^+\oplus S^-$. A pair $(\s,I)$ is called a \emph{real $\spinc$ structure} on $(X,\tau)$.
\end{defn}

Two real $\spinc$ structures are isomorphic if there is an isomorphism of the underlying $\spinc$ structures intertwining their real structures.  We next recall the spin case. Set $d \coloneqq \dim X$. Let $\mathfrak t$ be a spin structure on $X$, and let $P_{\mathfrak t}$ be the corresponding principal $\mathrm{Spin}(d)$-bundle.

\begin{defn}\label{def: real spin}
A \emph{real spin structure} on $(X,\tau)$ is a spin structure $\mathfrak t$ together with a lift $\widetilde\tau\colon P_{\mathfrak t}\to P_{\mathfrak t}$ of $\tau$ such that $\widetilde\tau^2=-1$, where $-1$ denotes the central element in the kernel of the spin double covering. This structure is also called an \emph{odd spin structure}. 
\end{defn}

To relate the two notions, let $j$ denote the quaternionic structure on the spinor bundle associated with a spin structure. Thus $j$ is anti-complex linear, $j^2=-1$, and it commutes with the action of $\mathrm{Spin}$ and hence with $\widetilde\tau$. Therefore $I \coloneqq j\widetilde\tau$ is an anti-complex linear involution compatible with Clifford multiplication. In particular, every real spin structure canonically induces a real $\spinc$ structure $(\s,I)$, where $\s$ is the $\spinc$ structure induced by $\mathfrak t$. Replacing $\widetilde\tau$ by $-\widetilde\tau$ replaces $I$ by $-I$; this is the sign ambiguity that will appear below.

The additional quaternionic symmetry $j$ preserves the real Seiberg--Witten equations associated with $(\s,I)$. This is the source of the $\Z_4$-symmetry in the real spin case, whereas a general real $\spinc$ structure gives only the corresponding $\Z_2$-symmetry.

\begin{rem}\label{rem: real spinc and spinc-minus}
The following observation is given in \cite[Lemma 2.3]{miyazawa2025satellite}. 
Suppose that $\tau$ is free, set $\overline X=X/\tau$, and let $\ell$ be the real line bundle associated with the double covering $X\to\overline X$. 
A real $\spinc$ structure on $(X,\tau)$ is equivalent to Nakamura's notion of a $\mathrm{spin}^{c-}$ structure on a pair $(\overline X,E)$ satisfying $\det E\cong\ell$; see \cite[Section~3(iii)]{NakamuraPinMinus}. 
More precisely, let $P\to\overline X$ be the principal $\mathrm{Spin}^{c-}$ bundle representing such a structure, and let $G_0$ denote the identity component of its structure group. Then $P/G_0\to\overline X$ is the double covering determined by $\det E$. Hence $P/G_0\cong X$, and, under this identification, the principal $G_0$-bundle $P\to X$ defines a $\spinc$ structure $\s$ on $X$. Right multiplication by the standard element in the nonidentity component of the structure group covers $\tau$ and induces an anti-complex linear involution $I$ on the spinor bundle. Thus $(\s,I)$ is a real $\spinc$ structure on $(X,\tau)$. Conversely, every real $\spinc$ structure on $(X,\tau)$ descends to a $\mathrm{spin}^{c-}$ structure on $\overline X$.

Under this correspondence, the $I$-invariant part of the Seiberg--Witten equations on $X$ is naturally identified with Nakamura's $\Pin^-(2)$-monopole equations on $\overline X$; see \cite[Section~4(v)]{NakamuraPinMinus}. Thus, when we use Nakamura's existence criterion for $\mathrm{spin}^{c-}$ structures below, the resulting structure on the double cover is precisely a real $\spinc$ structure in the sense of \Cref{def: real spinc}.
\end{rem}

Recall that, by \Cref{lem: F-kawauchi implies Q-kawauchi}, every $\wt{H}_F$-cobordism is an $\wt{H}_{\Q}$-cobordism for any field $F$.

\begin{lem}\label{lem: Z2-Kawauchi implies spin covers}
Let $p$ be a prime, let $F$ be a field of characteristic $p$, and let $(W,\phi_W)$ be a $\wt{H}_F$-cobordism between distinguished homology handles. Then, for every integer $n>0$, the short exact sequence
\[
0\longrightarrow H_2(\wt{W};F)/(t^{p^n}-1)\longrightarrow H_2(W^{p^n};F)\longrightarrow \operatorname{Tor}_1^{F[t,t^{-1}]}(H_1(\wt{W};F),F[t,t^{-1}]/(t^{p^n}-1))\longrightarrow 0
\]
splits. With respect to the absolute intersection form $Q_{p,n}(x,y)=x\cap j(y)$, where $$j\colon H_2(W^{p^n};F)\longrightarrow H_2(W^{p^n},\partial W^{p^n};F)$$ is the natural map, the following hold:
\begin{itemize}
\item The first summand is isotropic for every integer $n>0$.
\item For all sufficiently large integers $n>0$, the splitting can be chosen so that the second summand is also isotropic.
\end{itemize}
Furthermore, if $p=2$ and $W$ is smooth, then $W^{2^n}$ is spin for all sufficiently large integers $n>0$.
\end{lem}

\begin{proof}
Set $R \coloneqq F[t,t^{-1}]$, where $t$ acts by the generating deck transformation of $\wt{W}\to W$, and let $C_\ast$ be the cellular chain complex of $\wt{W}$ obtained by lifting a finite cell structure on $W$. Since $W$ is $F$-Kawauchi, we have $\dim_F H_\ast(\wt{W};F)<\infty$, so $H_\ast(C_\ast)$ is a finitely generated torsion $R$-module. Since $R$ is a PID, there is a chain homotopy equivalence
\[
C_\ast\simeq\bigoplus_{i=1}^k C_i,
\qquad
C_i=[R\alpha_i\overset{p_i(t)}{\longrightarrow}R\beta_i],
\]
where each $p_i$ is nonzero and the generators lie in the appropriate consecutive degrees. We use this direct sum as a chain model from now on.

Transport the equivariant chain-level intersection pairing, obtained from Poincar\'e--Lefschetz duality and the absolute-to-relative map, to this model. Its degree-two part is a sesquilinear pairing
\[
\cap\colon C_2\times C_2\longrightarrow R,
\qquad
(ax)\cap(by)=a\overline b(x\cap y),
\]
where $\overline t=t^{-1}$. Set $f_n \coloneqq t^{p^n}-1$. Reduction modulo $f_n$ gives a chain model for $W^{p^n}$ and the splitting in the statement. The reduced pairing induces the $R/(f_n)$-valued equivariant intersection form of $W^{p^n}$. The ordinary intersection form $Q_{p,n}$ is obtained by taking the coefficient of $1$ with respect to the basis $1,t,\ldots,t^{p^n-1}$ of $R/(f_n)$.

We first show that the $H_2(\wt{W};F)/f_nH_2(\wt{W};F)$ summand is isotropic. This summand is generated by the images of the classes $[\beta_i]$ for which $\deg\beta_i=2$. Since these classes are $R$-torsion and $R$ is torsion-free, their equivariant intersection pairings vanish:
\[
[\beta_i]\cap[\beta_j]=0.
\]
Their images therefore pair trivially after reduction modulo $f_n$, and hence under $Q_{p,n}$. This proves the first assertion for every integer $n>0$.

Next, consider the summands with $\deg\alpha_i=2$ and $\deg\beta_i=1$. Let $m_i$ be the largest nonnegative integer such that $(t-1)^{m_i}$ divides $p_i(t)$. Choose $n$ sufficiently large that $m_i<p^{n-1}$ for every such $i$. Since $\operatorname{char}(F)=p$, we have $f_n=(t-1)^{p^n}$. Taking the gcd to be monic, set
\[
a_i \coloneqq \frac{f_n}{\gcd(p_i,f_n)}=(t-1)^{p^n-m_i}.
\]
The classes
\[
\alpha_i' \coloneqq [a_i\alpha_i]\in H_2(C_i/f_nC_i)
\]
generate the $\operatorname{Tor}$ summand selected by the chain decomposition.

Since $\overline{t-1}=-t^{-1}(t-1)$, the product $a_i\overline{a_j}$ is a unit multiple of $(t-1)^{2p^n-m_i-m_j}$. Moreover,
\[
2p^n-m_i-m_j>2(p^n-p^{n-1})
=\left(2-\frac{2}{p}\right)p^n\geq p^n.
\]
Thus $f_n$ divides $a_i\overline{a_j}$. By sesquilinearity,
\[
(a_i\alpha_i)\cap(a_j\alpha_j)
=a_i\overline{a_j}(\alpha_i\cap\alpha_j)
\equiv0\pmod{f_n}.
\]
Consequently, $Q_{p,n}(\alpha_i',\alpha_j')=0$ for every $i,j$, proving that the chosen $\operatorname{Tor}$ summand is isotropic.

Finally, suppose that $p=2$ and that $n$ is sufficiently large. The preceding argument gives a splitting
\[
H_2(W^{2^n};F)\cong A_n\oplus T_n
\]
into two isotropic summands. For $x=x_a+x_t$, where $x_a\in A_n$ and $x_t\in T_n$, symmetry and characteristic $2$ give
\[
Q_{2,n}(x,x)
=Q_{2,n}(x_a,x_a)+2Q_{2,n}(x_a,x_t)+Q_{2,n}(x_t,x_t)
=0.
\]
The intersection form over $F$ is obtained from the intersection form over $\mathbb{F}_2$ by extension of scalars. Hence the latter also has vanishing self-intersections. The relative Wu formula therefore gives
\[
\langle w_2(W^{2^n}),x\rangle=0
\qquad\text{for every }x\in H_2(W^{2^n};\mathbb{F}_2).
\]
By the universal coefficient theorem, this implies that $w_2(W^{2^n})=0$. Thus $W^{2^n}$ is spin, completing the proof.
\end{proof}

\begin{lem}\label{lem: metabolic if no t-1 primary boundary}
Let $F$ be a field and let $(W,\phi_W)$ be an $F$-Kawauchi cobordism. Suppose that the $(t-1)$-primary submodule of $H_1(\partial\wt{W};F)$ is zero. Consider the intersection form $Q(x,y)=x\cap j(y)$, where $$j\colon H_2(W;F)\longrightarrow H_2(W,\partial W;F)$$ is the natural map, and let $\overline{Q}$ denote the induced nonsingular form on $\overline{H}_2(W;F) \coloneqq H_2(W;F)/\ker j$. Let $A$ be the image of the map $H_2(\wt{W};F)\to H_2(W;F)$ induced by the covering projection, and write $\overline{A} \coloneqq A/(A\cap\ker j)\subset\overline{H}_2(W;F)$. Then $\overline{A}$ is a metabolizer of $\overline{Q}$.
\end{lem}

\begin{proof}
The isotropy argument in the proof of \Cref{lem: Z2-Kawauchi implies spin covers}, using reduction modulo $t-1$, applies over any field $F$ and shows that $A$ is isotropic. Thus $\overline{A}$ is isotropic, and it remains to show that $\overline{A}^{\perp}=\overline{A}$.

Set $R \coloneqq F[t,t^{-1}]$, $\pi \coloneqq (t-1)\subset R$, and $K \coloneqq F(t)$. For an $R$-module $M$, write $M[\pi] \coloneqq \{x\in M:(t-1)x=0\}$ and abbreviate $M/\pi M$ to $M/\pi$. Since $W$ is $F$-Kawauchi, $H_\ast(\wt{W};F)$ is a finitely generated torsion $R$-module. Poincar\'e--Lefschetz duality over $K$ and the long exact sequence of the pair imply that the relative and boundary homology groups are also finitely generated torsion $R$-modules.

The nonsingular Blanchfield pairings
\[
H_i(\wt{W};F)\times H_{3-i}(\wt{W},\partial\wt{W};F)\longrightarrow K/R
\]
induce nondegenerate pairings
\[
\begin{aligned}
\mathrm{Bl}_2\colon H_2(\wt{W};F)/\pi\times H_1(\wt{W},\partial\wt{W};F)[\pi]&\longrightarrow F,\\
\mathrm{Bl}_1\colon H_2(\wt{W},\partial\wt{W};F)/\pi\times H_1(\wt{W};F)[\pi]&\longrightarrow F,
\end{aligned}
\]
using the identification $(K/R)[\pi]\cong R/\pi\cong F$. Under these pairings and Poincar\'e--Lefschetz duality, the two universal coefficient sequences
\[
\begin{gathered}
0\longrightarrow H_2(\wt{W};F)/\pi\longrightarrow H_2(W;F)\longrightarrow H_1(\wt{W};F)[\pi]\longrightarrow 0,\\
0\longrightarrow H_2(\wt{W},\partial\wt{W};F)/\pi\longrightarrow H_2(W,\partial W;F)\longrightarrow H_1(\wt{W},\partial\wt{W};F)[\pi]\longrightarrow 0
\end{gathered}
\]
are dual to one another. Hence, with respect to the intersection pairing $Q_0$ between $H_2(W;F)$ and $H_2(W,\partial W;F)$, we have
\[
\operatorname{Ann}_{Q_0}(A)
=\operatorname{Im}(H_2(\wt{W},\partial\wt{W};F)/\pi\longrightarrow H_2(W,\partial W;F)).
\]

Now consider the commutative diagram
\[
\xymatrix{
0\ar[r]
& H_2(\wt{W};F)/\pi\ar[r]^{\varphi}\ar[d]^{j_2}
& H_2(W;F)\ar[r]^{\psi}\ar[d]^{j}
& H_1(\wt{W};F)[\pi]\ar[r]\ar[d]^{j_1}
& 0\\
0\ar[r]
& H_2(\wt{W},\partial\wt{W};F)/\pi\ar[r]^{\varphi'}
& H_2(W,\partial W;F)\ar[r]^{\psi'}
& H_1(\wt{W},\partial\wt{W};F)[\pi]\ar[r]
& 0,
}
\]
where the vertical maps are induced by the absolute-to-relative maps. Since the $\pi$-primary submodule of $H_1(\partial\wt{W};F)$ is zero, we have $H_1(\partial\wt{W};F)_\pi=0$. Localizing the long exact sequence of the pair therefore shows that
\[
H_2(\wt{W};F)_\pi\longrightarrow H_2(\wt{W},\partial\wt{W};F)_\pi
\]
is surjective. Tensoring with $R/\pi$ shows that $j_2$ is surjective. Consequently,
\[
\operatorname{Ann}_{Q_0}(A)
=\operatorname{Im}\varphi'
=\operatorname{Im}(\varphi'\circ j_2)
=\operatorname{Im}(j\circ\varphi)
=j(A).
\]
Since $Q(x,y)=Q_0(x,j(y))$, it follows that
\[
\begin{aligned}
A^\perp
&=\{v\in H_2(W;F):Q_0(a,j(v))=0\text{ for every }a\in A\}\\
&=j^{-1}(\operatorname{Ann}_{Q_0}(A))\\
&=j^{-1}(j(A))\\
&=A+\ker j.
\end{aligned}
\]
Passing to the quotient by $\ker j$, we obtain
\[
\overline{A}^{\perp}
=A^\perp/\ker j
=(A+\ker j)/\ker j
=\overline{A},
\]
as desired.
\end{proof}

\begin{lem}\label{lem: boundary has no t-1 primary part}
Let $(Y,\phi)$ be a distinguished homology handle and let $F$ be a field. Then the $(t-1)$-primary submodule of $H_1(\wt{Y};F)$ is zero. Furthermore, if $\operatorname{char}(F)=0$, then the $\Phi_{p^n}(t)$-primary submodule of $H_1(\wt{Y};F)$ is zero for every prime $p$ and every integer $n>0$, where $\Phi_k(t)$ denotes the $k$-th cyclotomic polynomial.
\end{lem}

\begin{proof}
Suppose first that $\operatorname{char}(F)=0$. By the change-of-coefficients argument in the proof of \Cref{lem: kawauchi depends on char}, it suffices to consider $F=\Q$. Fix a prime $p$ and an integer $n>0$. Since $H_1(\wt{Y};\Q)$ is a finitely generated $\Q[t,t^{-1}]$-module, it suffices to show that $H_1(\wt{Y};\Q)/\Phi_{p^n}(t)=0$.

Suppose otherwise. By \Cref{rem: common refinement}, $Y$ is obtained by zero-framed surgery on a knot $K$ in an integral homology $3$-sphere. The module $H_1(\wt{Y};\Q)$ is the rational Alexander module of $K$, whose order is represented by its Alexander polynomial $\Delta_K(t)\in\Z[t,t^{-1}]$. Since $\Phi_{p^n}(t)$ is irreducible over $\Q$, the nonvanishing of the quotient implies that $\Phi_{p^n}(t)$ divides $\Delta_K(t)$ over $\Q[t,t^{-1}]$. By Gauss's lemma, this divisibility also holds over $\Z[t,t^{-1}]$. Evaluating at $t=1$, we obtain
\[
\Phi_{p^n}(1)=p\quad\text{divides}\quad\Delta_K(1)=\pm1,
\]
a contradiction.

Now let $F$ be an arbitrary field. The universal coefficient theorem for the infinite cyclic cover gives
\[
H_1(Y;F)\cong H_1(\wt{Y};F)/(t-1)\oplus H_0(\wt{Y};F)[t-1].
\]
Since $\wt{Y}$ is connected, $H_0(\wt{Y};F)[t-1]$ is one-dimensional. Since $Y$ is an integral homology $S^1\times S^2$, $H_1(Y;F)$ is also one-dimensional. Thus $H_1(\wt{Y};F)/(t-1)=0$. Since $H_1(\wt{Y};F)$ is finitely generated over $F[t,t^{-1}]$, its $(t-1)$-primary submodule is zero. The lemma follows.
\end{proof}

\begin{lem}\label{lem: vanishing of intersection form}
Let $F$ be a field, let $p$ be a prime, and let $(W,\phi_W)$ be an $\wt{H}_F$-cobordism. For each integer $n>0$, let
\[
\overline{H}_2(W^{p^n};F)
 \coloneqq H_2(W^{p^n};F)/\ker(H_2(W^{p^n};F)\longrightarrow H_2(W^{p^n},\partial W^{p^n};F)),
\]
and let $\overline{Q}_{p,n}$ denote the induced nonsingular intersection form. Then the following hold:
\begin{itemize}
\item $\sigma(W^{p^n})=0$.
\item If $\operatorname{char}(F)=0$, then $\overline{Q}_{p,n}$ is hyperbolic.
\item If $\operatorname{char}(F)=p$ and either $p\neq2$ or $n$ is sufficiently large, then $\overline{Q}_{p,n}$ is hyperbolic.
\item The reduced absolute intersection form of $W$ over $F$ is metabolic.
\end{itemize}
\end{lem}

\begin{proof}
Applying \Cref{lem: metabolic if no t-1 primary boundary,lem: boundary has no t-1 primary part} to the infinite cyclic cover $\wt{W}\to W$ shows that the reduced absolute intersection form of $W$ over $F$ is metabolic. This proves the last assertion.

Fix an integer $n>0$, and let $t$ denote the generating deck transformation of $\wt{W}\to W$. The deck transformation group of $\wt{W}\to W^{p^n}$ is generated by $s \coloneqq t^{p^n}$. Since the infinite cyclic cover is unchanged, $W^{p^n}$ is again $F$-Kawauchi.

Suppose first that $\operatorname{char}(F)=0$. We have
\[
s-1=t^{p^n}-1=(t-1)\prod_{r=1}^n\Phi_{p^r}(t).
\]
By \Cref{lem: boundary has no t-1 primary part}, applied to each boundary component of $W$, the $(s-1)$-primary submodule of $H_1(\partial\wt{W};F)$ is zero. Thus \Cref{lem: metabolic if no t-1 primary boundary}, applied to $\wt{W}\to W^{p^n}$, shows that $\overline{Q}_{p,n}$ is metabolic. Since $F$ has characteristic zero, it is hyperbolic.

By \Cref{lem: F-kawauchi implies Q-kawauchi}, every $\wt{H}_F$-cobordism is an $\wt{H}_{\Q}$-cobordism. Applying the characteristic-zero case over $\Q$ therefore gives
\[
\sigma(W^{p^n})=0
\]
for every integer $n>0$, regardless of the characteristic of $F$.

Finally, suppose that $\operatorname{char}(F)=p$. Since
\[
s-1=t^{p^n}-1=(t-1)^{p^n},
\]
\Cref{lem: boundary has no t-1 primary part} again implies that the $(s-1)$-primary submodule of $H_1(\partial\wt{W};F)$ is zero. By \Cref{lem: metabolic if no t-1 primary boundary}, the form $\overline{Q}_{p,n}$ is metabolic. If $p\neq2$, it follows that $\overline{Q}_{p,n}$ is hyperbolic.

If $p=2$ and $n$ is sufficiently large, the isotropic-splitting argument in the proof of \Cref{lem: Z2-Kawauchi implies spin covers} gives a decomposition
\[
H_2(W^{2^n};F)\cong A_n\oplus T_n
\]
into two isotropic subspaces. Since $Q_{2,n}$ is symmetric and $F$ has characteristic $2$, every $x=x_a+x_t$, with $x_a\in A_n$ and $x_t\in T_n$, satisfies
\[
Q_{2,n}(x,x)
=Q_{2,n}(x_a,x_a)+2Q_{2,n}(x_a,x_t)+Q_{2,n}(x_t,x_t)
=0.
\]
Hence $\overline{Q}_{2,n}$ is alternating. Since it is nonsingular, it is hyperbolic, completing the proof.
\end{proof}

\begin{lem}\label{lem: spin double cover real spin}
Let $X$ be a compact connected oriented smooth spin $4$-manifold,
possibly with boundary, and let $\alpha\in H^1(X;\Z)$ be primitive. Let
\[
p\colon \widehat X\longrightarrow X
\]
be the double covering classified by the reduction of $\alpha$ modulo
$2$, and let $\tau$ denote its nontrivial deck transformation. Then
$\widehat X$ admits a real spin structure with respect to $\tau$.
\end{lem}

\begin{proof}
Choose a spin structure $\mathfrak t$ on $X$. Choose a smooth map
$f\colon X\to S^1$ representing $\alpha$ and a point $z\in S^1$ such
that both $f$ and $f\vert_{\partial X}$ are transverse to $z$. Then
$M \coloneqq f^{-1}(z)$ is a compact oriented smooth $3$-manifold properly
embedded in $X$, and
\[
\operatorname{PD}[M,\partial M]=\alpha.
\]
Let $X_M$ be the manifold obtained by cutting $X$ open along $M$, and
denote the resulting two copies of $M$ in $\partial X_M$ by $M_+$ and
$M_-$. Since the reduction of $\alpha$ modulo $2$ classifies $p$, the
manifold $\widehat X$ is obtained by gluing two copies of $X_M$
crosswise along $M_+$ and $M_-$.

Let $P_{\mathfrak t}\to X$ be the principal $\mathrm{Spin}(4)$-bundle
representing $\mathfrak t$, and let $P_M$ be its restriction to $X_M$.
The gluing that reconstructs $P_{\mathfrak t}$ from $P_M$ determines an
isomorphism
\[
\gamma\colon
P_M\vert_{M_+}
\xrightarrow{\cong}
P_M\vert_{M_-}.
\]
Define a principal $\mathrm{Spin}(4)$-bundle over $\widehat X$ by
\[
\widehat P \coloneqq (P_M\times\{0,1\})/\mathord{\sim},
\]
where
\[
(z,0)_+\sim(\gamma(z),1)_-,
\qquad
(z,1)_+\sim(-\gamma(z),0)_-.
\]
Here $-z \coloneqq z\cdot(-1)$, where $-1$ denotes the nontrivial element of the
kernel of $\mathrm{Spin}(4)\to\mathrm{SO}(4)$. Since $-1$ acts trivially
on the oriented frame bundle, $\widehat P$ defines a spin structure on
$\widehat X$.

The deck transformation $\tau$ lifts to $\widehat P$ by
\[
\widetilde{\tau}[z,0]=[z,1],
\qquad
\widetilde{\tau}[z,1]=[-z,0].
\]
These formulas are compatible with the gluing relations, and
\[
\widetilde{\tau}^{\,2}[z,i]=[-z,i].
\]
Thus $\widetilde{\tau}^{\,2}=-1$, so the spin structure represented by
$\widehat P$ is real with respect to $\tau$.
\end{proof}

\begin{lem}\label{lem: Z2-Kawauchi implies real spin}
Let $(W,\phi_W)$ be a smooth $\wt{H}_{\mathbb{F}_2}$-cobordism between
distinguished homology handles. Then, for every sufficiently large
integer $n>0$, $W^{2^{n+1}}$ admits a real spin structure with respect
to the deck transformation $\tau_{n+1}$ of the double covering
\[
p_{n+1}\colon W^{2^{n+1}}
\longrightarrow
W^{2^n}.
\]
\end{lem}

\begin{proof}
Fix a sufficiently large integer $n>0$. By
\Cref{lem: Z2-Kawauchi implies spin covers}, $W^{2^n}$ is spin. Let
$\phi_{W,n}\in H^1(W^{2^n};\Z)$ be the primitive class determined by
the infinite cyclic covering $\wt W\to W^{2^n}$, with its sign chosen
so that the pullback of $\phi_W$ to $W^{2^n}$ is
$2^n\phi_{W,n}$. The reduction of $\phi_{W,n}$ modulo $2$ classifies
$p_{n+1}$. The result therefore follows from
\Cref{lem: spin double cover real spin}.
\end{proof}

\section{$\wt{H}_{\Q}$-cobordisms and real $\mathrm{spin}^c$ structures}

In this section, we study real $\mathrm{spin}^c$ structures on the finite cyclic covers associated with $\wt{H}_{\Q}$-cobordisms. Unlike in the $\wt{H}_{\mathbb{F}_2}$ setting, these covers need not be spin. Nevertheless, we show that sufficiently large $2$-power cyclic covers admit real $\mathrm{spin}^c$ structures whose restrictions to the boundary are the real spin structures, unique up to isomorphism and sign, on the corresponding boundary covers. This provides the geometric input for the real Seiberg--Witten theory developed in the next section.

\subsection{Real $\mathrm{spin}^c$ structures on double covers}

We begin with a general existence result for real $\mathrm{spin}^c$ structures on double covers of compact oriented smooth $4$-manifolds. Using the correspondence with $\mathrm{spin}^{c-}$ structures recalled in \Cref{rem: real spinc and spinc-minus}, the problem reduces to constructing a real vector bundle of rank $2$ with prescribed Stiefel--Whitney classes. The following two lemmas provide the obstruction-theoretic ingredients needed for this construction.

\begin{lem}\label{lem: realization lemma}
Let $X$ be a CW complex, let $\alpha\in H^1(X;\mathbb{F}_2)$ and
$\beta\in H^2(X;\mathbb{F}_2)$, and let $\ell$ denote the rank-one
$\Z$-local system determined by $\alpha$. There exists a rank-$2$ real
vector bundle $E\to X$ such that
\[
w_1(E)=\alpha,
\qquad
w_2(E)=\beta
\]
if and only if $\beta$ lifts to a class in $H^2(X;\ell)$ under
coefficient reduction $\ell\to\mathbb{F}_2$. Equivalently, $\beta$
lies in the kernel of the Bockstein homomorphism
\[
H^2(X;\mathbb{F}_2)
\longrightarrow
H^3(X;\ell)
\]
associated with the short exact sequence
\[
0\longrightarrow\ell
\overset{\cdot 2}{\longrightarrow}
\ell
\longrightarrow\mathbb{F}_2
\longrightarrow 0.
\]Moreover, for $\alpha \in H^1(X; \mathbb{F}_2)$ and every lift $\tilde{\beta} \in H^2(X; l)$ of $\beta$, $E$ can be chosen so that $w_1(E)=\alpha$ and $e_{\alpha}(E)=\tilde{\beta}$, where $e_{\alpha}(E)$ denotes the twisted Euler characteristic of $E$.
\end{lem}

\begin{proof}
Suppose first that $\beta$ admits a lift
$\widetilde{\beta}\in H^2(X;\ell)$. Consider the fibration
\[
BSO(2)\longrightarrow BO(2)
\xrightarrow{B\det}
B\Z_2.
\]
Since $[X,B\Z_2]\cong H^1(X;\mathbb{F}_2)$, the class $\alpha$ is
represented by a map $f_\alpha\colon X\to B\Z_2$. Form the pullback
diagram
\[
\xymatrix{
P_\alpha \ar[r]^{\widetilde f_\alpha} \ar[d]_{q_\alpha}
& BO(2) \ar[d]^{B\det} \\
X \ar[r]_{f_\alpha}
& B\Z_2.
}
\]
The fiber of $q_\alpha$ is $BSO(2)\simeq K(\Z,2)$. The induced
monodromy action on its second homotopy group is multiplication by
$-1$ along loops on which $\alpha$ is nonzero. Thus the corresponding
local coefficient system is precisely $\ell$.

A standard reflection of $\R^2$ determines a splitting
$\Z_2\to O(2)$ of the determinant homomorphism and hence a section of
$BO(2)\to B\Z_2$. Pulling this section back along $f_\alpha$ gives a
section $s_0$ of $q_\alpha$. Since the fiber is $K(\Z,2)$, obstruction
theory gives, after choosing $s_0$, an identification
\[
H^2(X;\ell)
\xrightarrow{\cong}
\pi_0\Gamma(q_\alpha);
\qquad
\delta\longmapsto[s_\delta],
\]
where $\Gamma(q_\alpha)$ denotes the space of sections of $q_\alpha$.

Let $E_{\alpha,\delta}\to X$ be the rank-$2$ real vector bundle
classified by
\[
\widetilde f_\alpha\circ s_\delta\colon X\to BO(2).
\]
Then $w_1(E_{\alpha,\delta})=\alpha$. Under the identification above,
$\delta$ is the twisted Euler class
\[
e_\alpha(E_{\alpha,\delta})=\delta\in H^2(X;\ell).
\]
Here the normalization is determined by $s_0$: the bundle associated
with the reflection splitting has a nowhere-zero section and hence
has vanishing twisted Euler class. Let
\[
r_2\colon H^2(X;\ell)
\longrightarrow
H^2(X;\mathbb{F}_2)
\]
denote coefficient reduction. Taking $\delta=\widetilde{\beta}$, we
obtain
\[
w_2(E_{\alpha,\widetilde{\beta}})
=
r_2(e_\alpha(E_{\alpha,\widetilde{\beta}}))
=
r_2(\widetilde{\beta})
=
\beta.
\]
Thus $E_{\alpha,\widetilde{\beta}}$ is the desired vector bundle.

Conversely, suppose that $E\to X$ is a rank-$2$ real vector bundle
satisfying $w_1(E)=\alpha$ and $w_2(E)=\beta$. Its twisted Euler class
\[
e_\alpha(E)\in H^2(X;\ell)
\]
reduces modulo $2$ to $w_2(E)=\beta$. Hence $e_\alpha(E)$ is a lift of
$\beta$, completing the proof.
\end{proof}

We record the following standard fact for completeness.

\begin{lem}\label{lem: square zero lemma}
Let $X$ be a CW complex and let $\alpha\in H^1(X;\mathbb{F}_2)$. If
$\alpha$ is the reduction modulo $2$ of a class in $H^1(X;\Z_4)$, then
\[
\alpha^2=0\in H^2(X;\mathbb{F}_2).
\]
\end{lem}

\begin{proof}
Recall that $\mathrm{Sq}^1$ is the Bockstein homomorphism associated
with the following short exact sequence of abelian groups:
\[
0\longrightarrow\Z_2
\overset{\cdot 2}{\longrightarrow}
\Z_4
\longrightarrow\Z_2
\longrightarrow 0.
\]
Since $\alpha$ lifts to a class in $H^1(X;\Z_4)$, exactness implies
that $\mathrm{Sq}^1(\alpha)=0$. Since $\alpha$ has degree $1$, we have
$\mathrm{Sq}^1(\alpha)=\alpha^2$. Therefore
\[
\alpha^2=\mathrm{Sq}^1(\alpha)=0,
\]
as desired.
\end{proof}

\begin{prop}\label{prop: relative existence spinc-minus}
Let $W$ be a compact connected oriented smooth $4$-manifold, possibly
with boundary, and let $\lambda\in H^1(W;\mathbb{F}_2)$. Let $\ell$ be
the rank-one $\Z$-local system determined by $\lambda$. Then there
exists a class $c\in H^2(W;\ell)$ whose reduction modulo $2$ is
$w_2(W)$.

If, in addition, $\lambda$ admits a lift to $H^1(W;\Z_4)$, then the
double cover $W_\lambda\to W$ classified by $\lambda$ admits a real
$\mathrm{spin}^c$ structure with respect to its deck transformation.
\end{prop}

\begin{proof}
Let $\rho$ denote reduction modulo $2$. We adapt the
universal-coefficient argument in
\cite[Proof of Lemma~5.2]{NakamuraPinMinus} to the relative setting.
Since $\ell^\vee\cong\ell$ and
$\operatorname{Hom}(\ell,\mathbb{F}_2)\cong\mathbb{F}_2$, the universal
coefficient theorem gives a commutative diagram with exact rows
\[
\begin{CD}
0 @>>> \operatorname{Ext}_{\Z}(H_1(W;\ell),\Z)
@>{\iota}>> H^2(W;\ell)
@>{h_{\Z}}>> \operatorname{Hom}_{\Z}(H_2(W;\ell),\Z) @>>> 0\\
@. @V{\rho_0}VV @V{\rho}VV @V{\rho_2}VV @.\\
0 @>>> \operatorname{Ext}_{\Z}(H_1(W;\ell),\mathbb{F}_2)
@>{\overline{\iota}}>> H^2(W;\mathbb{F}_2)
@>{h_2}>> \operatorname{Hom}_{\Z}(H_2(W;\ell),\mathbb{F}_2) @>>> 0.
\end{CD}
\]

For $\beta\in H_2(W;\ell)$, let
$b\in H^2(W,\partial W;\ell)$ be its Poincar\'e--Lefschetz dual, and let
\[
j\colon H^2(W,\partial W;\ell)\to H^2(W;\ell)
\]
be the natural map. Define
\[
S(\beta)
 \coloneqq 
\langle j(b)\cup b,[W,\partial W]\rangle
\bmod 2.
\]
This pairing has untwisted coefficients because
$\ell\otimes_{\Z}\ell\cong\Z$. The map
\[
S\colon H_2(W;\ell)\to\mathbb{F}_2
\]
is a homomorphism: the two cross terms in $S(\beta+\gamma)$ agree by
graded commutativity and therefore cancel modulo $2$.

Moreover, $S$ vanishes on the torsion subgroup of $H_2(W;\ell)$. Indeed,
if $\beta$ is torsion, then its Poincar\'e--Lefschetz dual $b$ is
torsion, and hence $j(b)\cup b$ is a torsion element of
\[
H^4(W,\partial W;\ell\otimes_{\Z}\ell)
\cong
H^4(W,\partial W;\Z)
\cong
\Z.
\]
It must therefore vanish. Since
$H_2(W;\ell)/\operatorname{Tor}$ is free abelian, $S$ admits an integral
lift
\[
\widetilde S\in
\operatorname{Hom}_{\Z}(H_2(W;\ell),\Z).
\]
By the surjectivity of $h_{\Z}$, choose
$c_0\in H^2(W;\ell)$ such that
\[
h_{\Z}(c_0)=\widetilde S.
\]

Let $\overline b$ denote the reduction of $b$ modulo $2$. Since $W$ is
oriented, its second Wu class is $w_2(W)$. The relative Wu formula
\cite[Section~7]{Kervaire_relativeWu} gives
\[
\begin{aligned}
h_2(w_2(W))(\beta)
&=
\langle w_2(W)\cup\overline b,[W,\partial W]_2\rangle\\
&=
\langle \mathrm{Sq}^2(\overline b),[W,\partial W]_2\rangle\\
&=
\langle j(\overline b)\cup\overline b,[W,\partial W]_2\rangle\\
&=
S(\beta).
\end{aligned}
\]
Consequently,
\[
h_2(w_2(W)-\rho(c_0))=0.
\]
By exactness of the lower row, there exists
\[
\delta'\in
\operatorname{Ext}_{\Z}(H_1(W;\ell),\mathbb{F}_2)
\]
such that
\[
\overline{\iota}(\delta')=w_2(W)-\rho(c_0).
\]

The map $\rho_0$ is surjective. Indeed, this follows by applying
$\operatorname{Hom}_{\Z}(H_1(W;\ell),-)$ to the coefficient sequence
$0\to\Z\to\Z\to\mathbb{F}_2\to 0$ and using
$\operatorname{Ext}_{\Z}^2(-,\Z)=0$. Choose
\[
\delta\in
\operatorname{Ext}_{\Z}(H_1(W;\ell),\Z)
\]
such that $\rho_0(\delta)=\delta'$, and set
$c \coloneqq c_0+\iota(\delta)$. Commutativity of the diagram gives
\[
\rho(c)
=
\rho(c_0)+\overline{\iota}(\delta')
=
w_2(W).
\]
This proves the first assertion.

By \Cref{lem: realization lemma}, there is a rank-$2$ real vector bundle
$E\to W$ such that
\[
w_1(E)=\lambda,
\qquad
w_2(E)=w_2(W).
\]
Suppose now that $\lambda$ admits a lift to $H^1(W;\Z_4)$. By
\Cref{lem: square zero lemma}, we have $\lambda^2=0$, and hence
\[
w_2(E)+w_1(E)^2
=
w_2(W)+\lambda^2
=
w_2(TW).
\]
Therefore, by \cite[Proposition~3.4]{NakamuraPinMinus}, the pair $(W,E)$
admits a $\mathrm{spin}^{c-}$ structure. Since
$w_1(\det E)=w_1(E)=\lambda$, its associated double cover is
$W_\lambda\to W$. The correspondence recalled in
\Cref{rem: real spinc and spinc-minus} then gives a real
$\mathrm{spin}^c$ structure on $W_\lambda$ with respect to its deck
transformation.
\end{proof}

\subsection{Asymptotic existence and uniqueness of real $\mathrm{spin}^c$ structures}

We can now begin constructing real $\mathrm{spin}^c$ structures on cyclic covers of distinguished homology handles and $\wt{H}_{\mathbb{Q}}$-cobordisms. First, we record the following standard topological facts, whose proofs we omit.

\begin{lem}\label{lem: double cover and transfer}
Let $X$ be a path-connected CW complex and let $\alpha\in H^1(X;\mathbb{F}_2)$. Denote the double cover classified by $\alpha$ by
\[
p\colon X_\alpha\longrightarrow X.
\]
Then $X_\alpha$ is path-connected if and only if $\alpha\neq 0$. Moreover, the transfer exact sequence contains
\[
\cdots\longrightarrow H^0(X;\mathbb{F}_2)
\overset{\delta_0}{\longrightarrow} H^1(X;\mathbb{F}_2)
\overset{p^\ast}{\longrightarrow} H^1(X_\alpha;\mathbb{F}_2)
\overset{\operatorname{tr}}{\longrightarrow} H^1(X;\mathbb{F}_2)
\overset{\delta_1}{\longrightarrow}\cdots,
\]
where the connecting homomorphisms are given by cup product with $\alpha$. In particular, $\delta_0(1)=\alpha$.\qed
\end{lem}

\begin{lem}\label{lem: double cover of Z2 homology handle}
Let $Y$ be a closed oriented $\mathbb{F}_2$-homology $S^1\times S^2$,
and let $\alpha\in H^1(Y;\mathbb{F}_2)\cong\mathbb{F}_2$ be the
generator. Suppose that $\alpha$ admits a primitive integral lift
$\phi\in H^1(Y;\Z)$, and let
\[
p\colon Y_\alpha\longrightarrow Y
\]
be the double cover classified by $\alpha$. Then $Y_\alpha$ is again
an $\mathbb{F}_2$-homology $S^1\times S^2$. Moreover, the infinite
cyclic cover determined by $\phi$ factors through $Y_\alpha$. Let
$\phi_\alpha\in H^1(Y_\alpha;\Z)$ denote the class determining the
resulting infinite cyclic cover $\wt Y\to Y_\alpha$. Equivalently,
$\phi_\alpha$ is characterized by
\[
p^\ast\phi=2\phi_\alpha.
\]
Then $\phi_\alpha$ is primitive, and its reduction modulo $2$ is
nonzero and hence generates $H^1(Y_\alpha;\mathbb{F}_2)$.
\end{lem}

\begin{proof}
Since $\alpha\neq 0$, \Cref{lem: double cover and transfer} implies
that $Y_\alpha$ is path-connected. Because $p$ is classified by
$\alpha$, we have $p^\ast\alpha=0$. Since $\alpha$ is the reduction of
$\phi$ modulo $2$, it follows that $p^\ast\phi$ evaluates evenly on
every loop in $Y_\alpha$. Hence there is a unique class
$\phi_\alpha\in H^1(Y_\alpha;\Z)$ such that
\[
p^\ast\phi=2\phi_\alpha.
\]
The kernel of $\phi_\alpha$ agrees with the kernel of the restriction
of $\phi$ to $\pi_1(Y_\alpha)$, so $\phi_\alpha$ determines the
infinite cyclic cover $\wt Y\to Y_\alpha$.

Since $\phi$ is primitive, there is a loop $\gamma$ in $Y$ such that
$\phi([\gamma])=1$. The loop $\gamma^2$ lifts to a loop
$\widetilde{\gamma^2}$ in $Y_\alpha$, and the identity above gives
$\phi_\alpha([\widetilde{\gamma^2}])=1$. Thus $\phi_\alpha$ is
primitive, and its reduction modulo $2$ is nonzero.

Now consider the transfer exact sequence
\[
H^0(Y;\mathbb{F}_2)
\overset{\delta_0}{\longrightarrow}
H^1(Y;\mathbb{F}_2)
\overset{p^\ast}{\longrightarrow}
H^1(Y_\alpha;\mathbb{F}_2)
\overset{\operatorname{tr}}{\longrightarrow}
H^1(Y;\mathbb{F}_2).
\]
By \Cref{lem: double cover and transfer}, we have
$\delta_0(1)=\alpha$. Since $\alpha$ generates
$H^1(Y;\mathbb{F}_2)$, the map $\delta_0$ is an isomorphism. Exactness
therefore shows that $p^\ast=0$ and that $\operatorname{tr}$ is
injective. Consequently,
$\dim_{\mathbb{F}_2}H^1(Y_\alpha;\mathbb{F}_2)\leq 1$.
Since the reduction of $\phi_\alpha$ modulo $2$ is nonzero, we obtain
\[
H^1(Y_\alpha;\mathbb{F}_2)\cong\mathbb{F}_2.
\]
Finally, since $Y_\alpha$ is a closed connected oriented
$3$-manifold, Poincar\'e duality shows that its
$\mathbb{F}_2$-homology agrees with that of $S^1\times S^2$.
Therefore $Y_\alpha$ is an $\mathbb{F}_2$-homology
$S^1\times S^2$.
\end{proof}

We now turn to the existence of real $\mathrm{spin}^c$ structures on
sufficiently large $2$-power cyclic covers of
$\wt{H}_{\Q}$-cobordisms. We first record the following torsion
property of their boundary covers.

\begin{lem}\label{lem: no 2-torsion}
Let $(Y,\phi)$ be a distinguished homology handle. Then, for every
integer $n>0$, the group $H^2(Y^{2^n};\Z)$ has no nonzero torsion element whose
order is even.
\end{lem}

\begin{proof}
Fix an integer $n>0$ and set $M \coloneqq Y^{2^n}$. By iterating
\Cref{lem: double cover of Z2 homology handle}, the manifold $M$ is an
$\mathbb{F}_2$-homology $S^1\times S^2$, and the induced infinite
cyclic cover $\wt Y\to M$ is determined by a primitive class
$\phi_n\in H^1(M;\Z)$. In particular, the reduction of $\phi_n$ modulo
$2$ is nonzero. Since $H^1(M;\mathbb{F}_2)\cong\mathbb{F}_2$, the
reduction map
\[
r_2\colon H^1(M;\Z)\longrightarrow H^1(M;\mathbb{F}_2)
\]
is therefore surjective.

Consider the portion
\[
H^1(M;\Z)
\overset{r_2}{\longrightarrow}
H^1(M;\mathbb{F}_2)
\overset{\beta}{\longrightarrow}
H^2(M;\Z)
\overset{\cdot 2}{\longrightarrow}
H^2(M;\Z)
\]
of the long exact sequence associated with
\[
0\longrightarrow\Z
\overset{\cdot 2}{\longrightarrow}
\Z
\longrightarrow\mathbb{F}_2
\longrightarrow 0.
\]
Since $r_2$ is surjective, we have $\beta=0$. Exactness then implies
that multiplication by $2$ on $H^2(M;\Z)$ is injective. Thus
$H^2(M;\Z)$ has no nonzero element of order $2$, as desired.
\end{proof}

% \begin{proof}
%     It suffices to show that $H^2(Y^{2^n};\Z)$ has no 2-torsion element. Recall from \Cref{lem: double cover of Z2 homology handle} that $Y^{2^n}$ is an $\mathbb{F}_2$-homology $S^1 \times S^2$, so $H^k(Y^{2^n};\mathbb{F}_2)\simeq \mathbb{F}_2$ for $0\le k\le 3$. Consider the integral class $\alpha\in H^1(Y^{2^n};\Z)$ corresponding to the connected infinite cyclic cover $\wt{W}\rightarrow Y^{2^n}$. Clearly the reduction mod 2 of $\alpha$ is nonzero, so the generator of $H^1(Y^{2^n};\mathbb{F}_2)$ admits an integral lift (which is $\alpha$). Hence $H^1(Y^{2^n};\Z)$ cannot have any 2-torsion. Applying Poincar\'{e} duality and universal coefficient theorem then also shows that the generator of $H^2(Y^{2^n};\mathbb{F}_2)$ also has an integral lift and thus $H^2(Y^{2^n};\Z)$ has no 2-torsion. The lemma follows.
% \end{proof}
\begin{lem}\label{lem: unique lift from spin c to real spin c}
Let $(Y,\phi)$ be a distinguished homology handle, let $n>0$ be an integer, and let $\s$ be a $\mathrm{spin}^c$ structure on $Y^{2^n}$. Denote by $\tau_n$ the nontrivial deck transformation of $Y^{2^n}\to Y^{2^{n-1}}$. Then, up to isomorphism, there exists at most one real $\mathrm{spin}^c$ structure on $Y^{2^n}$ with respect to $\tau_n$ whose underlying $\mathrm{spin}^c$ structure is $\s$.
\end{lem}

\begin{proof}
By \Cref{rem: real spinc and spinc-minus}, real $\mathrm{spin}^c$ structures correspond to $\mathrm{spin}^{c-}$ structures on the quotient with the prescribed double cover. The classification argument in \cite[Proof of Proposition~2.3(2)]{Nakamura:2015}, which applies equally in dimension three, shows that their isomorphism classes, if nonempty, form a torsor over $H^2(Y^{2^{n-1}};\ell_{n-1})$. Under this action, changing a structure by a class $a$ changes its underlying $\mathrm{spin}^c$ structure by $p^\ast a$. Consequently, the isomorphism classes with underlying structure $\s$, if nonempty, form a torsor over
\[
K_\Z \coloneqq \ker(H^2(Y^{2^{n-1}};\ell_{n-1})
\overset{p^\ast}{\longrightarrow}
H^2(Y^{2^n};\Z)).
\]
The integral transfer exact sequence gives
\[
K_\Z\cong\operatorname{coker}(H^1(Y^{2^n};\Z)
\overset{\operatorname{tr}}{\longrightarrow}
H^1(Y^{2^{n-1}};\Z)).
\]
Since $\operatorname{tr}\circ p^\ast=2\id$, we have $2K_\Z=0$.

Consider the mod-$2$ analogue
\[
K_2 \coloneqq \ker(H^2(Y^{2^{n-1}};\mathbb{F}_2)
\overset{p^\ast}{\longrightarrow}
H^2(Y^{2^n};\mathbb{F}_2)).
\]
Since $\ell_{n-1}$ reduces modulo $2$ to the constant local system $\mathbb{F}_2$, coefficient reduction induces a map $r_2\colon K_\Z\to K_2$. We claim that this map is injective. Suppose that $x\in K_\Z$ satisfies $r_2(x)=0$. By the coefficient exact sequence, there exists $y\in H^2(Y^{2^{n-1}};\ell_{n-1})$ such that $x=2y$. Thus
\[
0=p^\ast x=2p^\ast y.
\]
By \Cref{lem: no 2-torsion}, the group $H^2(Y^{2^n};\Z)$ has no nonzero element of order $2$, so $p^\ast y=0$. Hence $y\in K_\Z$, and therefore $x=2y=0$. This proves the claim. Consequently, it suffices to show that $K_2=0$.

By iterating \Cref{lem: double cover of Z2 homology handle}, both $Y^{2^{n-1}}$ and $Y^{2^n}$ are $\mathbb{F}_2$-homology $S^1\times S^2$'s. The nonzero class $\alpha\in H^1(Y^{2^{n-1}};\mathbb{F}_2)$ classifying $p$ is the reduction modulo $2$ of the primitive integral class determining the infinite cyclic cover $\wt Y\to Y^{2^{n-1}}$. In particular, $\alpha$ admits a $\Z_4$-coefficient lift, so $\alpha^2=0$ by \Cref{lem: square zero lemma}.

The mod-$2$ transfer exact sequence contains
\[
H^1(Y^{2^{n-1}};\mathbb{F}_2)
\overset{\cup\alpha}{\longrightarrow}
H^2(Y^{2^{n-1}};\mathbb{F}_2)
\overset{p^\ast}{\longrightarrow}
H^2(Y^{2^n};\mathbb{F}_2).
\]
Since $\alpha$ generates $H^1(Y^{2^{n-1}};\mathbb{F}_2)$ and $\alpha^2=0$, the first map is zero. Exactness therefore implies that $p^\ast$ is injective, so $K_2=0$. Thus $K_\Z=0$, proving the lemma.
\end{proof}

\begin{lem}\label{lem: adding difference to real spin c}
Let $W$ be a compact oriented smooth manifold equipped with a free orientation-preserving smooth involution $f$. Given a real $\mathrm{spin}^c$ structure $\s$ on $W$ with respect to $f$ and a class $\alpha\in H^2(W;\Z)$, there exists a real $\mathrm{spin}^c$ structure $\s_\alpha$ on $W$ with respect to $f$ such that
\[
c_1(\s_\alpha)=c_1(\s)+2(\alpha-f^\ast\alpha).
\]
\end{lem}

\begin{proof}
Let
\[
p\colon W\longrightarrow W_0 \coloneqq W/\langle f\rangle
\]
be the quotient map, let $\lambda\in H^1(W_0;\mathbb{F}_2)$ be the class classifying this double cover, and let $\ell$ be the corresponding rank-one integral local system. The transfer exact sequence associated with $p$ contains
\[
\cdots\longrightarrow H^2(W_0;\ell)
\overset{p^\ast}{\longrightarrow} H^2(W;\Z)
\overset{\operatorname{tr}}{\longrightarrow} H^2(W_0;\Z)
\longrightarrow\cdots.
\]
Since the transfer is invariant under the deck transformation, we have $\operatorname{tr}\circ f^\ast=\operatorname{tr}$. Hence
\[
\operatorname{tr}(\alpha-f^\ast\alpha)
=\operatorname{tr}(\alpha)-\operatorname{tr}(f^\ast\alpha)
=0.
\]
By exactness, there exists a class $\widetilde\alpha\in H^2(W_0;\ell)$ such that
\[
p^\ast\widetilde\alpha=\alpha-f^\ast\alpha.
\]

Under the correspondence in \Cref{rem: real spinc and spinc-minus}, the real $\mathrm{spin}^c$ structure $\s$ descends to a $\mathrm{spin}^{c-}$ structure on a pair $(W_0,E)$, where $E\to W_0$ is a rank-$2$ real vector bundle satisfying $w_1(E)=\lambda$. The prescribed double cover identifies the orientation local system of $E$ with $\ell$. Let $c \coloneqq e_\lambda(E)\in H^2(W_0;\ell)$ be its twisted Euler class. Its pullback is the first Chern class of the underlying $\mathrm{spin}^c$ structure on $W$, so
\[
p^\ast c=c_1(\s).
\]

Set $c' \coloneqq c+2\widetilde\alpha$. By \Cref{lem: realization lemma}, there exists a rank-$2$ real vector bundle $E'\to W_0$, with its orientation local system identified with $\ell$, such that
\[
w_1(E')=\lambda,
\qquad
e_\lambda(E')=c'.
\]
Since $c'-c=2\widetilde\alpha$, the classes $c'$ and $c$ have the same reduction modulo $2$. The twisted Euler class reduces modulo $2$ to the second Stiefel--Whitney class, so
\[
w_2(E')=w_2(E).
\]
The existence of the original $\mathrm{spin}^{c-}$ structure therefore gives
\[
w_2(TW_0)=w_2(E)+w_1(E)^2=w_2(E')+w_1(E')^2.
\]
Thus, by \cite[Proposition~3.4]{NakamuraPinMinus}, the pair $(W_0,E')$ admits a $\mathrm{spin}^{c-}$ structure. Using the prescribed identification of its associated double cover with $W$, let $\s_\alpha$ be the corresponding real $\mathrm{spin}^c$ structure on $W$. Its first Chern class satisfies
\[
\begin{aligned}
c_1(\s_\alpha)
&=p^\ast c'\\
&=p^\ast c+2p^\ast\widetilde\alpha\\
&=c_1(\s)+2(\alpha-f^\ast\alpha),
\end{aligned}
\]
as required.
\end{proof}

\begin{lem}\label{lem: homology handle real spin}
Let $(Y,\phi)$ be a distinguished homology handle and let $n>0$ be an integer. Denote by
\[
p_n\colon Y^{2^n}\longrightarrow Y^{2^{n-1}}
\]
the double covering and by $\tau_n$ its nontrivial deck transformation. Then $Y^{2^n}$ admits a real spin structure with respect to $\tau_n$, unique up to isomorphism and sign. Moreover,
\[
H^\ast(Y^{2^n};\Q)^{-\tau_n^\ast}=0.
\]
\end{lem}

\begin{proof}
Set $Y_m \coloneqq Y^{2^m}$ for $m\geq 0$. By iterating \Cref{lem: double cover of Z2 homology handle}, each $Y_m$ is an $\mathbb{F}_2$-homology $S^1\times S^2$. Moreover, the infinite cyclic cover of $Y_m$ induced from that of $Y$ determines a primitive class $\phi_m\in H^1(Y_m;\Z)$, so $b_1(Y_m)\geq 1$. Since $\dim_{\mathbb{F}_2}H^1(Y_m;\mathbb{F}_2)=1$, it follows that $b_1(Y_m)=1$. Thus $Y_m$ is also a rational homology $S^1\times S^2$.

Let $\alpha_n\in H^1(Y_{n-1};\mathbb{F}_2)$ be the class classifying the double covering $p_n$. Since $\alpha_n$ is the reduction modulo $2$ of $\phi_{n-1}$, it admits a $\Z_4$-coefficient lift. Hence $\alpha_n^2=0$ by \Cref{lem: square zero lemma}. Since $\tau_n$ is free, \cite[Lemma~2.13]{miyazawa2025satellite} implies that $Y_n$ admits a real spin structure with respect to $\tau_n$.

We next prove uniqueness. By \cite[Lemma~2.13]{miyazawa2025satellite}, any two real spin structures differ by tensoring with a real line bundle on $Y_n$ equipped with an order-two lift of $\tau_n$, together with a choice of sign. Such an equivariant line bundle descends to a real line bundle on $Y_{n-1}$. On the other hand, \Cref{lem: double cover of Z2 homology handle} implies that
\[
p_n^\ast\colon H^1(Y_{n-1};\mathbb{F}_2)
\longrightarrow H^1(Y_n;\mathbb{F}_2)
\]
is zero. Therefore every real line bundle on $Y_{n-1}$ pulls back to a trivial line bundle on $Y_n$. Thus all real spin structures on $Y_n$ have isomorphic underlying spin structures, and the remaining ambiguity is the choice of sign.

Finally, the rational pullback map
\[
p_n^\ast\colon H^\ast(Y_{n-1};\Q)\longrightarrow H^\ast(Y_n;\Q)
\]
is injective because $\operatorname{tr}\circ p_n^\ast=2\id$. Since both $Y_{n-1}$ and $Y_n$ are rational homology $S^1\times S^2$'s, the source and target have the same dimension in each degree. Hence $p_n^\ast$ is an isomorphism. Since $p_n\circ\tau_n=p_n$, every class in its image is fixed by $\tau_n^\ast$. Thus $\tau_n^\ast=\id$, and consequently
\[
H^\ast(Y_n;\Q)^{-\tau_n^\ast}=0,
\]
as required.
\end{proof}

\begin{lem}\label{thm: Q-Kawauchi implies real spin c}
Let $(W,\phi_W)$ be a smooth $\wt{H}_{\Q}$-cobordism from a distinguished homology handle $(Y,\phi)$ to another distinguished homology handle $(Y',\phi')$. Then, for every sufficiently large integer $n>0$, there exists a real $\mathrm{spin}^c$ structure on $W^{2^n}$ with respect to the nontrivial deck transformation of the double covering $W^{2^n}\to W^{2^{n-1}}$. Moreover, this structure can be chosen so that its restrictions to $Y^{2^n}$ and $(Y')^{2^n}$ are induced by the real spin structures on these boundary components, which are unique up to isomorphism and sign.
\end{lem}

\begin{proof}
Fix a sufficiently large integer $n>0$, and let $\tau_n$ denote the nontrivial deck transformation of $W^{2^n}\to W^{2^{n-1}}$. The class in $H^1(W^{2^{n-1}};\mathbb{F}_2)$ classifying this double cover is the reduction modulo $2$ of the primitive integral class determined by the infinite cyclic cover $\wt W\to W^{2^{n-1}}$. In particular, it admits a lift to $H^1(W^{2^{n-1}};\Z_4)$. By \Cref{prop: relative existence spinc-minus}, the manifold $W^{2^n}$ therefore admits a real $\mathrm{spin}^c$ structure $\s$ with respect to $\tau_n$. We have
\[
-c_1(\s)=c_1(\bar\s)=c_1(\tau_n^\ast\s)=\tau_n^\ast c_1(\s).
\]
Since $\tau_n^\ast$ acts trivially on $H^\ast(W^{2^n};\Q)$ by \Cref{cor: f-homology action}, the image of $c_1(\s)$ in rational cohomology vanishes. Thus $c_1(\s)$ is torsion.

Consider the primary decomposition
\[
H^2(W^{2^n};\Z)_{\mathrm{tor}}\cong T_2\oplus T_{\mathrm{odd}},
\]
where $T_2$ and $T_{\mathrm{odd}}$ are the $2$-primary and odd-order torsion subgroups, respectively. These groups are finite because $W^{2^n}$ is compact. The automorphism $\tau_n^\ast$ preserves this decomposition. Write
\[
c_1(\s)=\alpha_2+\alpha_{\mathrm{odd}},
\qquad
\alpha_2\in T_2,
\qquad
\alpha_{\mathrm{odd}}\in T_{\mathrm{odd}}.
\]
The anti-invariance of $c_1(\s)$ implies that $\tau_n^\ast\alpha_{\mathrm{odd}}=-\alpha_{\mathrm{odd}}$. Since multiplication by $2$ is an automorphism of $T_{\mathrm{odd}}$, we have
\[
T_{\mathrm{odd}}^{-\tau_n^\ast}
=
\operatorname{Im}(1-\tau_n^\ast\colon T_{\mathrm{odd}}\longrightarrow T_{\mathrm{odd}}).
\]
Hence there exists $\alpha_0\in T_{\mathrm{odd}}$ such that
\[
-\alpha_{\mathrm{odd}}=\alpha_0-\tau_n^\ast\alpha_0.
\]
Choose $\alpha\in T_{\mathrm{odd}}$ such that $2\alpha=\alpha_0$. Then
\[
c_1(\s)+2(\alpha-\tau_n^\ast\alpha)
=
c_1(\s)-\alpha_{\mathrm{odd}}
=
\alpha_2.
\]
By \Cref{lem: adding difference to real spin c}, we may therefore replace $\s$ with a real $\mathrm{spin}^c$ structure whose first Chern class has order a power of $2$.

By \Cref{lem: no 2-torsion}, the group $H^2(Y^{2^n};\Z)$ has no nonzero element of $2$-power order. Consequently,
\[
c_1(\s\vert_{Y^{2^n}})=0.
\]
Moreover, the absence of $2$-torsion implies that a $\mathrm{spin}^c$ structure on $Y^{2^n}$ is determined by its first Chern class. Thus the underlying $\mathrm{spin}^c$ structure of $\s\vert_{Y^{2^n}}$ is isomorphic to that induced by the real spin structure supplied by \Cref{lem: homology handle real spin}, whose first Chern class is zero. By \Cref{lem: unique lift from spin c to real spin c}, these structures are also isomorphic as real $\mathrm{spin}^c$ structures. The same argument applies to $(Y')^{2^n}$, completing the proof.
\end{proof}

\begin{rem}
An $\wt{H}_{\Q}$-cobordism need not admit a spin finite cyclic cover associated with its distinguished cohomology class. For an explicit example, consider
\[
Z \coloneqq (S^2\times S^2)/\langle\iota\rangle,
\qquad
\iota(x,y)=(-x,-y).
\]
This is a closed oriented smooth rational homology $4$-sphere. The diagonal in $S^2\times S^2$ descends to an embedded $\mathbb{RP}^2\subset Z$ with self-intersection $1$ modulo $2$, so $Z$ is nonspin.

Set
\[
W \coloneqq Z\#(S^1\times S^2\times I),
\]
where the connected sum is taken in the interior, and let $\phi_W\in H^1(W;\Z)$ be the generator induced by the $S^1$-factor. Then $(W,\phi_W)$ is a distinguished rational homology cobordism and hence an $\wt{H}_{\Q}$-cobordism by \Cref{lem: homology cobordism implies Kawauchi}. For every integer $k>0$, its $k$-fold cyclic cover associated with $\phi_W$ is
\[
W^k\cong Z^{\# k}\#(S^1\times S^2\times I).
\]
Each such cover contains a punctured copy of $Z$ and is therefore nonspin. Thus the real $\mathrm{spin}^c$ structures constructed in this subsection cannot, in general, be replaced by real spin structures.
\end{rem}

\section{Real Seiberg--Witten theory for distinguished homology handles}\label{sec: real SWF for handles}

We recall the real Seiberg--Witten Floer homotopy type and apply it to the $2$-power cyclic covers of distinguished homology handles. We distinguish the $\Z_4$-equivariant theory in the real spin case from the $\Z_2$-equivariant theory in the real $\mathrm{spin}^c$ case.

\subsection{Defining the real Floer homotopy type}

We briefly recall the real Seiberg--Witten Floer homotopy type that will be used below; see \cite{KMT21,miyazawa2025satellite} for more details. Let $Y$ be a closed connected oriented smooth $3$-manifold equipped with an orientation-preserving smooth involution $\tau$, and let $(\s,I)$ be a real $\mathrm{spin}^c$ structure on $(Y,\tau)$ with spinor bundle $S$. We assume that
\[
H^1(Y;\R)^{-\tau^\ast}=0.
\]

Fix a $\tau$-invariant Riemannian metric $g$ and an $I$-invariant reference $\mathrm{spin}^c$ connection $B_0$. In Coulomb gauge, the formal gradient of the Chern--Simons--Dirac functional has the form $l+c$, where
\[
l=(*d,D_{B_0})
\]
is self-adjoint first order elliptic and $c$ is compact from Sobolev order $k$ to order $k-1$ for sufficiently large $k$. The real structure acts on the Coulomb slice by
\[
\mathcal I(a,\Phi)=(-\tau^\ast a,I\Phi).
\]
The Seiberg--Witten vector field is $\mathcal I$-equivariant, so its flow preserves the fixed-point slice
\[
\mathcal V_R=(i\ker d^\ast)^{-\tau^\ast}\oplus\Gamma(S)^I.
\]
Since $H^1(Y;\R)^{-\tau^\ast}=0$, this slice has no harmonic $1$-form directions, and the usual finite-dimensional approximation can be carried out as in the rational homology sphere case.

We first specify the stable homotopy categories in which we regard the resulting Conley indices. If $(\s,I)$ is induced from a real spin structure, the quaternionic symmetry of the spinor bundle gives a $\Z_4$-action. In this case, we use the universe
\[
\mathcal U_{\Z_4}
= \widetilde{\R}^{\,\infty}\oplus\C^\infty.
\]
If we retain only the real $\mathrm{spin}^c$ symmetry, we obtain a $\Z_2$-action and use
\[
\mathcal U_{\Z_2}
=
\R^\infty\oplus\widetilde{\R}^{\,\infty}.
\]
Thus the two cases considered below are
\[
(G,H)=(\Z_4,\Z_2)
\qquad\text{and}\qquad
(G,H)=(\Z_2,\Z_2).
\]

We briefly recall the corresponding stable category $\mathfrak C_G$. An object of $\mathfrak C_G$ is represented by a triple $(X,m,n)$, where $X$ is a pointed finite $G$-CW complex, $m\in\Z$, and $n\in\Q$. We require that $X^H$ be $G$-homotopy equivalent to a representation sphere and that $G$ act freely on $X\smallsetminus X^H$. For $G=\Z_4$, stabilization is taken with respect to $\widetilde{\R}$ and $\C$; for $G=\Z_2$, it is taken with respect to the corresponding real representations.

A morphism from $(X,m,n)$ to $(X',m',n')$ is represented by a pointed $G$-equivariant map after sufficiently many suspensions by finite-dimensional subrepresentations of the chosen universe. More precisely, when $G=\Z_4$, it is represented by a map
\[
\Sigma^V\Sigma^{m\widetilde{\R}+n\C}X
\longrightarrow
\Sigma^V\Sigma^{m'\widetilde{\R}+n'\C}X'
\]
for sufficiently large finite-dimensional subrepresentation $V\subset\mathcal U_G$, modulo further suspension and pointed $G$-equivariant homotopy. We denote these two stable categories by $\mathfrak C_{\Z_4}$ and $\mathfrak C_{\Z_2}$.\footnote{Sometime we will allow suspensions by rational powers of $\C$; these are formal, and we only define stable maps between them if rational powers of $\C$ on its domain and codomain differ by integers.}

For real numbers $\lambda<0<\mu$, let $\mathcal V^\mu_\lambda(g)\subset\mathcal V_R$ be the direct sum of the eigenspaces of $l\vert_{\mathcal V_R}$ whose eigenvalues lie in $(\lambda,\mu]$. Projecting $l+c$ to $\mathcal V^\mu_\lambda(g)$ and applying the usual cut-off gives a finite-dimensional equivariant flow. For sufficiently large $\mu$ and $-\lambda$, the relevant invariant set is contained in a compact equivariant isolating neighborhood $N\subset\mathcal V^\mu_\lambda(g)$.

An equivariant index pair is a pair $L\subset N$ such that $N\smallsetminus L$ contains the isolated invariant set in its interior, $L$ is positively invariant relative to $N$, and every positive trajectory leaving $N$ passes through $L$. The corresponding equivariant Conley index is
\[
I^\mu_\lambda(Y,\tau,\s,I,g) \coloneqq N/L.
\]
Its equivariant homotopy type is independent of the choice of index pair.

Let $\mathcal V^0_\lambda(g)$ denote the sum of the eigenspaces of $l\vert_{\mathcal V_R}$ with eigenvalues in $(\lambda,0]$. We define the metric-dependent real Seiberg--Witten Floer homotopy type by
\[
SWF_R(Y,\tau,\s,I,g)
 \coloneqq 
\Sigma^{-\mathcal V^0_\lambda(g)}
I^\mu_\lambda(Y,\tau,\s,I,g).
\]
Here the desuspension is taken in the stable category determined by the universe above. This desuspension compensates for changes in $\lambda$ and $\mu$, so the resulting stable homotopy type is independent of the finite-dimensional approximation once $g$ is fixed.

If $g_t$ is a path of $\tau$-invariant metrics, continuation identifies the corresponding Conley indices up to the representation suspension determined by the spectral flow of the real linearized Seiberg--Witten operators. With the assumption $H^1(Y; \R)^{-\tau^*}=0$, the corresponding spectral flow can be written as just half of the spectral flow of the Dirac operator with respect to a 1-parameter family of $\tau$-invariant metrics. 
Incorporating the corresponding real Atiyah--Patodi--Singer correction as a formal grading gives a metric-independent object
\[
SWF_R(Y,\tau,\s,I)
\]
in $\mathfrak C_{\Z_4}$ in the real spin case and in $\mathfrak C_{\Z_2}$ in the real $\mathrm{spin}^c$ case. See \cite{KMT21, miyazawa2025satellite} for the precise conventions. 

Let $(Y,\phi)$ be a distinguished homology handle and let $n>0$ be an integer. By \Cref{lem: homology handle real spin}, $Y^{2^n}$ admits a real spin structure with respect to the deck transformation $\tau_n$ of $Y^{2^n}\to Y^{2^{n-1}}$, unique up to isomorphism and sign. We denote the induced real $\mathrm{spin}^c$ structure by $(\s_{Y,\phi}^n,I_{Y,\phi}^n)$, where $I_{Y,\phi}^n$ is the real involution on the spinor bundle associated with $\s_{Y,\phi}^n$.

\begin{defn}
We define the finite $\Z_4$-spectrum
\[
SWF_R^n(Y,\phi)
 \coloneqq 
SWF_R(Y^{2^n},\tau_n,\s_{Y,\phi}^n,I_{Y,\phi}^n)
\in\mathfrak C_{\Z_4}.
\]
Restricting the action to $\Z_2\subset\Z_4$ gives a finite $\Z_2$-spectrum, which we denote by
\[
SWF_{R,c}^n(Y,\phi)\in\mathfrak C_{\Z_2}.
\]
\end{defn}

\begin{rem}
The spectra $SWF_R^n(Y,\phi)$ and $SWF_{R,c}^n(Y,\phi)$ are independent of the choice of sign of $\phi$. Indeed, replacing $\phi$ by $-\phi$ does not change the corresponding $2^n$-fold cyclic cover. If $t$ denotes the generator of its deck transformation group determined by $\phi$, then replacing $\phi$ by $-\phi$ replaces $t$ by $t^{-1}$. The deck involution of $Y^{2^n}\to Y^{2^{n-1}}$ is given by $t^{2^{n-1}}$, and $t^{-2^{n-1}}=t^{2^{n-1}}$. Thus the involution $\tau_n$ is unchanged. Since the real spin structure with respect to $\tau_n$ is unique up to sign, the resulting real Seiberg--Witten Floer spectra are canonically identified. In particular,
\[
SWF_R^n(Y,\phi)=SWF_R^n(Y,-\phi),
\qquad
SWF_{R,c}^n(Y,\phi)=SWF_{R,c}^n(Y,-\phi).
\]
\end{rem}

Our invariants lie in the groups $\mathcal{LE}_{\Z_4}$ and $\mathcal{LE}_{\Z_2}$, which we now define. A morphism
\[
f\colon (X,m,n)\longrightarrow (X',m',n')
\]
in $\mathfrak C_G$ is called a \emph{local map} if, after sufficiently many suspensions by finite-dimensional subrepresentations of the chosen universe, it has a representative of the form $\widetilde f\colon V^+\wedge X\to W^+\wedge X'$ such that the induced map on $H$-fixed point sets
\[
\widetilde f^{\,H}\colon (V^H)^+\wedge X^H
\longrightarrow
(W^H)^+\wedge (X')^H
\]
is a $G/H$-equivariant homotopy equivalence. This condition is independent of the choice of stabilized representative of $f$.

Two objects $(X,m,n)$ and $(X',m',n')$ of $\mathfrak C_G$ are said to be \emph{locally equivalent} if there exist local maps
\[
(X,m,n)\longrightarrow (X',m',n')
\qquad\text{ and }\qquad
(X',m',n')\longrightarrow (X,m,n).
\]
We denote the local equivalence class of $(X,m,n)$ by $[(X,m,n)]_{\operatorname{loc}}$.

Let $\mathcal{LE}_G$ denote the set of local equivalence classes. The smash product gives $\mathcal{LE}_G$ the structure of an abelian group:
\[
[(X,m,n)]_{\operatorname{loc}}
+
[(X',m',n')]_{\operatorname{loc}}
 \coloneqq 
[(X\wedge X',m+m',n+n')]_{\operatorname{loc}}.
\]
The identity element is $[(S^0,0,0)]_{\operatorname{loc}}$, and the inverse is given by the Spanier--Whitehead dual together with the grading change $(-m,-n)$. For $G=H=\Z_2$, the analogous definitions use $\R$ and $\widetilde{\R}$ as the stabilizing representations. Restricting the action to the subgroup $\Z_2\subset\Z_4$ gives a natural homomorphism
\[
\mathcal{LE}_{\Z_4}\longrightarrow\mathcal{LE}_{\Z_2}.
\]

\begin{defn}
We may regard representatives of local equivalence classes in $\mathcal{LE}_{\Z_4}$ as finite $\Z_4$-spectra $X$ such that $X^{\Z_2}$ is equivalent to the sphere spectrum up to suspension. We call such $X$ \emph{SWF-like $\Z_4$-spectra}. Restricting the action to $\Z_2\subset\Z_4$ gives spectra that we call \emph{SWF-like $\Z_2$-spectra}.
\end{defn}

\begin{rem}
For any distinguished homology handle $(Y,\phi)$ and any integer $n>0$, we have
\[
[SWF_R^n(Y,\phi)]_{\operatorname{loc}}\in\mathcal{LE}_{\Z_4};
\qquad
[SWF_{R,c}^n(Y,\phi)]_{\operatorname{loc}}\in\mathcal{LE}_{\Z_2}.
\]
We will use these classes extensively in \Cref{sec: additive invariants}.
\end{rem}

\subsection{Real Bauer--Furuta invariants}

We next recall the real Bauer--Furuta invariant for a cobordism. We describe only the part of the construction used in this paper; see \cite{KMT21,miyazawa2025satellite} for details. Let $(W,\tau,\s,I)$ be a compact oriented real $\mathrm{spin}^c$ cobordism from $(Y_0,\tau_0,\s_0,I_0)$ to $(Y_1,\tau_1,\s_1,I_1)$. Thus $\partial W=-Y_0\sqcup Y_1$, the involution $\tau$ preserves the orientation and each boundary component, and $(\s,I)$ restricts to the prescribed real $\mathrm{spin}^c$ structures on the boundary. We assume that
\[
\tau_i^\ast=\id\text{ on }H^1(Y_i;\R)\quad(i=0,1)
\qquad\text{ and }\qquad
\tau^\ast=\id\text{ on }H^1(W;\R).
\]

Choose a $\tau$-invariant Riemannian metric $g$ on $W$ and an $I$-invariant reference $\mathrm{spin}^c$ connection $A_0$, both of product form near the boundary. We use the double Coulomb slice
\[
(i\Omega^1(W)_{\mathrm{CC}})^{-\tau^\ast}
=
\{a\in i\Omega^1(W)\mid d^\ast a=0,\ d_\partial^\ast(a\vert_{\partial W})=0,\ \tau^\ast a=-a\},
\]
where $d_\partial^\ast$ is the $L^2$-formal adjoint of $d$ on $\partial W$. \footnote{The usual componentwise normalization
$\int_C\mathbf t_C(*a)=0$ for each connected component
$C\subset\partial W$ is automatic on the anti-invariant slice, since
$\tau$ preserves the orientation of each boundary component.}

After taking suitable Sobolev completions, the real Seiberg--Witten map together with boundary restriction gives a map
\begin{equation}\label{eq: real monopole map with boundary}
\mathcal F_W\colon (i\Omega^1(W)_{\mathrm{CC}})^{-\tau^\ast}\oplus\Gamma(S^+)^I \longrightarrow (i\Omega^+(W))^{-\tau^\ast}\oplus\Gamma(S^-)^I\oplus\mathcal V_R(\partial W).
\end{equation}
Here $S^\pm$ denote the positive and negative spinor bundles, and $\mathcal V_R(\partial W)$ denotes the real Coulomb slice on the boundary introduced above. More explicitly, the first two components of \eqref{eq: real monopole map with boundary} are given by
\[
(a,\Phi)\longmapsto
(d^+a-q(\Phi) + F_{A_0}^+,D_{A_0+a}\Phi),
\]
and the last component is obtained by restricting $(a,\Phi)$ to the boundary. As usual, after choosing the reference connection $A_0$, we write $\mathcal F_W=L_W+C_W$, where $L_W$ is the linearization and $C_W$ is compact on bounded subsets.

We briefly explain how the relative invariant is obtained from this map. Choose sufficiently large spectral cutoffs on $Y_0$ and $Y_1$, and let $(N_i,L_i)$ be index pairs for the corresponding finite-dimensional approximations of the real Seiberg--Witten flows. Choose a sufficiently large finite-dimensional subspace of the target of \eqref{eq: real monopole map with boundary}, and take its inverse image under the linear part $L_W$. The projection of $\mathcal F_W$ then gives, after one-point compactification, an equivariant map of the form
\begin{equation}\label{eq: finite dimensional real BF}
\psi\colon
U^+\wedge(N_0/L_0)
\longrightarrow
(U')^+\wedge(N_1/L_1).
\end{equation}
The properness needed to extend the finite-dimensional approximation to the one-point compactifications follows from the compactness statement used in the construction of the relative Bauer--Furuta invariant. We do not review this analytic part here.

The difference between the representations $U$ and $U'$ is determined by the index of the linearized Seiberg--Witten operator. The spinor part contributes the real APS index of the Dirac operator, while the $1$-form part contributes the positive part of the real intersection form. We denote these by $\ind^R_{\mathrm{APS}}(D_{W,g})$ and $H_R^+(W)$, respectively. Thus, after desuspending the boundary Conley indices as in the definition of the metric-dependent real Floer homotopy types, \eqref{eq: finite dimensional real BF} determines a stable map
\begin{equation}\label{eq: metric dependent real BF}
BF^{R,c}_{W,\s,I,g}\colon \Sigma^{\ind^R_{\mathrm{APS}}(D_{W,g})}SWF_R(Y_0,\tau_0,\s_0,I_0,g_0) \longrightarrow \Sigma^{H_R^+(W)}SWF_R(Y_1,\tau_1,\s_1,I_1,g_1).
\end{equation}
Here $g_i=g\vert_{Y_i}$, and $H_R^+(W)$ can equivalently be described as the cokernel contribution of the real $d^+$-operator. Its dimension is the positive rank of the anti-invariant part of the intersection form.

As in the three-dimensional construction, \eqref{eq: metric dependent real BF} is initially metric-dependent. If the metric varies through $\tau$-invariant metrics, the changes in the boundary Floer spectra are measured by the corresponding spectral flows, while the change in the APS index is given by the same spectral-flow terms with the opposite sign. Consequently, after incorporating the correction terms used in the definition of the metric-independent real Floer homotopy types, these changes cancel. We therefore obtain a metric-independent stable map, which we write as
\begin{equation}\label{eq: metric independent real BF}
BF^{R,c}_{W,\s,I}\colon \Sigma^{\ind^R_t(D_W)}SWF_R(Y_0,\tau_0,\s_0,I_0) \longrightarrow \Sigma^{H_R^+(W)}SWF_R(Y_1,\tau_1,\s_1,I_1).
\end{equation}
Here $\ind^R_t(D_W)$ denotes the metric-independent topological real Dirac index obtained from the APS index together with the boundary correction terms.  If $(\s,I)$ is induced from a real spin structure, the additional quaternionic symmetry of the spinor bundles makes the construction $\Z_4$-equivariant. 

% We suppress the representation suspensions in \eqref{eq: metric independent real BF} whenever they are absorbed into the gradings of the Floer spectra.

The main property of the real Bauer--Furuta invariant needed below concerns its restriction to the fixed-point set. Suppose that $\tau$ is free, let $\overline W=W/\tau$, and let $\ell$ be the rank-one $\Z$-local system on $\overline W$ associated with the double covering $W\to\overline W$. Anti-invariant differential forms on $W$ are naturally identified with differential forms on $\overline W$ with coefficients in the associated flat real line bundle $\ell\otimes_{\Z}\R$. In particular, the anti-invariant part of the de Rham complex on $W$ is identified with the $\ell\otimes_{\Z}\R$-twisted de Rham complex on $\overline W$.

Let $H\subset G$ be the subgroup used in the definition of a local map. Thus $G=H=\Z_2$ in the real $\mathrm{spin}^c$ case, while $G=\Z_4$ and $H=\langle j^2\rangle\cong\Z_2$ in the real spin case. The nontrivial element of $H$ acts by $-1$ on the spinor summands and trivially on the $1$-form summands. Hence all spinor coordinates vanish on the $H$-fixed-point set, and the restriction of the real Bauer--Furuta map to this set is determined entirely by the linear $1$-form part of the monopole map. This gives the following consequence of the relative real Bauer--Furuta construction.

\begin{prop}\label{prop: real BF gives local map}
Let $(W,\tau,\s,I)$ be as above and suppose that $\tau$ is free. Set $\overline W=W/\tau$, and let $\ell$ be the rank-one $\Z$-local system associated with the double covering $W\to\overline W$. Suppose that
\begin{equation}\label{eq: local BF hypotheses}
H^1(\overline W;\ell\otimes_{\Z}\R)=0
\qquad\text{ and }\qquad
b^+_\ell(\overline W)=0 
\end{equation}
where $b^+_\ell(\overline W)$ denotes the positive rank of the intersection form with coefficients in $\ell\otimes_{\Z}\R$. Then the stable map 
\[
BF^{R,c}_{W,\s,I}\colon \Sigma^{\ind^R_t(D_W)}SWF_R(Y_0,\tau_0,\s_0,I_0) \longrightarrow SWF_R(Y_1,\tau_1,\s_1,I_1)
\]
is a local map. If $(\s,I)$ is induced from a real spin structure, then the $\Z_4$-equivariant map $BF^R_{W,\s,I}$ is also a local map.
\end{prop}

\begin{proof}
We first consider the real $\mathrm{spin}^c$ case, so that $G=H=\Z_2$. Choose a finite-dimensional approximation $\psi$ as in \eqref{eq: finite dimensional real BF}. On the $H$-fixed-point set, the spinor coordinates vanish. Hence the quadratic term $q(\Phi)$ and the other nonlinear terms vanish, and the restriction $\psi^H$ is obtained from the finite-dimensional approximation of the linear $1$-form part of the monopole map.

Since $\tau$ is free, anti-invariant forms on $W$ are identified with forms on $\overline W$ with coefficients in $\ell\otimes_{\Z}\R$. Thus the relevant fixed-point operator is the twisted double Coulomb operator on $\overline W$, together with the boundary spectral projections determined by the chosen cutoffs. The standard Hodge-theoretic calculation for this boundary problem shows that the assumption $H^1(\overline W;\ell\otimes_{\Z}\R)=0$ makes this operator injective and identifies its cokernel with the space $H^+_\ell(\overline W)$ of twisted self-dual harmonic $2$-forms. In particular, $\dim H^+_\ell(\overline W)=b^+_\ell(\overline W)$.

After cancelling the boundary eigenspaces appearing in the desuspensions defining the two Floer spectra, the restriction of a sufficiently stabilized representative of $BF^{R,c}_{W,\s,I}$ to the $H$-fixed-point sets is therefore stably homotopic to the one-point compactification of
\[
0\longrightarrow H^+_\ell(\overline W).
\]
By \eqref{eq: local BF hypotheses}, the target of this inclusion is zero. Hence the fixed-point restriction is a pointed homotopy equivalence. This is precisely the condition that $BF^{R,c}_{W,\s,I}$ be a local map.

Suppose now that $(\s,I)$ is induced from a real spin structure. The same finite-dimensional approximation is $\Z_4$-equivariant, with $H=\langle j^2\rangle$. Its restriction to the $H$-fixed-point set is the same linear $1$-form map considered above, now retaining the action of $G/H\cong\Z_2$. The preceding one-point compactification is therefore a $G/H$-equivariant homotopy equivalence. Thus $BF^R_{W,\s,I}$ is also a local map.
\end{proof}

We now apply the local-map criterion to sufficiently large cyclic covers of $\wt{H}_{\Q}$-cobordisms. The resulting maps will be used to establish cobordism invariance in the next section.

\begin{thm}\label{thm: cyclic covers give local maps}
Let $(W,\phi_W)$ be an $\wt{H}_{\Q}$-cobordism from $(Y,\phi)$ to $(Y',\phi')$, and let $\tau_n^W$ denote the deck transformation of the double covering $W^{2^n}\to W^{2^{n-1}}$. For all sufficiently large integers $n>0$, every real $\mathrm{spin}^c$ structure $(\s_W^n,I_W^n)$ on $W^{2^n}$ with respect to $\tau_n^W$ extending the prescribed boundary structures induces a $\Z_2$-equivariant local map
\[
BF^{R,c}_{W^{2^n},\s_W^n,I_W^n}\colon SWF_{R,c}^n(Y,\phi)\longrightarrow SWF_{R,c}^n(Y',\phi').
\]
If $(\s_W^n,I_W^n)$ is induced from a real spin structure extending the chosen boundary real spin structures, then the map
\[
BF^R_{W^{2^n},\s_W^n,I_W^n}:SWF^n_R(Y,\phi)\rightarrow SWF^n_R(Y',\phi')
\]
is also a local map.
\end{thm}

\begin{proof}
Let $\ell_{n-1}$ be the rank-one $\Z$-local system on $W^{2^{n-1}}$ associated with the double covering $W^{2^n}\to W^{2^{n-1}}$. By \Cref{cor: f-homology action}, for all sufficiently large integers $n>0$, we have
\[
H^\ast(W^{2^{n-1}};\ell_{n-1}\otimes_{\Z}\R)=0.
\]
Identifying these groups with the anti-invariant cohomology of $W^{2^n}$, we obtain
\[
H^1(W^{2^n};\R)^{-(\tau_n^W)^\ast}=0
\qquad\text{ and }\qquad
H_R^+(W^{2^n})=0.
\]
The corresponding vanishing of anti-invariant first cohomology on the boundary follows from \Cref{lem: homology handle real spin}. Thus the hypotheses of \Cref{prop: real BF gives local map} are satisfied.

It remains to show that the normalized Dirac-index suspension vanishes. Since $(\s_W^n,I_W^n)$ is a real $\mathrm{spin}^c$ structure, we have
\[
(\tau_n^W)^\ast c_1(\s_W^n)=-c_1(\s_W^n).
\]
On the other hand, \Cref{cor: f-homology action} implies that $(\tau_n^W)^\ast$ acts trivially on $H^2(W^{2^n};\Q)$ for sufficiently large $n$. Hence $c_1(\s_W^n)$ is torsion, so $c_1(\s_W^n)^2=0$. Moreover, \Cref{lem: vanishing of intersection form} gives $\sigma(W^{2^n})=0$. By the APS index formula, after incorporating the boundary correction terms used to normalize the Floer spectra, the real dimension of the Dirac-index grading is
\[
\frac{c_1(\s_W^n)^2-\sigma(W^{2^n})}{8}=0.
\]
The Dirac-index grading is a multiple of the sign representation $\widetilde{\R}$ in the real $\mathrm{spin}^c$ case and of the real rotation representation $\C$ in the real spin case. Thus the normalized index suspension vanishes in both cases.

Applying \Cref{prop: real BF gives local map}, with both the Dirac-index suspension and the positive-form suspension equal to zero, gives the stated $\Z_2$-equivariant local map. When the structure is induced from a real spin structure compatible with the chosen boundary structures, the same proposition gives the stated $\Z_4$-equivariant local map.
\end{proof}

The following proposition is the precise form of real excision used
in this paper.  It is a local-equivalence version of
\cite[Theorem~5.1 and Section~5.2]{miyazawa2025satellite}, in which
the boundary components are allowed to be homology
$S^1\times S^2$'s rather than homology spheres. 
Since the proofs are identical to the original one and just based on computations of (local coefficient) cohomology, we omit them. 

\begin{prop}
\label{prop: cyclic real excision}
Let $(Y,\phi)$ and $(Y',\phi')$ be distinguished homology handles.
Choose oriented simple closed curves $\gamma\subset Y$ and
$\gamma'\subset Y'$ satisfying
$\phi([\gamma])=\phi'([\gamma'])=1$, and let $N_\gamma$ and
$N_{\gamma'}$ be their tubular neighborhoods.  Using the gluing
identification appearing in the definition of the circle sum, set
\[
Z\coloneqq
(Y\smallsetminus\operatorname{int}N_\gamma)
\cup
(Y'\smallsetminus\operatorname{int}N_{\gamma'}),
\qquad
Z'\coloneqq N_\gamma\cup N_{\gamma'}.
\]
Let $\psi$ and $\psi'$ be the induced distinguished classes.  Thus
$(Z,\psi)$ represents the circle sum of $(Y,\phi)$ and
$(Y',\phi')$, while
$(Z',\psi')\cong(S^1\times S^2,\xi)$.

Let
\[
X\colon Y\sqcup Y'\longrightarrow Z\sqcup Z'
\]
be the standard excision cobordism induced by the given splicings of $Y$ and $Y'$ (and thus also of $Z$ and $Z'$).  It carries a primitive
class $\phi_X\in H^1(X;\Z)$ restricting to the distinguished classes
on all four boundary components.  For $n\geq0$, let
$X_n\coloneqq X^{2^n}$ be the cyclic cover determined by $\phi_X$,
and, for $n>0$, let $
p_n\colon X_n\longrightarrow X_{n-1}$
be the resulting double covering with deck transformation $\tau_n$.
Denote the restrictions of $\tau_n$ to the boundary components by
the same symbol.

Then the following hold for every $n>0$.

\begin{enumerate}
\item
The four boundary components of $X_n$ are the corresponding cyclic
covers of $Y$, $Y'$, $Z$, and $Z'$.  Each is an
$\mathbb F_2$-homology $S^1\times S^2$ and admits a real spin
structure with respect to $\tau_n$, unique up to isomorphism and
sign.

\item
The manifold $X_n$ is spin, and $(X_n,\tau_n)$ admits a real spin
structure $
(\mathfrak t_n,\widetilde\tau_n)$,
unique up to isomorphism and sign.  Its restriction to each boundary
component agrees, after choosing the sign, with the real spin
structure in (1).  

\item
We have $
H^\ast(X_n;\R)^{-\tau_n^\ast}=0, 
\sigma(X_n)=0$.
Consequently, the real Bauer--Furuta maps associated with $X_n$ and
its reverse,
\[
\begin{aligned}
BF^R_{X_n}\colon\;&
SWF_R^n(Y,\phi)\wedge SWF_R^n(Y',\phi')
   \longrightarrow
SWF_R^n(Z,\psi)\wedge SWF_R^n(Z',\psi'),\\
BF^R_{-X_n}\colon\;&
SWF_R^n(Z,\psi)\wedge SWF_R^n(Z',\psi')
   \longrightarrow
SWF_R^n(Y,\phi)\wedge SWF_R^n(Y',\phi'),
\end{aligned}
\]
are $\Z_4$-equivariant local maps.  In particular,
\[
SWF_R^n(Y,\phi)\wedge SWF_R^n(Y',\phi')
\quad\text{and}\quad
SWF_R^n(Z,\psi)\wedge SWF_R^n(Z',\psi')
\]
are locally equivalent.  Restricting the symmetry from $\Z_4$ to
$\Z_2$ gives the analogous local equivalence for
$SWF_{R,c}^n$.
\end{enumerate}

More generally, any real $\spinc$ structure on $(X_n,\tau_n)$ gives
$\Z_2$-equivariant local maps in both directions between the real
$\spinc$ Floer spectra associated with its boundary restrictions. \qed
\end{prop}

\begin{rem}\label{rem: spinc structures in cyclic real excision}
Real $\spinc$ extensions need not be unique.  Let $\ell_{n-1}$ be the
rank-one integral local system on $X_{n-1}$ associated with $p_n$.
After fixing one real $\spinc$ structure
$(\mathfrak s_n^0,I_n^0)$ on $(X_n,\tau_n)$, the isomorphism classes
of real $\spinc$ structures form a torsor over
\[
H^2(X_{n-1};\ell_{n-1}).
\]
Under this action, changing the real $\spinc$ structure by
$a\in H^2(X_{n-1};\ell_{n-1})$ changes its underlying $\spinc$
structure by $p_n^\ast a$.  The boundary restriction map is affine
over
\[
\operatorname{res}\colon
H^2(X_{n-1};\ell_{n-1})
\longrightarrow
H^2(\partial X_{n-1};
          \ell_{n-1}|_{\partial X_{n-1}}).
\]
Consequently, a prescribed tuple of boundary real $\spinc$ structures
extends if and only if its difference from the restriction of
$(\mathfrak s_n^0,I_n^0)$ lies in the image of
$\operatorname{res}$.  When it extends, the set of extensions is a
torsor over $\ker(\operatorname{res})$.  If the underlying $\spinc$
structure on $X_n$ is also fixed, the possible differences are
further restricted to
\[
\ker(\operatorname{res})\cap\ker(p_n^\ast).
\]

For the application in this paper, we take
$(\mathfrak s_n^0,I_n^0)$ to be the real $\spinc$ structure induced
by the real spin structure
$(\mathfrak t_n,\widetilde\tau_n)$.  Its boundary restrictions are
precisely the real $\spinc$ structures induced by the canonical
boundary real spin structures, so the required tuple always extends. 
\end{rem}

Finally, we record a fact that will not be used in this paper but can be potentially useful for developing and computing a parallel theory using real Heegaard Floer homology, which is a Heegaard Floer version of real Seiberg--Witten theory defined in a work of Guth--Manolescu \cite{guth2025real}. This is because, currently, real Heegaard Floer homology is not yet defined for free involutions; see \cite[Theorem 1]{guth2025real} for a precise description of a setting where it is well-defined.

\begin{prop}
    For any distinguished homology handle $(Y, \phi)$ and $n \geq 1$, we have  
    \[
    [SWF_{R}^n(Y, \phi) ]_{\mathrm{loc}} = [SWF_R^{n} (Z, K)]_{\mathrm{loc}}
    \] 
    for an oriented homology 3-sphere $Z$ and oriented knot $K$ in $Z$, where $
SWF_R^{n} (Z, K)$ is a $\Z_4$-equivariant spectrum obtained by applying the $\Z_4$-equivariant real Floer homotopy type \cite{KMT21} to the $2^n$-fold branched cover along $K$ equipped with the unique odd spin structure with respect to $\tau$ and with the $\Z_2$-subgroup of $\Z_{2^n}$.

\end{prop}
\begin{proof}
   For any distinguished homology handle $(Y, \phi)$, we can take a pair $  (Z, K)$ of an oriented homology 3-sphere $Z$ and an oriented knot $K \subset Z$ such that its $0$-surgery satisfies $(Z_0(K), \phi_K) \cong (Y, \phi)$, where $\phi_K$ is a generator of $H^1(Z_0(K))$ comes from the orientation of $K$. 
Therefore, we have 
\[
[SWF_{R}^n(Z_0(K), \phi_K)]_{\mathrm{loc}} = [SWF_{R}^n(Y, \phi) ]_{\mathrm{loc}}.
\]
On the other hand, we have the zero-surgery trace cobordism $ X_0(K)$ from $ Z$ to $Z_0(K)$. By considering the branched $2^n$-cover $X_0(K)^{2^n}$ along the 2-handle core of $X_0(K)$, we get a $\mathbb{Z}_{2^n}$-equivariant coborsism from $\Sigma_{2^n}(Z,K)$ to $Z_0(K)^{2^n}$, where $\Sigma_{2^n}(Z,K)$ denotes the $2^n$-fold branched cover along $K$ and $Z_0(K)^{2^n}$ denotes the $2^n$-cover associated to $\phi_K$.
One can check $X_0(K)^{2^n}$ and $-X_0(K)^{2^n}$ equipped with real spin structures give a local equivalence 
\[
[SWF_R^{n} (Z, K)]_{\mathrm{loc}} =[SWF_{R,c}^n(Z_0(K), \phi_K)]_{\mathrm{loc}} . 
\]
\end{proof}

% We finally record how this applies to the cyclic covers considered in this paper. Let $(W,\phi_W)$ be an $\wt{H}_{\Q}$-cobordism from $(Y,\phi)$ to $(Y',\phi')$, and let $\tau_n^W$ be the deck transformation of the double covering $W^{2^n}\to W^{2^{n-1}}$. Denote by $\ell_{n-1}$ the associated rank-one $\Z$-local system on $W^{2^{n-1}}$. For all sufficiently large integers $n>0$, \Cref{cor: f-homology action} gives $H^\ast(W^{2^{n-1}};\ell_{n-1}\otimes_{\Z}\R)=0$. In particular, the vanishing in degree $2$ implies that the twisted intersection form has no positive part. Thus \Cref{prop: real BF gives local map} applies, and every real $\mathrm{spin}^c$ structure $(\s_W^n,I_W^n)$ on $W^{2^n}$ with respect to $\tau_n^W$ extending the prescribed boundary structures gives a local map
% \[
% BF^{R,c}_{W^{2^n},\s_W^n,I_W^n}\colon SWF_{R,c}^n(Y,\phi)\longrightarrow SWF_{R,c}^n(Y',\phi').
% \]
% If $(\s_W^n,I_W^n)$ is induced from a real spin structure, the same construction gives a $\Z_4$-equivariant local map
% \[
% BF^R_{W^{2^n},\s_W^n,I_W^n}\colon SWF_R^n(Y,\phi)\longrightarrow SWF_R^n(Y',\phi').
% \]
% These are the cobordism maps used below.

\section{Additive $\wt{H}_{\Q}$-cobordism invariants from real Seiberg--Witten theory}\label{sec: additive invariants}

In this section, we construct additive invariants of distinguished homology handles from the real Floer spectra of their cyclic covers. We first establish cobordism invariance and then prove additivity under the group operation.

% an oriented homology $S^1\times S^2$ and a positive integer $k$

%\textcolor{red}{Probably explain local equivalences in both $\Z_4$ and $\Z_2$ settings? Also the corresponding local equivalence groups $\mathcal{LE}_{\Z_4}$ and $\mathcal{LE}_{\Z_2}$? For their construction, let's use the following terminology: we say that a finite $\Z_4$-spectrum is \emph{SWF-like} if its $\Z_2$-fixed locus is $\Z_4$-equivariantly stably homotopy equivalent to a representation sphere.}
\subsection{Constructing the invariants $SWF_R^\infty$ and $SWF_{R,c}^\infty$}

We first establish the following cobordism invariance property of the spectra constructed in \Cref{sec: real SWF for handles}.

\begin{lem}\label{lem: local equivalence lemma}
Let $(Y,\phi)$ and $(Y',\phi')$ be distinguished homology handles. If $(Y,\phi)\sim_{\mathbb{F}_2}(Y',\phi')$, then $SWF_R^n(Y,\phi)$ and $SWF_R^n(Y',\phi')$ are locally equivalent for all sufficiently large integers $n>0$. If $(Y,\phi)\sim_{\Q}(Y',\phi')$, then $SWF_{R,c}^n(Y,\phi)$ and $SWF_{R,c}^n(Y',\phi')$ are locally equivalent for all sufficiently large integers $n>0$.
\end{lem}

\begin{proof}
Choose a cobordism realizing the given relation. For all sufficiently large $n$, \Cref{lem: Z2-Kawauchi implies real spin,thm: Q-Kawauchi implies real spin c} provide the required real spin and real $\mathrm{spin}^c$ structures, respectively. Applying \Cref{thm: cyclic covers give local maps} to the cobordism and its reverse gives local maps in both directions, proving the desired local equivalences.
\end{proof}

Our aim is to construct homomorphisms from $\Omega_{\mathrm{spin}}(S^1\times S^2)$ and $\Omega(S^1\times S^2)$ using the Floer spectra constructed above. We first pass to reduced products of local equivalence groups, which identify sequences that agree for all sufficiently large indices. By \Cref{lem: local equivalence lemma}, this gives well-defined maps on the cobordism groups; we will prove their additivity below.

\begin{defn}
Given a collection $\{A_i\}_{i\in I}$ of abelian groups, we define their \emph{reduced product} to be the quotient group
\[
\wt\prod_{i\in I}A_i \coloneqq \prod_{i\in I}A_i/\bigoplus_{i\in I}A_i,
\]
where $\bigoplus_{i\in I}A_i\subset\prod_{i\in I}A_i$ consists of tuples with only finitely many nonzero entries. If each $A_i$ is torsion-free, then so is $\wt\prod_{i\in I}A_i$.
\end{defn}

\begin{lem}\label{lem: product of local equivalence groups}
The maps
\[
\begin{aligned}
SWF_R^\infty\colon \Omega_{\mathrm{spin}}(S^1\times S^2)&\longrightarrow\wt\prod_{n>0}\mathcal{LE}_{\Z_4};\qquad [(Y,\phi)]\longmapsto[([SWF_R^n(Y,\phi)]_{\operatorname{loc}})_{n>0}],\\
SWF_{R,c}^\infty\colon \Omega(S^1\times S^2)&\longrightarrow\wt\prod_{n>0}\mathcal{LE}_{\Z_2};\qquad [(Y,\phi)]\longmapsto[([SWF_{R,c}^n(Y,\phi)]_{\operatorname{loc}})_{n>0}],
\end{aligned}
\]
are well-defined functions.
\end{lem}

\begin{proof}
By \Cref{lem: local equivalence lemma}, cobordant representatives determine the same local equivalence classes for all sufficiently large $n$. Their sequences therefore differ in only finitely many entries and define the same class in the reduced product.
\end{proof}

We call $SWF_R^\infty$ and $SWF_{R,c}^\infty$ the \emph{asymptotic real spin} and \emph{asymptotic real $\mathrm{spin}^c$ Seiberg--Witten Floer invariants}, respectively. The next step is to show that $SWF_R^\infty$ and $SWF_{R,c}^\infty$ are group homomorphisms.

\begin{lem}\label{lem: identity element}
Let $(S^1\times S^2,\phi)$ be the distinguished homology handle with $\phi$ a generator of $H^1(S^1\times S^2;\Z)$. Then, for every integer $n>0$, we have
\[
SWF_R^n(S^1\times S^2,\phi)\simeq S^0.
\]
\end{lem}
\begin{proof}
The $2^n$-fold cyclic cover of $S^1\times S^2$ determined by $\phi$ is given by
\[
S^1\times S^2\longrightarrow S^1\times S^2;\qquad (z,w)\longmapsto(z^{2^n},w).
\]
Its deck involution $(z,w)\mapsto(-z,w)$ preserves the product metric of positive scalar curvature. For this metric, the unique real reducible is nondegenerate, and there are no irreducible solutions. Hence the metric-dependent real Floer homotopy type is $S^0$ with a metric-dependent suspension. The APS grading correction also vanishes for this product metric, as can be seen using the spin filling $S^1\times D^3$ with a positive scalar curvature metric that is a product near the boundary. Therefore $SWF_R^n(S^1\times S^2,\phi)\simeq S^0$, as desired.
\end{proof}

\begin{lem}\label{lem: additivity}
The maps $SWF_R^\infty$ and $SWF_{R,c}^\infty$ are group homomorphisms.
\end{lem}
\begin{proof}
This follows from \cref{prop: cyclic real excision} and \cref{lem: identity element}. 
\end{proof}

\begin{rem}
Although \Cref{thm:main} only detects a subgroup isomorphic to $\Z$ in the kernel of the smooth-to-topological homomorphism, the image of $SWF_{R,c}^\infty$ has infinite rank. To see this, consider
\[
K_p \coloneqq T(2,2^p-1),\qquad p\geq 2.
\]
By \cite[Theorem~1.6]{fukumoto2026real}, the sequence of local equivalence classes associated with $K_p$, and hence the sequence obtained by applying any of the real Fr{\o}yshov invariants introduced below, is eventually periodic in the covering exponent. In this case the period is exactly $p$, since the order of $2$ modulo $2^p-1$ is $p$. Moreover, \cite[Lemma~5.15]{fukumoto2026real} gives
\[
\delta_R^{(1)}(K_p)=\frac{2^{p-1}-1}{8}, \qquad \delta_R^{(k)}(K_p)=-\frac18 \quad (2\leq k\leq p-1),
\qquad
\delta_R^{(p)}(K_p)=\frac{2^p-1}{8},
\]
so this periodic tail is nonconstant.
From this computation, we can see
\[
SWF_{R,c}^\infty([S^3_0(K_p)]),
\]
are linearly independent, and therefore the image of $SWF_{R,c}^\infty$ has infinite rank. Notice, however, that this does not directly strengthen \Cref{thm:main}: the manifolds $S^3_0(K_p)$ are already nontrivial in the topological Kawauchi cobordism group. Indeed, the signatures
\[
\sigma(K_p)=-(2^p-2)
\]
are nonzero, so the knots $K_p$ represent nontrivial algebraic concordance classes, which are detected by the natural map from the topological Kawauchi cobordism group to the algebraic concordance group.
\end{rem}

\subsection{SWF-like $\Z_4$-spectra and almost $j$-complexes}

We regard $\mathbb{F}_2[v]$ as a dg algebra with $v$ of degree $1$ and zero differential. We say that a dg module $C$ over $\mathbb{F}_2[v]$ is \emph{SWF-like} if it is perfect and $C\otimes_{\mathbb{F}_2[v]}\mathbb{F}_2[v,v^{-1}]$ is homotopy equivalent to $\mathbb{F}_2[v,v^{-1}]$ up to a degree shift. Throughout this subsection, we represent perfect modules by finite free dg models and allow overall rational grading shifts. Given SWF-like complexes $C$ and $C'$, we say that a dg $\mathbb{F}_2[v]$-module homomorphism $f\colon C\to C'$ is \emph{local} if the localized map
\[
f\otimes\id\colon C\otimes_{\mathbb{F}_2[v]}\mathbb{F}_2[v,v^{-1}]
\longrightarrow C'\otimes_{\mathbb{F}_2[v]}\mathbb{F}_2[v,v^{-1}]
\]
is a homotopy equivalence.

\begin{defn}
Given a dg module $C$ over $\mathbb{F}_2[v]$, we define its \emph{hat-flavored truncation} by
\[
\hat{C}\coloneqq C\otimes^L_{\mathbb{F}_2[v]}\mathbb{F}_2.
\]
Similarly, given a dg $\mathbb{F}_2[v]$-module homomorphism $f\colon C\to D$, we define its \emph{hat-flavored truncation} by
\[
\hat{f}\coloneqq f\otimes\id\colon
C\otimes^L_{\mathbb{F}_2[v]}\mathbb{F}_2
\longrightarrow D\otimes^L_{\mathbb{F}_2[v]}\mathbb{F}_2.
\]
\end{defn}

\begin{defn}
An \emph{almost $j$-complex} is a pair $(C,j)$, where $C$ is an SWF-like dg module over $\mathbb{F}_2[v]$ and $j$ is an endomorphism class of $\hat{C}$ satisfying $j^2\sim\id_{\hat{C}}+\hat{\Phi}_v$, where $\Phi_v$ denotes the derived endomorphism class of $C$ defined in \Cref{rem: Phi v for Z2 is well defined}. Given two almost $j$-complexes $(C,j)$ and $(C',j')$, a dg $\mathbb{F}_2[v]$-module homomorphism $f\colon C\to C'$ is called an \emph{isomorphism}, a \emph{homotopy equivalence}, or an \emph{almost $j$-local map}, respectively, if it is an isomorphism, a homotopy equivalence, or a local map of the underlying dg modules and satisfies $\hat{f}\circ j\sim j'\circ\hat{f}$. We say that $(C,j)$ and $(C',j')$ are \emph{almost $j$-locally equivalent} if they admit almost $j$-local maps in both directions.
\end{defn}

Such almost $j$-complexes arise naturally in real Seiberg--Witten theory. Let $(Y,\phi)$ be a distinguished homology handle and let $n>0$ be an integer. Then $SWF_R^n(Y,\phi)$ is a finite $\Z_4$-spectrum whose $\Z_2$-fixed-point spectrum is $\Z_4$-equivariantly stably homotopy equivalent to a representation sphere. We use a reduced Borel cochain model $C^\ast_{\Z_2}(SWF_R^n(Y,\phi);\mathbb{F}_2)$ over $\mathbb{F}_2[v]$. This module is perfect, and the Borel localization theorem implies that it is SWF-like. The action of the generator $j\in\Z_4$ induces a dg $\mathbb{F}_2[v]$-module endomorphism $j^\ast$ satisfying $(j^\ast)^2\sim\id+\Phi_v$ by \Cref{lem: t squares to v and v squares to 0,lem: extracting j from eqv cochain}. Taking its hat-flavored truncation, we obtain an almost $j$-complex
\[
(C^\ast_{\Z_2}(SWF_R^n(Y,\phi);\mathbb{F}_2),\widehat{j^\ast}).
\]

\begin{defn}
The \emph{$n$th almost $j$-complex} of $(Y,\phi)$, denoted by $\mathcal{C}_R^n(Y,\phi)$, is the homotopy equivalence class of
$(C^\ast_{\Z_2}(SWF_R^n(Y,\phi);\mathbb{F}_2),\widehat{j^\ast})$.
\end{defn}

We next define the almost local $j$-equivalence group. The group axioms can be verified by adapting the corresponding arguments in \cite{dai2023infinite}.

\begin{defn}\label{defn: first phi v defn}
We say that a dg $\mathbb{F}_2[v]$-module $M$ is \emph{clean} if it is semifree and admits a homogeneous free basis for which the degree-$0$ dg endomorphism
\[
\Phi_v\colon M\longrightarrow M;
\qquad
x\longmapsto d_M\left(\frac{\partial x}{\partial v}\right)
+\frac{\partial(d_Mx)}{\partial v}
\]
satisfies $\Phi_v^2=0$. Here $d_M$ denotes the differential of $M$, and differentiation is taken coefficientwise with respect to the chosen basis. We require $\Phi_v^2$ to vanish, not merely to be nullhomotopic.
\end{defn}

\begin{lem}\label{lem: clean resolution}
Any perfect dg $\mathbb{F}_2[v]$-module is homotopy equivalent to a clean module.
\end{lem}

\begin{proof}
Since $\mathbb{F}_2[v]$ is a PID and $\deg(v)>0$, any perfect dg $\mathbb{F}_2[v]$-module is homotopy equivalent to a finite direct sum of grading shifts of free summands with zero differential and elementary complexes of the form
\[
E_n\coloneqq
\left[\mathbb{F}_2[v]\overset{v^n}{\longrightarrow}\mathbb{F}_2[v]\right],
\qquad n>0,
\]
where the generators of the source and target have degrees $n-1$ and $0$, respectively. The standard basis makes each $E_n$ clean, and free summands are also clean. Since cleanness is preserved under grading shifts and direct sums, the lemma follows.
\end{proof}

Given dg $\mathbb{F}_2[v]$-modules $M$ and $N$, we call $N$ a \emph{clean resolution} of $M$ if $N$ is clean and $M\simeq N$. By \Cref{lem: clean resolution}, every perfect dg $\mathbb{F}_2[v]$-module admits a clean resolution. When convenient, we replace a perfect dg $\mathbb{F}_2[v]$-module by a clean resolution.

\begin{defn}\label{defn: almost local j eqv group}
We define the \emph{almost local $j$-equivalence group} $\hat{\mathfrak{J}}$ to be the set
\[
\hat{\mathfrak{J}}\coloneqq
\{\text{almost }j\text{-complexes}\}/(\text{almost }j\text{-local equivalence}),
\]
equipped with the tensor product operation
\[
(C,j)\otimes(C',j')
\coloneqq
((\wt{C}\otimes_{\mathbb{F}_2[v]}\wt{C}',d_\otimes),j\otimes j').
\]
Here $\otimes_{\mathbb{F}_2[v]}$ denotes the ordinary tensor product, and $\wt{C}$ and $\wt{C}'$ are clean resolutions of $C$ and $C'$, respectively, with chosen bases witnessing cleanness. We transport $j$ and $j'$ to the hat-flavored truncations of these resolutions and retain the same notation. The differential is
\[
d_\otimes
\coloneqq
d_{\wt{C}}\otimes\id_{\wt{C}'}
+\id_{\wt{C}}\otimes d_{\wt{C}'}
+v\Phi_v\otimes\Phi_v.
\]
Since $\wt{C}$ and $\wt{C}'$ are clean, we have $d_\otimes^2=0$. Well-definedness is verified below. Associativity is proved in \Cref{lem: associativity}, and symmetry of the formula gives commutativity. The identity and inverses are described below.
\end{defn}

We now verify that the homotopy equivalence class of $(C,j)\otimes(C',j')$ is independent of the clean resolutions and their chosen bases. Let $C,C',D$ be clean dg $\mathbb{F}_2[v]$-modules, and let $f\colon C\to C'$ be a dg module homomorphism. Differentiating the chain map condition $d_{C'}f+fd_C=0$ with respect to $v$ gives
\[
\Phi_v f+f\Phi_v
=
d_{C'}\frac{\partial f}{\partial v}
+\frac{\partial f}{\partial v}d_C,
\]
where
\[
\frac{\partial f}{\partial v}(x)
\coloneqq
f\left(\frac{\partial x}{\partial v}\right)
+\frac{\partial(f(x))}{\partial v}.
\]
Here each occurrence of $\Phi_v$ denotes the endomorphism of the corresponding module. Define
\[
F\coloneqq
f\otimes\id_D
+v\frac{\partial f}{\partial v}\otimes\Phi_v.
\]
Let $d_\otimes$ and $d'_\otimes$ denote the differentials on the tensor products with underlying modules $C\otimes_{\mathbb{F}_2[v]}D$ and $C'\otimes_{\mathbb{F}_2[v]}D$, respectively. Since $\Phi_v$ commutes with the differentials and $\Phi_v^2=0$, we obtain
\[
\begin{aligned}
d'_\otimes F+Fd_\otimes
&=
(d_{C'}\otimes\id+\id\otimes d_D+v\Phi_v\otimes\Phi_v)
\left(f\otimes\id_D+v\frac{\partial f}{\partial v}\otimes\Phi_v\right)\\
&\quad+
\left(f\otimes\id_D+v\frac{\partial f}{\partial v}\otimes\Phi_v\right)
(d_C\otimes\id+\id\otimes d_D+v\Phi_v\otimes\Phi_v)\\
&=
v\left(
\Phi_v f+f\Phi_v
+d_{C'}\frac{\partial f}{\partial v}
+\frac{\partial f}{\partial v}d_C
\right)\otimes\Phi_v\\
&=0.
\end{aligned}
\]
Thus $F$ is a chain map. The construction respects composition and chain homotopies, and the same construction applies after inverting $v$. Moreover,
\[
\hat{F}=\hat{f}\otimes\id_{\hat{D}}.
\]
Consequently, the map
\[
F\colon
(C,j_C)\otimes(D,j_D)
\longrightarrow
(C',j_{C'})\otimes(D,j_D)
\]
has the following properties:
\begin{itemize}
\item $F$ is an isomorphism if $f$ is an isomorphism;
\item $F$ is a homotopy equivalence if $f$ is a homotopy equivalence;
\item $\hat{F}\circ(j_C\otimes j_D)\sim(j_{C'}\otimes j_D)\circ\hat{F}$ if $\hat{f}\circ j_C\sim j_{C'}\circ\hat{f}$;
\item $F$ is local if $f$ is local.
\end{itemize}
The same properties hold for
\[
F'\coloneqq
\id_D\otimes f
+v\Phi_v\otimes\frac{\partial f}{\partial v}.
\]
We have therefore proved the following proposition.

\begin{prop}\label{prop: invariance and partial order}
The homotopy equivalence class and the almost $j$-local equivalence class of $(C,j)\otimes(C',j')$ depend only on the corresponding classes of $(C,j)$ and $(C',j')$. Furthermore, if there exists an almost $j$-local map from $(C,j)$ to $(C',j')$, then, for every almost $j$-complex $(D,j_D)$, there exist almost $j$-local maps
\[
(C,j)\otimes(D,j_D)\longrightarrow(C',j')\otimes(D,j_D)
\qquad\text{ and }\qquad
(D,j_D)\otimes(C,j)\longrightarrow(D,j_D)\otimes(C',j').
\]
\end{prop}

\begin{proof}
This follows from the preceding argument.
\end{proof}

The identity element of $\hat{\mathfrak{J}}$ is represented by $(\mathbb{F}_2[v],\id)$, and the inverse of $[(C,j)]$ is represented by $(C^\vee,(j^3)^\vee)$, where
\[
C^\vee\coloneqq\operatorname{RHom}_{\mathbb{F}_2[v]}(C,\mathbb{F}_2[v]).
\]
Here the dual action is viewed on $\widehat{C^\vee}$ via its natural identification with the $\mathbb{F}_2$-linear dual of $\hat C$.

\begin{rem}
By \Cref{lem: phi v over F2 v is 1-torsion}, we have $v\Phi_v\sim0$ for every perfect dg $\mathbb{F}_2[v]$-module. Consequently, the correction term $v\Phi_v\otimes\Phi_v$ does not change the underlying dg-module homotopy type of the tensor product.

To see this, replace $C$ and $C'$ by clean resolutions and retain the same notation. Choose a null-homotopy $H$ of $v\Phi_v$ on $C$, so that $d_CH+Hd_C=v\Phi_v$, and define
\[
F\coloneqq\id_C\otimes\id_{C'}+H\otimes\Phi_v.
\]
Since $\Phi_v^2=0$ on $C'$, we have $F^2=\id$. Moreover,
\[
\begin{aligned}
(d_C\otimes\id+\id\otimes d_{C'})F+Fd_\otimes
&=
(d_C\otimes\id+\id\otimes d_{C'})(\id\otimes\id+H\otimes\Phi_v)\\
&\quad+
(\id\otimes\id+H\otimes\Phi_v)
(d_C\otimes\id+\id\otimes d_{C'}+v\Phi_v\otimes\Phi_v)\\
&=
(d_CH+Hd_C)\otimes\Phi_v
+H\otimes(d_{C'}\Phi_v+\Phi_vd_{C'})
+v\Phi_v\otimes\Phi_v\\
&=0,
\end{aligned}
\]
where we have again used $\Phi_v^2=0$. Thus $F$ is a chain isomorphism, and
\[
\begin{aligned}
(C\otimes_{\mathbb{F}_2[v]}C',d_\otimes)
&\cong
(C\otimes_{\mathbb{F}_2[v]}C',
d_C\otimes\id_{C'}+\id_C\otimes d_{C'})\\
&\simeq C\otimes^L_{\mathbb{F}_2[v]}C'.
\end{aligned}
\]
In other words, the underlying dg $\mathbb{F}_2[v]$-module of a tensor product of two almost $j$-complexes is homotopy equivalent to the derived tensor product of their underlying dg $\mathbb{F}_2[v]$-modules.
\end{rem}

We now verify associativity of the tensor product defined in \Cref{defn: almost local j eqv group}.

\begin{lem}\label{lem: associativity}
For any three almost $j$-complexes $(C_i,j_i)$, $i=1,2,3$, there exists a homotopy equivalence
\[
\bigl((C_1,j_1)\otimes(C_2,j_2)\bigr)\otimes(C_3,j_3)
\simeq
(C_1,j_1)\otimes\bigl((C_2,j_2)\otimes(C_3,j_3)\bigr).
\]
\end{lem}

\begin{proof}
By \Cref{prop: invariance and partial order}, we may assume that $C_1,C_2,C_3$ are clean. Since the second formal derivative with respect to $v$ vanishes in characteristic $2$, the $\Phi_v$-operator on a tensor product is
\[
\Phi_v\otimes\id+\id\otimes\Phi_v+\Phi_v\otimes\Phi_v.
\]
Its square is zero. Together with the preceding remark, this shows that the intermediate tensor products are again clean.

Under the canonical identification of the underlying graded modules, it remains to check that the differentials coincide: the endomorphism on the hat-flavored truncation is $j_1\otimes j_2\otimes j_3$ on both sides. Write $d_i$ for the differential of $C_i$. The differential $d_\otimes$ on the left-hand side is
\[
\begin{aligned}
d_\otimes
&=
(d_1\otimes\id+\id\otimes d_2+v\Phi_v\otimes\Phi_v)\otimes\id
+\id\otimes\id\otimes d_3\\
&\quad+
v(\Phi_v\otimes\id+\id\otimes\Phi_v+\Phi_v\otimes\Phi_v)\otimes\Phi_v\\
&=
d_1\otimes\id\otimes\id
+\id\otimes d_2\otimes\id
+\id\otimes\id\otimes d_3\\
&\quad+
v(\Phi_v\otimes\Phi_v\otimes\id
+\Phi_v\otimes\id\otimes\Phi_v
+\id\otimes\Phi_v\otimes\Phi_v
+\Phi_v\otimes\Phi_v\otimes\Phi_v).
\end{aligned}
\]
Similarly, the differential $d'_\otimes$ on the right-hand side is
\[
\begin{aligned}
d'_\otimes
&=
d_1\otimes\id\otimes\id
+\id\otimes(d_2\otimes\id+\id\otimes d_3+v\Phi_v\otimes\Phi_v)\\
&\quad+
v\Phi_v\otimes(\Phi_v\otimes\id+\id\otimes\Phi_v+\Phi_v\otimes\Phi_v)\\
&=
d_1\otimes\id\otimes\id
+\id\otimes d_2\otimes\id
+\id\otimes\id\otimes d_3\\
&\quad+
v(\Phi_v\otimes\Phi_v\otimes\id
+\Phi_v\otimes\id\otimes\Phi_v
+\id\otimes\Phi_v\otimes\Phi_v
+\Phi_v\otimes\Phi_v\otimes\Phi_v).
\end{aligned}
\]
Thus $d_\otimes=d'_\otimes$, proving the lemma.
\end{proof}

We also need an analogue of the connected-complex construction underlying \emph{connected Heegaard Floer homology} in \cite{hendricks2021applications}. Given an almost $j$-complex $(C,j)$, consider the set
\[
\mathrm{ALoc}(C,j)\coloneqq
\{\text{almost }j\text{-local maps }C\to C\},
\]
which is nonempty since $\id_C\in\mathrm{ALoc}(C,j)$. Choose $f\in\mathrm{ALoc}(C,j)$ such that, whenever $g\in\mathrm{ALoc}(C,j)$ satisfies $\ker f\subseteq\ker g$, we have $\ker f=\ker g$. Such an $f$ exists: with the finite free graded models fixed above, there are only finitely many grading-preserving $\mathbb{F}_2[v]$-module endomorphisms of $C$, so $\mathrm{ALoc}(C,j)$ is finite. We call such an $f$ a \emph{maximal almost $j$-local self-equivalence}. Adapting the proofs of \cite[Lemmas~3.4 and~3.5]{hendricks2021applications} gives the following:
\begin{itemize}
\item If $f\in\mathrm{ALoc}(C,j)$ is maximal, then $f|_{\operatorname{Im}f}$ is a chain isomorphism, and $f$ induces a splitting of the underlying chain complexes over $\mathbb{F}_2[v]$:
\[
C\cong\operatorname{Im}f\oplus\ker f.
\]
\item If $f,g\in\mathrm{ALoc}(C,j)$ are maximal, then $f|_{\operatorname{Im}g}\colon\operatorname{Im}g\to\operatorname{Im}f$ is a chain isomorphism.
\end{itemize}

For a maximal $f$, the splitting above induces natural identifications
\[
\operatorname{Im}\hat f\cong\widehat{\operatorname{Im}f},
\qquad
\ker\hat f\cong\widehat{\ker f}.
\]
Using these identifications and the natural inclusions into $\hat C$, define
\[
\begin{aligned}
j_f
&\coloneqq
\hat f\circ j\circ
(\hat f|_{\operatorname{Im}\hat f})^{-1}
\in\operatorname{End}(\widehat{\operatorname{Im}f}),\\
j_f^\perp
&\coloneqq
\bigl(\id_{\hat C}
+(\hat f|_{\operatorname{Im}\hat f})^{-1}\circ\hat f\bigr)
\circ j|_{\ker\hat f}
\in\operatorname{End}(\widehat{\ker f}).
\end{aligned}
\]
The relation $\hat f\circ j\sim j\circ\hat f$ implies that the off-diagonal components of $j$ with respect to the induced splitting of $\hat C$ are nullhomotopic. Consequently, the argument of \cite[Lemma~3.6]{hendricks2021applications} gives a homotopy equivalence
\[
(C,j)\simeq
(\operatorname{Im}f,j_f)\oplus(\ker f,j_f^\perp)
\]
of the underlying dg modules, intertwining the endomorphisms on their hat-flavored truncations up to homotopy. Since $\Phi_v$ commutes up to homotopy with dg $\mathbb{F}_2[v]$-module homomorphisms, we also have
\[
j_f^2\sim\id_{\widehat{\operatorname{Im}f}}+\hat{\Phi}_v,
\qquad
(j_f^\perp)^2\sim\id_{\widehat{\ker f}}+\hat{\Phi}_v,
\]
where $\Phi_v$ denotes the operator on the corresponding summand. Since $f$ is local, $\ker f$ becomes contractible after inverting $v$. Thus $(\operatorname{Im}f,j_f)$ is an almost $j$-complex, and the projection and inclusion give an almost $j$-local equivalence between $(C,j)$ and $(\operatorname{Im}f,j_f)$.

Similarly, adapting the proofs of \cite[Lemmas~3.7--3.8 and Proposition~3.10]{hendricks2021applications} shows that, if $(C,j)$ and $(C',j')$ are almost $j$-locally equivalent and $f\in\mathrm{ALoc}(C,j)$ and $g\in\mathrm{ALoc}(C',j')$ are maximal, then $(\operatorname{Im}f,j_f)$ and $(\operatorname{Im}g,j_g)$ are homotopy equivalent as almost $j$-complexes. This leads to the following definition, analogous to \cite[Definition~3.9]{hendricks2021applications}.

\begin{defn}
Given an almost $j$-local equivalence class $[(C,j)]\in\hat{\mathfrak{J}}$, choose a representative $(C,j)$ and a maximal $f\in\mathrm{ALoc}(C,j)$. We define its \emph{core} to be the homotopy equivalence class of the almost $j$-complex $(\operatorname{Im}f,j_f)$ and denote it by $\mathrm{Core}(C,j)$.
\end{defn}

We next relate the almost local $j$-equivalence group $\hat{\mathfrak{J}}$ to $\mathcal{LE}_{\Z_4}$.

\begin{lem}\label{lem: LE to J homomorphism}
The map
\[
\hat{\mathcal{J}}\colon\mathcal{LE}_{\Z_4}\longrightarrow\hat{\mathfrak{J}};
\qquad
[X]\longmapsto
[(C^\ast_{\langle j^2\rangle}(X;\mathbb{F}_2),\widehat{j^\ast})]
\]
is a well-defined group homomorphism. Consequently, the map
\[
\mathcal{C}_R^\infty\colon
\Omega_{\mathrm{spin}}(S^1\times S^2)
\longrightarrow\wt\prod_{n>0}\hat{\mathfrak{J}};
\qquad
[(Y,\phi)]\longmapsto
\left[([\mathcal{C}_R^n(Y,\phi)])_{n>0}\right]
\]
is also a well-defined group homomorphism.
\end{lem}

\begin{proof}
Since $X$ is finite, the Borel localization theorem implies that the restriction map
\[
C^\ast_{\langle j^2\rangle}(X;\mathbb{F}_2)
\otimes_{\mathbb{F}_2[v]}\mathbb{F}_2[v,v^{-1}]
\longrightarrow
C^\ast_{\langle j^2\rangle}(X^{\langle j^2\rangle};\mathbb{F}_2)
\otimes_{\mathbb{F}_2[v]}\mathbb{F}_2[v,v^{-1}]
\]
is a quasi-isomorphism. By naturality, a local map between SWF-like $\Z_4$-spectra induces an almost $j$-local map between their cochain models in the opposite direction. Thus locally equivalent spectra give almost $j$-locally equivalent complexes, proving that $\hat{\mathcal{J}}$ is well-defined. Its additivity follows from \Cref{cor: equivariant chain tensor product,prop: tensor product formula}.

Finally, $\mathcal{C}_R^\infty$ is the composition
\[
\Omega_{\mathrm{spin}}(S^1\times S^2)
\xrightarrow{\hspace{1.5em}SWF_R^\infty\hspace{1.5em}}
\wt{\prod}_{n>0}\mathcal{LE}_{\Z_4}
\xrightarrow{\hspace{1.5em}\wt{\prod}_{n>0}\hat{\mathcal{J}}\hspace{1.5em}}
\wt{\prod}_{n>0}\hat{\mathfrak{J}}.
\]
Since both maps are group homomorphisms, so is $\mathcal{C}_R^\infty$.
\end{proof}

The following lemma gives a criterion for constructing a subgroup isomorphic to $\Z^\infty$ in $\Omega_{\mathrm{spin}}(S^1\times S^2)$.

\begin{lem}\label{lem: how to find Z infty in Omega spin}
Let $\{(Y_k,\phi_k)\}_{k>0}$ be a sequence of distinguished homology handles. For every pair of positive integers $k,s$, suppose that there exists a nonnegative integer $C$ such that $v^C$ annihilates the $\mathbb{F}_2[v]$-torsion submodule of
\[
H^\ast\mathrm{Core}(\mathcal{C}_R^n(Y_k,\phi_k)^{\otimes s})
\]
for all sufficiently large integers $n>0$, and let $C_{k,s}$ be the least such integer. Suppose that $C_{k,s}$ is independent of $s$, and write $C_k\coloneqq C_{k,s}$. If $\limsup_{k\to\infty}C_k=\infty$, then the subgroup of $\Omega_{\mathrm{spin}}(S^1\times S^2)$ generated by the $\wt{H}_{\mathbb{F}_2}$-cobordism classes $[(Y_k,\phi_k)]$ contains a subgroup isomorphic to $\Z^\infty$.
\end{lem}

\begin{proof}
Since $\limsup_{k\to\infty}C_k=\infty$, after passing to a subsequence and relabeling, we may assume that
\[
0<C_1<C_2<\cdots.
\]
It suffices to show that the classes $[(Y_k,\phi_k)]$ are linearly independent.

Suppose otherwise. After cancelling common terms, there is a nontrivial relation
\[
\bigcirc_{a=1}^r[(Y_{i_a},\phi_{i_a})]
=
\bigcirc_{b=1}^s[(Y_{j_b},\phi_{j_b})],
\]
where the two sets of indices are disjoint. We allow an empty circle sum, which denotes the identity. After reordering the indices and interchanging the two sides if necessary, assume that $i_1\leq\cdots\leq i_r$ and that the largest index appearing in the relation is $m=i_r$. Let $t>0$ be its multiplicity, so that the last $t$ indices on the left equal $m$ and all remaining indices are smaller than $m$.

By \Cref{lem: local equivalence lemma,lem: additivity}, for all sufficiently large integers $n>0$, the spectra
\[
\bigwedge_{a=1}^r SWF_R^n(Y_{i_a},\phi_{i_a})
\qquad\text{ and }\qquad
\bigwedge_{b=1}^s SWF_R^n(Y_{j_b},\phi_{j_b})
\]
are locally equivalent. Passing to the associated almost $j$-complexes via \Cref{lem: LE to J homomorphism}, we obtain
\[
[\mathcal{C}_R^n(Y_m,\phi_m)^{\otimes t}]
=
\left[
\bigotimes_{b=1}^s\mathcal{C}_R^n(Y_{j_b},\phi_{j_b})
\otimes
\bigotimes_{a=1}^{r-t}\mathcal{C}_R^n(Y_{i_a},\phi_{i_a})^{-1}
\right]
\quad\text{in }\hat{\mathfrak{J}}.
\]

Set
\[
D\coloneqq
\max\bigl(
\{0\}\cup
\{C_{i_a}:1\leq a\leq r-t\}\cup
\{C_{j_b}:1\leq b\leq s\}
\bigr).
\]
Then $D<C_m$. Replace each non-inverse factor on the right by a finite free model of its core, and each inverse factor by the dual of a finite free model of the corresponding core. These replacements do not change the class in $\hat{\mathfrak{J}}$.

By the definition of the bounds $C_k$, for all sufficiently large $n$, multiplication by $v^D$ annihilates the torsion submodule of the cohomology of each of these cores. As shown above, the underlying dg module of the tensor product of almost $j$-complexes is homotopy equivalent to the derived tensor product of their underlying dg modules. By the graded PID decomposition, the underlying complexes of these cores are homotopy equivalent to direct sums of grading shifts of free summands with zero differential and elementary complexes whose differential is multiplication by $v^a$, with $a\leq D$. Duality preserves the exponent $a$, while tensoring elementary complexes with exponents $a$ and $b$ gives two shifted elementary complexes with exponent $\min\{a,b\}$. Thus $v^D$ annihilates the torsion submodule of the cohomology of the resulting tensor product. Since taking the core gives a chain direct summand, the same holds for the cohomology of its core.

The equality in $\hat{\mathfrak{J}}$ therefore implies that $v^D$ annihilates the torsion submodule of
\[
H^\ast\mathrm{Core}(\mathcal{C}_R^n(Y_m,\phi_m)^{\otimes t})
\]
for all sufficiently large $n$. By the minimality of $C_{m,t}$, we obtain
\[
C_m=C_{m,t}\leq D<C_m,
\]
a contradiction. Hence the chosen classes are linearly independent, proving the lemma.
\end{proof}

The same argument gives the following corollary.

\begin{cor}\label{cor: linear independence}
Under the assumptions of \Cref{lem: how to find Z infty in Omega spin}, if the sequence $\{C_k\}_{k>0}$ is positive and strictly increasing, then the $\wt{H}_{\mathbb{F}_2}$-cobordism classes $[(Y_k,\phi_k)]$ are linearly independent in $\Omega_{\mathrm{spin}}(S^1\times S^2)$.
\end{cor}

\begin{proof}
Apply the linear-independence argument in the proof of \Cref{lem: how to find Z infty in Omega spin} to the entire sequence.
\end{proof}

\subsection{SWF-like $\Z_2$-spectra and real Fr{\o}yshov invariants}

We review the definition and basic properties of the real Fr{\o}yshov invariant $\delta_R$, introduced in \cite{konno2024involutions}. Let $X$ be an SWF-like $\Z_2$-spectrum, so that $X$ is a finite $\Z_2$-spectrum and $X^{\Z_2}$ is stably homotopy equivalent to the sphere spectrum up to suspension. By the Borel localization theorem,
\[
H^\ast_{\Z_2}(X;\mathbb{F}_2)\otimes_{\mathbb{F}_2[v]}\mathbb{F}_2[v,v^{-1}]
\]
is free of rank $1$ over $\mathbb{F}_2[v,v^{-1}]$, where $H^\ast_{\Z_2}$ denotes reduced Borel cohomology. This leads to the following definition.

\begin{defn}
We say that a homogeneous element $\alpha\in H^\ast_{\Z_2}(X;\mathbb{F}_2)$ is a \emph{tower generator} if it is nontorsion, or equivalently, if it does not vanish after inverting $v$. We define the \emph{real Fr{\o}yshov invariant} of $X$ by
\[
\delta_R(X) \coloneqq \frac{1}{2}\min\left\{\deg\alpha\mid\alpha\in H^\ast_{\Z_2}(X;\mathbb{F}_2)\text{ is a tower generator}\right\}\in\Q.
\]
\end{defn}

Note that the following lemma was observed in \cite{konno2024involutions}. 

\begin{lem}[Basically {\cite[Lemma 3.8 and Lemma 3.10]{konno2024involutions}}]\label{lem: Froyshov of SWF-like Z2 spectra}
The map
\[
\delta_R\colon\mathcal{LE}_{\Z_2}\longrightarrow\Q;
\qquad [X]\longmapsto\delta_R(X)
\]
is a group homomorphism.
\end{lem}

\begin{comment}
\begin{proof}
Let $f\colon X\to Y$ be a local map between SWF-like $\Z_2$-spectra. The induced pullback
\[
f^\ast\colon H^\ast_{\Z_2}(Y;\mathbb{F}_2)\longrightarrow H^\ast_{\Z_2}(X;\mathbb{F}_2)
\]
is a degree-preserving $\mathbb{F}_2[v]$-module homomorphism that becomes an isomorphism after inverting $v$. It therefore sends tower generators to tower generators, so $\delta_R(X)\leq\delta_R(Y)$. If $X$ and $Y$ are locally equivalent, applying this argument in both directions gives $\delta_R(X)=\delta_R(Y)$. Thus $\delta_R$ is well-defined on $\mathcal{LE}_{\Z_2}$.

It remains to show that $\delta_R$ is additive. By \Cref{cor: equivariant chain tensor product} and the fact that $\mathbb{F}_2[v]$ is a PID (and thus all perfect dg modules over it are formal), we have
\[
H^\ast_{\Z_2}(X\wedge Y;\mathbb{F}_2)\simeq H^\ast\left( H^\ast_{\Z_2}(X;\mathbb{F}_2)\otimes^L_{\mathbb{F}_2[v]} H^\ast_{\Z_2}(Y;\mathbb{F}_2) \right).
\]
Applying the universal coefficient theorem then gives
\[
H^\ast_{\Z_2}(X\wedge Y;\mathbb{F}_2)/\mathrm{tor}\simeq \left[H^\ast_{\Z_2}(X;\mathbb{F}_2)/\mathrm{tor}\right]\otimes\left[H^\ast_{\Z_2}(Y;\mathbb{F}_2)/\mathrm{tor}\right].
\]
Therefore we deduce that $\delta_R(X\wedge Y)=\delta_R(X)+\delta_R(Y)$, as required.
\end{proof}
\end{comment}

Combining \Cref{lem: additivity,lem: Froyshov of SWF-like Z2 spectra}, we obtain a group homomorphism from $\Omega(S^1\times S^2)$ by taking the composition
\[
\Omega(S^1\times S^2)
\xrightarrow{\hspace{1.5em}SWF_{R,c}^\infty\hspace{1.5em}}
\wt{\prod}_{n>0}\mathcal{LE}_{\Z_2}
\xrightarrow{\hspace{1.5em}\wt{\prod}_{n>0}\delta_R\hspace{1.5em}}
\wt{\prod}_{n>0}\Q.
\]
This sends an $\wt{H}_{\Q}$-cobordism class $[(Y,\phi)]$ to the equivalence class of the sequence $\{\delta_R(SWF_{R,c}^n(Y,\phi))\}_{n>0}$.

\begin{defn}
We call the above homomorphism obtained as the composition the \emph{asymptotic real Fr{\o}yshov invariant} and denote it by $\delta_R^\infty$.
\end{defn}

\subsection{The connected sum formula and computations for $\wt{H}$-cobordisms}

We focus on the connected sum homomorphisms defined in \Cref{cor: connected sum homomorphism}:
\[
\Sigma_\Q\colon\Theta^3_\Z\longrightarrow\Omega(S^1\times S^2)
\qquad\text{ and }\qquad
\Sigma_{\mathbb{F}_2}\colon\Theta^3_\Z\longrightarrow\Omega_{\mathrm{spin}}(S^1\times S^2).
\]
We start by studying their relation with real Seiberg--Witten theory.

\begin{lem}\label{lem: SWF of connected sum map}
Let $Y$ be an oriented integral homology $3$-sphere, and consider the distinguished homology handle $(Y\# S^1\times S^2,\phi_Y)$, where $\phi_Y$ is a generator of $H^1(Y\# S^1\times S^2;\Z)$. Let $SWF(Y)$ denote the Seiberg--Witten Floer spectrum of $Y$ with respect to its unique spin structure, regarded as a finite $\mathrm{Pin}(2)$-spectrum. Then, for every integer $n>0$,
\[
SWF_R^n(Y\# S^1\times S^2,\phi_Y)
\quad\text{is locally equivalent to}\quad
SWF(Y)^{\wedge 2^{n-1}}.
\]
Here the generator of $\Z_4$ acts by $j^{\wedge 2^{n-1}}$, where $j\in\mathrm{Pin}(2)$ acts on each $SWF(Y)$ factor.
\end{lem}

\begin{proof}
For simplicity, denote the distinguished homology handle $(Y\# S^1\times S^2,\phi_Y)$ by $\Sigma(Y)$. For every integer $n>0$, we have
\[
\Sigma(Y)^{2^n}\cong Y^{\# 2^{n-1}}\# S^1\times S^2\# Y^{\# 2^{n-1}}.
\]
The deck involution of $\Sigma(Y)^{2^n}\to\Sigma(Y)^{2^{n-1}}$ rotates the $S^1$ factor and exchanges the two $Y^{\# 2^{n-1}}$ summands. Consider the local unknot
\[
(Y^{\# 2^{n-1}},U_n) \coloneqq Y^{\# 2^{n-1}}\#(S^3,U),
\]
where $U\subset S^3$ is the unknot. We denote by $(Y^{\# 2^{n-1}})_0(U_n)$ the manifold obtained by zero-framed surgery on $U_n$.

In the notation of \cite{miyazawa2025satellite}, the spectrum $SWF_R^n(\Sigma(Y))$ is denoted by $SWF_R((Y^{\# 2^{n-1}})_0(U_n))$. By \cite[Theorem~4.10]{miyazawa2025satellite}, this is locally equivalent to $SWF_R(U_n)$, defined using the double cover of $Y^{\# 2^{n-1}}$ branched along $U_n$, which is diffeomorphic to $Y^{\# 2^n}$. Hence
\[
[SWF_R^n(\Sigma(Y))]=[SWF_R(Y^{\# 2^n})]\in\mathcal{LE}_{\Z_4},
\]
where $Y^{\# 2^n}$ is equipped with the branched-covering involution exchanging the two $Y^{\# 2^{n-1}}$ summands.

To compute $SWF_R(Y^{\# 2^n})$, set $A \coloneqq SWF(Y)^{\wedge 2^{n-1}}$. By the $\mathrm{Pin}(2)$-equivariant connected sum formula \cite[Theorem~3.19]{Dai-Sasahira-Stoffregen:2026}, we have
\[
SWF(Y^{\# 2^n})\simeq A\wedge A.
\]
The equivalence is induced by the Bauer--Furuta map of the standard $1$-handle cobordism and, by naturality, may be chosen equivariantly with respect to the involution exchanging the two factors. After fixing the sign of its odd lift, its action on a pointed model $X\wedge X$ for $A\wedge A$ is given by
\[
\widetilde{\tau}(x\wedge y)=y\wedge(-x),
\]
where $-x$ denotes the action of $-1\in S^1\subset\mathrm{Pin}(2)$ on $x$. Then $\widetilde{\tau}^2=-1$, as required. The real involution is $I=j\widetilde{\tau}$, so
\[
I(x\wedge y)=jy\wedge(-jx).
\]
It follows that the map
\[
\Delta_{-j}\colon X\longrightarrow(X\wedge X)^I;
\qquad x\longmapsto x\wedge(-jx)
\]
is an identification. Moreover, $j\Delta_{-j}(x)=jx\wedge x=\Delta_{-j}(jx)$, so this identification intertwines the $\Z_4$-action on $(X\wedge X)^I$ induced by $j$ with the diagonal $j$-action on $A$. The same fixed-point identification for the representation suspensions gives precisely the grading of $A$. Hence, by the fixed-point construction of the real Floer spectrum,
\[
SWF_R(Y^{\# 2^n})
\simeq SWF(Y^{\# 2^n})^I
\simeq SWF(Y)^{\wedge 2^{n-1}}
\]
as $\Z_4$-spectra. The lemma follows.
\end{proof}

For an oriented integral homology $3$-sphere $Y$, let $\delta(Y)$ denote the monopole Fr{\o}yshov invariant \cite{Froyshov2010MonopoleFloer}. It is one-half the least degree of a homogeneous nontorsion class in the reduced $S^1$-equivariant cohomology of $SWF(Y)$, viewed as a module over $\mathbb{F}_2[U]$ with $\deg U=2$. We use the convention $\delta(\Sigma(2,3,5))=1$, where $\Sigma(2,3,5)$ has the boundary orientation of the negative-definite $E_8$ plumbing.

\begin{lem}\label{lem: real Froyshov of connected sum map}
Let $Y$ be an oriented integral homology $3$-sphere. Then
\[
\delta_R^\infty(\Sigma_\Q(Y))
=[\{2^{n-1}\delta(Y)\}_{n>0}]
\in\wt\prod_{n>0}\Q.
\]
\end{lem}

\begin{proof}
By \Cref{lem: SWF of connected sum map},
\[
SWF_{R,c}^n(Y\# S^1\times S^2,\phi_Y)
\quad\text{is locally equivalent to}\quad
SWF(Y)^{\wedge 2^{n-1}},
\]
where the generator of $\Z_2$ acts by $-1$ on each $SWF(Y)$ factor. This action is the restriction of the diagonal $S^1$-action to $\Z_2\subset S^1$.

The induced map
\[
H^\ast(BS^1;\mathbb{F}_2)\longrightarrow H^\ast(B\Z_2;\mathbb{F}_2)
\]
sends $U$ to $v^2$. Since $\mathbb{F}_2[v]$ is free over $\mathbb{F}_2[U]$ under this map, the standard change-of-groups formula shows that restriction preserves the lowest degree of a tower generator. Thus $\delta_R(SWF(Y))=\delta(Y)$. By \Cref{lem: Froyshov of SWF-like Z2 spectra}, we obtain
\[
\delta_R(SWF_{R,c}^n(Y\# S^1\times S^2,\phi_Y))
=2^{n-1}\delta_R(SWF(Y))
=2^{n-1}\delta(Y).
\]
The lemma follows.
\end{proof}

We are now ready to prove \Cref{thm:main,cor: nonextendable cyclic action,cor:main}, whose statements we recall for convenience.

\mainkerneltheorem*

\begin{proof}
Let $Y=\Sigma(2,3,5)$ be the Poincar\'e homology $3$-sphere, with the orientation fixed above. By \Cref{cor: connected sum homomorphism},
\[
\Sigma_\Q(Y)\in\ker(\Omega(S^1\times S^2)\longrightarrow\Omega^{\mathrm{top}}(S^1\times S^2)).
\]
Since $\delta(Y)=1$, \Cref{lem: real Froyshov of connected sum map} gives
\[
\delta_R^\infty(\Sigma_\Q(Y))
=[\{2^{n-1}\delta(Y)\}_{n>0}]
=[\{2^{n-1}\}_{n>0}]
\in\wt\prod_{n>0}\Q.
\]
This class has infinite order, since no nonzero integer multiple of the sequence is eventually zero. Since $\delta_R^\infty$ is a group homomorphism, $\Sigma_\Q(Y)$ also has infinite order. The theorem follows.
\end{proof}

\nonextendablecor*

\begin{proof}
Attach a $1$-handle to a contractible topological filling of $\Sigma(2,3,5)$ provided by \cite[Theorem~1.4$^\prime$]{Freedman:1982}. The resulting manifold fills $N=\Sigma(2,3,5)\#(S^1\times S^2)$ and is homotopy equivalent to $S^1$. Its universal cover is therefore contractible and has boundary $M$ with the prescribed deck action. Thus it gives the required contractible free $\Z$-equivariant topological filling of $M$. By \cite[Theorem~8.2]{FQ90}, this universal cover admits a smooth structure extending the prescribed boundary smoothing, giving a nonequivariant smooth contractible filling of $M$.

Suppose that $M$ admitted a free $\Z$-equivariant smooth filling $D$ with finite-dimensional total rational homology. The quotient $D/\Z$ is a compact smooth filling of $N$. Removing the interior of a tubular neighborhood of an embedded interior loop on which the covering class evaluates to $1$ gives a distinguished cobordism between $N$ and $S^1\times S^2$. Its infinite cyclic cover is obtained from $D$ by removing the interior of the lifted tubular neighborhood $\R\times B^3$. Since the common boundary is $\R\times S^2$, the Mayer--Vietoris sequence shows that this cover still has finite-dimensional total rational homology. The resulting cobordism is therefore a smooth $\wt{H}_{\Q}$-cobordism, contradicting \Cref{thm:main}.
\end{proof}

\Fkerneltheorem*

\begin{proof}
Consider the example $Y=\Sigma(2,3,5)$ used in the proof of \Cref{thm:main}. By \Cref{cor: connected sum homomorphism}, $\Sigma_F(Y)$ lies in the kernel of the natural map $\Omega_F(S^1\times S^2)\to\Omega_F^{\mathrm{top}}(S^1\times S^2)$.

We claim that $\Sigma_F(Y)$ has infinite order. Otherwise, there exists an integer $n>0$ and a smooth $\wt{H}_F$-cobordism $W$ from $Y^{\# n}\#(S^1\times S^2)$ to $S^1\times S^2$. By \Cref{lem: F-kawauchi implies Q-kawauchi}, $W$ is also an $\wt{H}_{\Q}$-cobordism. This implies that $n\Sigma_\Q(Y)=0$ in $\Omega(S^1\times S^2)$, contradicting the infinite order of $\Sigma_\Q(Y)$ established in the proof of \Cref{thm:main}. Thus $\Sigma_F(Y)$ generates the required infinite cyclic subgroup.
\end{proof}

The following corollary is a refinement of \Cref{cor: nonextendable cyclic action}.

\begin{cor}\label{cor: cyclic cover Betti growth}
Let $M$ be the $\Z$-periodic homology $\R\times S^2$ constructed in the proof of \Cref{cor: nonextendable cyclic action}. Every free $\Z$-equivariant smooth filling $W$ of $M$ satisfies
\[
b_2(W/d\Z;\Q)\geq d
\qquad\text{for every integer }d>0,
\]
where $W/d\Z$ denotes the quotient by the subgroup $d\Z\subset\Z$.
This bound is sharp: there exists such a filling $W_0$ with $b_2(W_0/d\Z;\Q)=d$ for every integer $d>0$. In contrast, $M$ admits a free $\Z$-equivariant topological filling $V$ satisfying
\[
b_2(V/d\Z;\Q)=0
\qquad\text{for every integer }d>0.
\]
\end{cor}

\begin{proof}
Set $\overline W\coloneqq W/\Z$ and $K\coloneqq\Q(t)$, with local coefficients determined by the covering class. We first claim that $H_2(\overline W;K)\neq0$. Suppose otherwise. Perform interior surgeries on embedded loops in the kernel of the covering class representing nonzero classes in $H_1(-;K)$. Since the boundary is $K$-acyclic, Poincar\'e--Lefschetz duality and the surgery exact sequences show that each such surgery decreases the dimension of $H_1(-;K)$ by one without changing $H_2(-;K)$. These surgeries preserve the boundary and extend the covering class. After finitely many surgeries, $H_1(-;K)$ vanishes, and duality implies that the resulting filling is $K$-acyclic. Its infinite cyclic cover therefore has finite-dimensional total rational homology, contradicting \Cref{cor: nonextendable cyclic action}.

Write $\rho\coloneqq\dim_K H_2(\overline W;K)\geq1$. Since $H_2(W;\Q)$ has rank $\rho$ over $\Q[t,t^{-1}]$, the universal coefficient theorem for cyclic covers gives
\[
b_2(W/d\Z;\Q)\geq d\rho\geq d.
\]

For sharpness, the surgery trace of $(-1)$-surgery on the left-handed trefoil gives a smooth filling of $\Sigma(2,3,5)$ homotopy equivalent to $S^2$. Attaching a $1$-handle gives a filling of $N$ homotopy equivalent to $S^1\vee S^2$. Its infinite cyclic cover $W_0$ is a free $\Z$-equivariant smooth filling of $M$, and each quotient $W_0/d\Z$ has second Betti number $d$.

Finally, take the contractible free $\Z$-equivariant topological filling $V$ constructed in the proof of \Cref{cor: nonextendable cyclic action}. Each quotient $V/d\Z$ is homotopy equivalent to $S^1$, so its second Betti number is zero.
\end{proof}

\subsection{Computations for $\wt{H}_{\mathbb{F}_2}$-cobordisms}

Finally, we study the map $\Sigma_{\mathbb{F}_2}\colon\Theta^3_\Z\to\Omega_{\mathrm{spin}}(S^1\times S^2)$.

\begin{defn}
For any positive even integer $n$, let $P_n$ be the almost $j$-complex whose underlying dg $\mathbb{F}_2[v]$-module is freely generated by $x,y,z$ of degrees $0,0,n-1$, respectively, with differential and $j$-action given by
\[
dx=dy=0,\quad dz=v^n x+v^n y,\quad j(x)=y,\quad j(y)=x,\quad j(z)=z.
\]
Here, in the formulas for $j$, we use the same symbols for the images of $x,y,z$ in $\hat P_n$.

We say that an almost $j$-complex is \emph{locally $n$-simple} if it is almost $j$-locally equivalent to $P_n$ up to a degree shift.
\end{defn}

\begin{lem}\label{lem: powers of locally simple complexes}
Let $(C,j)$ be a locally $n$-simple almost $j$-complex. Then, for every integer $k>0$, the torsion submodule of $H^\ast\mathrm{Core}((C,j)^{\otimes k})$ is a nonzero direct sum of degree-shifted copies of $\mathbb{F}_2[v]/(v^n)$.
\end{lem}

\begin{proof}
Since degree shifts do not affect the assertion and cores are invariant under almost $j$-local equivalence, it suffices to consider $(C,j)=P_n$. The core of $P_n^{\otimes k}$ is represented by a chain direct summand of $P_n^{\otimes k}$. As a dg $\mathbb{F}_2[v]$-module, $P_n$ splits as a free summand generated by $x$ and a two-generator complex on $x+y$ and $z$, whose differential is multiplication by $v^n$. Tensoring these elementary complexes shows that the torsion submodule of $H^\ast(P_n^{\otimes k})$ is a direct sum of degree-shifted copies of $\mathbb{F}_2[v]/(v^n)$, while its torsion-free quotient is a single copy of $\mathbb{F}_2[v]$ generated in degree zero. The same description of the torsion summands therefore holds for $H^\ast\mathrm{Core}(P_n^{\otimes k})$. It remains to show that this torsion submodule is nonzero.

Define a partial order on $\hat{\mathfrak{J}}$ by declaring that $[(C,j)]\leq[(C',j')]$ if there exists an almost $j$-local map from $(C,j)$ to $(C',j')$. This order is well-defined and translation-invariant by \Cref{prop: invariance and partial order}. Let $P_0$ be the free rank-one $\mathbb{F}_2[v]$-module generated by $w$ in degree zero, with zero differential and $j=\id$, so that $[P_0]=0$ in $\hat{\mathfrak{J}}$. The chain map
\[
P_n\longrightarrow P_0;\qquad x,y\longmapsto w,\quad z\longmapsto 0
\]
is almost $j$-local, so $[P_n]\leq 0$.

Conversely, suppose that $f\colon P_0\to P_n$ is an almost $j$-local map. Since $P_n$ has no elements of negative degree, the homotopy between $j\circ \hat{f}$ and $\hat{f}$ vanishes on $w$. Thus $f(w)$ is a $j$-invariant cycle of degree zero, so $\hat{f}(w)\in\{0,x+y\}$, which then implies $f(w)\in \{0,x+y\}$ as well. Both choices represent torsion classes, contradicting the locality of $f$. Hence $0\not\leq[P_n]$, and therefore $[P_n]<0$. Translation invariance gives
\[
0>[P_n]>2[P_n]>\cdots>k[P_n]
\]
for every integer $k>0$. In particular, $[P_n]$ has infinite order.

If $H^\ast\mathrm{Core}(P_n^{\otimes k})$ were torsion-free, it would be a single copy of $\mathbb{F}_2[v]$ generated in degree zero, with the identity induced $j$-action. The core would then be homotopy equivalent to $P_0$ as an almost $j$-complex, giving $k[P_n]=0$, a contradiction. Thus its torsion submodule is nonzero, completing the proof.
\end{proof}

We are now ready to prove \Cref{thm: main2}, whose statement we recall for convenience.

\mainspinkerneltheorem*

\begin{proof}
For each integer $k>0$, consider the oriented Brieskorn sphere
\[
Y_k=-\Sigma(2k+1,4k+1,4k+3),
\]
and set
\[
A_k \coloneqq (C^\ast_{\Z_2}(SWF(Y_k);\mathbb{F}_2),j^\ast).
\]
The computation in \cite[Theorems~1.8 and~5.6]{Sto172} (see \cite[Section 5.2]{Stoffregen20} for a more detailed discussion), together with orientation-reversal duality, shows that $A_k$ is locally $2k$-simple.

By \Cref{lem: SWF of connected sum map} and the smash-to-tensor comparison, for all integers $n,s>0$, we have
\[
H^\ast\mathrm{Core}(\mathcal{C}_R^n(Y_k\# S^1\times S^2,\phi_{Y_k})^{\otimes s})
\cong H^\ast\mathrm{Core}(A_k^{\otimes(2^{n-1}s)}).
\]
By \Cref{lem: powers of locally simple complexes}, the torsion submodule of this cohomology is nonzero, and its least annihilating power of $v$ is $v^{2k}$. Thus the bounds in \Cref{lem: how to find Z infty in Omega spin} satisfy $C_{k,s}=2k$, which is positive, strictly increasing in $k$, and independent of $s$.

By \Cref{cor: linear independence}, the elements $\Sigma_{\mathbb{F}_2}(Y_k)\in\Omega_{\mathrm{spin}}(S^1\times S^2)$ are linearly independent. By \Cref{cor: connected sum homomorphism}, these elements lie in the kernel of the smooth-to-topological homomorphism, proving the theorem.
\end{proof}

\appendix

\section{Parametrized K\"{u}nneth formula and $\Z_4$-equivariant cochain complexes}

This appendix provides the algebraic input for the passage from $\Z_4$-equivariant spectra to almost $j$-complexes. We first establish a parametrized K\"{u}nneth formula and then describe the associated almost $j$-structures and their behavior under tensor products.

\subsection{A parametrized K\"{u}nneth formula over a connected base} \label{subsec: appendix kunneth}

We prove a K\"{u}nneth formula for parametrized spectra under suitable finiteness hypotheses. For completeness and clarity, we include precise statements and proofs, together with relevant references. We will be following the arguments used in \cite{RamziEMSS} extremely closely, where a version of \Cref{prop: tensor product for connected base} is proven for pullback squares of spaces, under a very similar set of conditions on the base and the monodromy action.

\begin{prop}\label{prop: tensor product for connected base}
Let $B$ be a connected CW complex, let $S$ be a commutative Noetherian ring of finite global dimension, let $R$ be a commutative $S$-algebra, and let $\mathcal{X},\mathcal{Y}$ be parametrized spectra over $B$. Suppose that the fiber $\mathcal{X}_p$ of $\mathcal{X}$ at $p$ is a finite spectrum for every $p\in B$. Fix a basepoint $b\in B$ and suppose that the monodromy representation
\[
\rho_b\colon\pi_1(B,b)\longrightarrow
\mathrm{Aut}_S(H^\ast(\mathcal{X}_b;S))
\]
is unipotent, meaning that $H^\ast(\mathcal{X}_b;S)$ admits a $\pi_1(B,b)$-invariant filtration by graded $S$-submodules
\[
0=F_0\subset\cdots\subset F_n=H^\ast(\mathcal{X}_b;S)
\]
such that the action on $F_k/F_{k-1}$ is trivial for each $k=1,\ldots,n$. Then the natural map
\[
C^\ast(\mathcal{X};R)\otimes^L_{C^\ast(B;R)}
C^\ast(\mathcal{Y};R)
\longrightarrow
C^\ast(\mathcal{X}\wedge_B\mathcal{Y};R)
\]
is a quasi-isomorphism of $C^\ast(B;R)$-modules.
\end{prop}

\begin{proof}
We use the following notation throughout the proof. Let
$\mathrm{Sp}_B\coloneqq\mathrm{Fun}(B,\mathrm{Sp})$
denote the $\infty$-category of parametrized spectra over $B$. For a spectrum $\mathcal{S}$, let $\underline{\mathcal{S}}_B$ denote the corresponding constant family over $B$. We also write
$\mathcal{L}_R(B)\coloneqq\mathrm{Fun}(B,\mathrm{Mod}_{HR})$,
which can be viewed as the $\infty$-category of parametrized $HR$-modules over $B$. Tensor products (and in general, any natural operations) in these $\infty$-categories are understood to be derived.

For a parametrized spectrum $\mathcal{Z}$ and a commutative ring $T$, set
\[
\mathcal{M}_T(\mathcal{Z})
\coloneqq F_B(\mathcal{Z},\underline{HT}_B).
\]
The monodromy action on
\[
\pi_{-\ast}(\mathcal{M}_S(\mathcal{X})_b)
\cong H^\ast(\mathcal{X}_b;S)
\]
is given by $\rho_b$. Since $\mathcal{X}_b$ is finite, its $S$-cohomology is bounded and finitely generated. Moreover, since $S$ is Noetherian and has finite global dimension, each successive quotient in the given filtration admits a finite resolution by finitely generated projective $S$-modules. The finite Postnikov tower of $\mathcal{M}_S(\mathcal{X})$ therefore shows that it lies in the thick subcategory of $\mathcal{L}_S(B)$ generated by the unit $\underline{HS}_B$.

By \cite[Lemma~4.2]{ABG18}, the finiteness of the fibers implies that $\mathcal{X}$ is dualizable in $\mathrm{Sp}_B$. Hence its dual
$\mathcal{X}^\vee\coloneqq F_B(\mathcal{X},\underline{\mathbb{S}}_B)$
satisfies
\[
F_B(\mathcal{X},\mathcal{Z})
\simeq\mathcal{X}^\vee\wedge_B\mathcal{Z}
\]
naturally in $\mathcal{Z}$. In particular,
\[
\begin{aligned}
\mathcal{M}_S(\mathcal{X})
\otimes_{\underline{HS}_B}\underline{HR}_B
&\simeq
(\mathcal{X}^\vee\wedge_B\underline{HS}_B)
\otimes_{\underline{HS}_B}\underline{HR}_B\\
&\simeq\mathcal{X}^\vee\wedge_B\underline{HR}_B\\
&\simeq\mathcal{M}_R(\mathcal{X}).
\end{aligned}
\]
Since base change preserves finite colimits and retracts, it follows that $\mathcal{M}_R(\mathcal{X})$ lies in the thick subcategory of $\mathcal{L}_R(B)$ generated by the unit $\underline{HR}_B$.

The endomorphism algebra of this unit is
\[
\mathrm{End}_{\mathcal{L}_R(B)}(\underline{HR}_B)
\simeq C^\ast(B;R).
\]
By \cite[Equation~7.4]{MNN17} and the discussion following it, there is an adjunction
\[
-\otimes_{C^\ast(B;R)}\underline{HR}_B
\colon
\mathrm{Mod}_{C^\ast(B;R)}
\;\leftrightharpoons\;
\mathcal{L}_R(B)
\colon
\mathrm{Hom}_{\mathcal{L}_R(B)}(\underline{HR}_B,-),
\]
which canonically makes the left adjoint $\mathrm{Hom}_{\mathcal{L}_R(B)}(\underline{HR}_B,-)$ symmetric monoidal. This adjunction restricts to an equivalence between perfect $C^\ast(B;R)$-modules and the thick subcategory generated by $\underline{HR}_B$. Set
\[
\Gamma_B\coloneqq
\mathrm{Hom}_{\mathcal{L}_R(B)}(\underline{HR}_B,-).
\]
For any $N\in\mathcal{L}_R(B)$, the natural comparison
\[
\Gamma_B(M)\otimes_{C^\ast(B;R)}\Gamma_B(N)
\longrightarrow
\Gamma_B(M\otimes_{\underline{HR}_B}N)
\]
is an equivalence when $M=\underline{HR}_B$. Since both sides are exact in $M$ and preserve retracts, it is an equivalence for every $M$ in the thick subcategory generated by the unit. In particular,
\[
\Gamma_B(\mathcal{M}_R(\mathcal{X}))
\otimes_{C^\ast(B;R)}\Gamma_B(N)
\longrightarrow
\Gamma_B(\mathcal{M}_R(\mathcal{X})
\otimes_{\underline{HR}_B}N)
\]
is an equivalence.

Furthermore, dualizability gives
\[
\begin{aligned}
\mathcal{M}_R(\mathcal{X})
\otimes_{\underline{HR}_B}
F_B(\mathcal{Y},\underline{HR}_B)
&\simeq
\mathcal{X}^\vee\wedge_B
F_B(\mathcal{Y},\underline{HR}_B)\\
&\simeq
F_B(\mathcal{X},F_B(\mathcal{Y},\underline{HR}_B))\\
&\simeq
F_B(\mathcal{X}\wedge_B\mathcal{Y},\underline{HR}_B).
\end{aligned}
\]
Taking $N=\mathcal{M}_R(\mathcal{Y})$ therefore yields
\[
\Gamma_B(\mathcal{M}_R(\mathcal{X}))
\otimes_{C^\ast(B;R)}
\Gamma_B(\mathcal{M}_R(\mathcal{Y}))
\simeq
\Gamma_B(\mathcal{M}_R(\mathcal{X}\wedge_B\mathcal{Y})).
\]

Finally, let $f\colon B\to\ast$ be the projection. Then
$f^\ast HR=\underline{HR}_B$, and the underlying $HR$-module of $\Gamma_B(M)$ is $f_\ast M$. Thus we have natural equivalences
\[
\begin{aligned}
\Gamma_B(\mathcal{M}_R(\mathcal{Z}))
&\simeq f_\ast F_B(\mathcal{Z},f^\ast HR)\\
&\simeq F(f_!\mathcal{Z},HR)\\
&\simeq C^\ast(\mathcal{Z};R).
\end{aligned}
\]
Here the second equivalence follows from
\cite[Proposition~6.8(5)]{ABG18}, and the last is the definition of cochains on a parametrized spectrum. Consequently, the natural comparison identifies with
\[
\begin{aligned}
C^\ast(\mathcal{X};R)
\otimes^L_{C^\ast(B;R)}C^\ast(\mathcal{Y};R)
&\simeq
\Gamma_B(\mathcal{M}_R(\mathcal{X}))
\otimes_{C^\ast(B;R)}
\Gamma_B(\mathcal{M}_R(\mathcal{Y}))\\
&\simeq
\Gamma_B(\mathcal{M}_R(\mathcal{X}\wedge_B\mathcal{Y}))\\
&\simeq C^\ast(\mathcal{X}\wedge_B\mathcal{Y};R).
\end{aligned}
\]
The proposition follows.
\end{proof}

Applying \Cref{prop: tensor product for connected base} gives the following corollary.

\begin{cor}\label{cor: equivariant chain tensor product}
Let $G$ be a topological group, let $R$ be a commutative ring, and let $X,Y$ be $G$-spectra. Suppose that the underlying nonequivariant spectrum of $X$ is finite and that at least one of the following conditions holds:
\begin{itemize}
\item $G$ is a finite $p$-group and $R$ is a commutative $\mathbb{F}_p$-algebra for some prime $p$.
\item $G$ is path-connected.
\end{itemize}
Then the natural map
\[
C^\ast_G(X;R)\otimes^L_{C^\ast(BG;R)}C^\ast_G(Y;R)
\longrightarrow C^\ast_G(X\wedge Y;R)
\]
is a quasi-isomorphism of $C^\ast(BG;R)$-modules.
\end{cor}

\begin{proof}
Consider the associated parametrized spectra
\[
\mathcal{X}\coloneqq X\times_G EG,
\qquad
\mathcal{Y}\coloneqq Y\times_G EG,
\qquad
B\coloneqq BG.
\]
Suppose first that $G$ is a finite $p$-group and $R$ is an $\mathbb{F}_p$-algebra. The augmentation ideal of $\mathbb{F}_p[G]$ is nilpotent. Its powers therefore give a finite filtration of $H^\ast(X;\mathbb{F}_p)$ with trivial-action successive quotients. Hence the monodromy representation
\[
\rho\colon\pi_1(BG,b)\cong\pi_0(G)
\longrightarrow
\mathrm{Aut}_{\mathbb{F}_p}(H^\ast(X;\mathbb{F}_p))
\]
is unipotent. Applying \Cref{prop: tensor product for connected base} with $S=\mathbb{F}_p$ shows that
\[
C^\ast_G(X;R)\otimes^L_{C^\ast(BG;R)}C^\ast_G(Y;R)
\longrightarrow C^\ast_G(X\wedge Y;R)
\]
is a quasi-isomorphism.

If $G$ is path-connected, then $\pi_1(BG)\cong\pi_0(G)=0$, so the monodromy action is trivial. Applying \Cref{prop: tensor product for connected base} with $S=\Z$ gives the same conclusion.
\end{proof}

The following example shows that the hypotheses of
\Cref{cor: equivariant chain tensor product} cannot be omitted in general.

\begin{ex}
Consider
\[
G=\Z_2,
\qquad
k=\mathbb{F}_3,
\qquad
X=\Sigma^\infty G_+,
\]
where $G$ acts on itself by translation and fixes the basepoint. Since $|G|$ is invertible in $k$, the positive-degree cohomology of $BG$ with coefficients in $k$ vanishes. Thus
\[
C^\ast(BG;k)\simeq k.
\]
The actions on $G$ and on $G\times G$, with the latter equipped with the diagonal action, are free. Consequently,
\[
\begin{aligned}
C^\ast_G(X;k)
&\simeq C^\ast(\Sigma^\infty(G/G)_+;k)
\simeq k,\\
C^\ast_G(X\wedge X;k)
&\simeq C^\ast(\Sigma^\infty((G\times G)/G)_+;k)
\simeq k^2.
\end{aligned}
\]
It follows that
\[
C^\ast_G(X;k)\otimes^L_{C^\ast(BG;k)}C^\ast_G(X;k)
\simeq k
\not\simeq k^2
\simeq C^\ast_G(X\wedge X;k).
\]
Thus the characteristic condition in
\Cref{cor: equivariant chain tensor product} cannot be dropped in general. This example also shows that the unipotence hypothesis in
\Cref{prop: tensor product for connected base} cannot be omitted: the monodromy representation is generated by
\[
\begin{pmatrix}
0&1\\
1&0
\end{pmatrix}
\in\mathrm{GL}_2(\mathbb{F}_3),
\]
which is not unipotent because it has eigenvalue $-1$.
\end{ex}

\subsection{$\Z_4$-equivariant cochains and almost $j$-complexes}

We now turn to $\Z_4$-equivariant cochains.

\begin{lem} \label{lem: BZ4 computation}
    We have a homotopy equivalence
    \[
    C^\ast(B\Z_4;\mathbb{F}_2) \simeq \mathcal{A} \coloneqq (\mathbb{F}_2[t,v],dv=t^2)
    \]
    of dgas over $\mathbb{F}_2$.
\end{lem}
\begin{proof}
    We follow the computations in \cite[Appendix B]{kang2025exotic}. Consider the augmentation ideal
    \[
    I = (x) \subset \mathbb{F}_2[x]/(x^4) \simeq \mathbb{F}_2[\Z_4].
    \]
    Then \cite[Lemma A.17]{MME19} gives 
    \[
    C_\ast(B\Z_4;\mathbb{F}_2)\simeq B\mathbb{F}_2[\Z_4]\simeq  \left( \bigoplus_{n=0}^\infty \underbrace{I[1]\otimes_{\mathbb{F}_2}\cdots\otimes_{\mathbb{F}_2}I[1]}_n,\partial_B \right)
    \]
    as $\mathbb{F}_2$-codgas, where the comultiplication is given by deconcatenation and the differential $\partial_B$ is given by
    \[
    \partial_B(r_1\otimes\cdots\otimes r_n) = \sum_{k=1}^{n-1} r_1 \otimes \cdots\otimes r_{i-1}\otimes r_i r_{r+1}\otimes r_{i+2}\otimes\cdots\otimes r_n. 
    \]
    Dualizing over $\mathbb{F}_2$ and renaming $x,x^2,x^3$ by $a,b,c$, respectively, gives
    \[
    C^\ast(B\Z_4;\mathbb{F}_2)\simeq (\mathbb{F}_2\langle a,b,c \rangle,d),
    \]
    where $da=0$, $db=a^2$, and $dc=ab+ba$; note that $a,b,c$ are noncommuting variables. Its cohomology is clearly generated by $a$ and $b^2+ac+ca$. Hence the dga morphism
    \[
        C^\ast(B\Z_4;\mathbb{F}_2) \rightarrow \mathcal{A},\quad a \mapsto t,\quad  b \mapsto v,\quad c \mapsto 0
    \]
    is a quasi-isomorphism and thus a homotopy equivalence. Note that its homotopy inverse cannot be realized as a strict dga morphism under these models, but rather as an $A_\infty$-morphism that involves nontrivial higher components.
\end{proof}

\begin{rem}
    While no detailed computation of $C^\ast(B\Z_4;\mathbb{F}_2)$ was previously done, \Cref{lem: BZ4 computation} itself was known for a long time. For example, it also appears in the discussions in \cite[Section 10]{HMZ18}.
\end{rem}

\begin{rem}
Observe that \Cref{lem: BZ4 computation} identifies the underlying associative dg algebras, but does not identify the $E_\infty$-structure on cochains with that arising from the strictly commutative multiplication on $\mathcal{A}$. Throughout this subsection, tensor products of cochain modules are taken using the $E_\infty$-structure on $C^\ast(B\Z_4;\mathbb{F}_2)$. When we abbreviate this operation by $\otimes^L_\mathcal{A}$, we mean the transported tensor product, not the ordinary derived tensor product over the commutative dg algebra $\mathcal{A}$.
\end{rem}

For a bounded cochain complex $(C,d_C)$ of $\mathbb{F}_2[\Z_4]$-modules, not necessarily free, applying a bar construction, similar to the one used in the proof of \Cref{lem: BZ4 computation}, gives the Borel model
\[
\mathrm{RHom}_{\mathbb{F}_2[\Z_4]}(\mathbb{F}_2,C)
\simeq
(C\otimes_{\mathbb{F}_2}\mathcal{A},d_C^{\Z_4}),
\]
where
\[
d_C^{\Z_4}
=
\id_C\otimes d_\mathcal{A}
+d_C\otimes\id_\mathcal{A}
+t(\id_C+j_C)\otimes\id_\mathcal{A}
+v(\id_C+j_C^2)\otimes\id_\mathcal{A}.
\]
Here $\Z_4=\langle j\rangle$, and $j_C$ denotes the action of $j$ on $C$. These models are semifree over $\mathcal{A}$: this follows from the nilpotent augmentation-ideal filtration of $\mathbb{F}_2[\Z_4]$ together with the bounded cochain-degree filtration of $C$.

The restriction map induced by the inclusion $\Z_2\hookrightarrow\Z_4$ is represented by the dg algebra morphism
\[
(B(\Z_2\hookrightarrow\Z_4))^\ast
\colon\mathcal{A}\longrightarrow\mathbb{F}_2[v];
\qquad
t\longmapsto0,\quad v\longmapsto v.
\]
We also note that
\[
\frac{\partial}{\partial t},
\frac{\partial}{\partial v}
\colon\mathcal{A}\longrightarrow\mathcal{A}
\]
are dg derivations of degree $-1$.

Given a dg $\mathcal{A}$-module $M$, choose a semifree replacement $\wt{M}\to M$ and a homogeneous free basis $B$. With respect to this basis, write
\[
d_{\wt{M}}=d_\mathcal{A}+Q,
\]
where $Q$ represents a degree-one $\mathcal{A}$-linear endomorphism. Define
\[
\Phi_t\coloneqq
\left[d_{\wt{M}},\frac{\partial}{\partial t}\right],
\qquad
\Phi_v\coloneqq
\left[d_{\wt{M}},\frac{\partial}{\partial v}\right].
\]
These are degree-zero $\mathcal{A}$-linear maps satisfying
\[
[d_{\wt{M}},\Phi_t]=[d_{\wt{M}},\Phi_v]=0,
\]
and hence are dg endomorphisms of $\wt{M}$. Adapting the arguments in
\cite[Proofs of Lemma~2.8 and Corollary~2.9]{Zemke19connectedsum}
shows that their homotopy classes are independent of $B$. Moreover, for any dg morphism $f\colon\wt{M}\to\wt{N}$ between semifree dg $\mathcal{A}$-modules, we have
\[
\Phi_t f\sim f\Phi_t
\qquad\text{ and }\qquad
\Phi_v f\sim f\Phi_v.
\]
This gives the following lemma.

\begin{lem}\label{lem: Phi t and Phi v are well defined}
For any dg $\mathcal{A}$-module $M$, the derived endomorphism classes $\Phi_t$ and $\Phi_v$ are independent of the choices of the homogeneous free basis and the semifree replacement. Furthermore, for any dg morphism $f\colon M\to N$, we have
\[
\Phi_t f\sim f\Phi_t
\qquad\text{ and }\qquad
\Phi_v f\sim f\Phi_v
\]
as derived morphisms from $M$ to $N$.
\end{lem}

\begin{proof}
It remains to explain independence of the semifree replacement. Given two semifree replacements $\wt{M}_1,\wt{M}_2\to M$, the identity of $M$ induces a homotopy equivalence
$f\colon\wt{M}_1\to\wt{M}_2$, well-defined up to homotopy. The relations
$\Phi_t f\sim f\Phi_t$  and $\Phi_v f\sim f\Phi_v$ show that this equivalence identifies the endomorphism classes obtained from the two replacements.
\end{proof}

\begin{rem}\label{rem: Phi v for Z2 is well defined}
The same argument shows that, for any dg $\mathbb{F}_2[v]$-module $M$, the map $\Phi_v$ determines a well-defined derived endomorphism class of $M$.
\end{rem}

For perfect dg $\mathbb{F}_2[v]$-modules, $\Phi_v$ satisfies the following torsion property.

\begin{lem}\label{lem: phi v over F2 v is 1-torsion}
For any perfect dg $\mathbb{F}_2[v]$-module $C$, we have
\[
v\Phi_v\sim0.
\]
\end{lem}

\begin{proof}
Since $\mathbb{F}_2[v]$ is a PID and $\deg(v)>0$, the perfect module $C$ is represented, up to homotopy equivalence, by a finite direct sum $C_0$ of grading shifts of free summands with zero differential and elementary complexes
\[
E_m=
\left[
\mathbb{F}_2[v]\cdot x
\xrightarrow{v^m}
\mathbb{F}_2[v]\cdot y
\right],
\qquad m>0.
\]
The map $\Phi_v$ vanishes on the free summands. On $E_m$, it is given by
\[
\Phi_v(x)=mv^{m-1}y,
\qquad
\Phi_v(y)=0.
\]
The $\mathbb{F}_2[v]$-linear map $H$ defined by
\[
H(x)=mx,
\qquad
H(y)=0
\]
satisfies $dH+Hd=v\Phi_v$. Thus $v\Phi_v$ is nullhomotopic on each summand, and hence on $C_0$. Naturality of $\Phi_v$ under the homotopy equivalence $C_0\simeq C$ gives the result.
\end{proof}

We next explain how to obtain an almost $j$-complex from a module over $C^\ast(B\Z_4;\mathbb{F}_2)$.

\begin{defn}
A perfect dg $C^\ast(B\Z_4;\mathbb{F}_2)$-module, represented by a dg $\mathcal{A}$-module $M$, is called \emph{SWF-like} if its derived truncation
\[
M_{\Z_2}\coloneqq
M\otimes^L_\mathcal{A}\mathbb{F}_2[v]
\]
is SWF-like as a dg $\mathbb{F}_2[v]$-module.
\end{defn}

Here the truncation is extension of scalars along the restriction map $\mathcal{A}\to\mathbb{F}_2[v]$. Since this map commutes with differentiation with respect to $v$, extension of scalars carries $\Phi_v$ on $M$ to $\Phi_v$ on $M_{\Z_2}$.

\begin{lem}\label{lem: t squares to v and v squares to 0}
Let $M$ be an SWF-like dg $\mathcal{A}$-module. Then the derived endomorphism
$j_M\coloneqq\id_M+\Phi_t$ satisfies
\[
j_M^2\sim\id_M+\Phi_v.
\]
Furthermore, $\Phi_v^2\sim0$.
\end{lem}

\begin{proof}
We may assume that $M$ is semifree. With respect to a homogeneous free basis, write $d_M=d_\mathcal{A}+Q$, where $Q$ represents a degree-one $\mathcal{A}$-linear endomorphism. Then
\[
\Phi_t=\frac{\partial Q}{\partial t},
\qquad
\Phi_v=\frac{\partial Q}{\partial v}.
\]
Since $d_\mathcal{A}=t^2\frac{\partial}{\partial v}$, the identity $d_M^2=0$ gives
\[
t^2\frac{\partial Q}{\partial v}+Q^2=0.
\]

For each integer $k\geq0$, let $D_t^{(k)}$ denote the $k$th Hasse derivative with respect to $t$. Thus
\[
D_t^{(k)}(t^mv^n)=\binom{m}{k}t^{m-k}v^n,
\qquad
D_t^{(k)}(xy)=\sum_{i=0}^kD_t^{(i)}(x)D_t^{(k-i)}(y).
\]
Note that $D^{(1)}_t = \frac{\partial}{\partial t}$ and $D^{(0)}_t = \mathrm{Id}$. Applying $D_t^{(2)}$ to the preceding identity gives
\[
\begin{aligned}
0
&=
t^2D_t^{(2)}\left(\frac{\partial Q}{\partial v}\right)
+\frac{\partial Q}{\partial v}
+(D_t^{(2)}Q)Q
+\left(\frac{\partial Q}{\partial t}\right)^2
+Q(D_t^{(2)}Q)\\
&=
t^2\frac{\partial(D_t^{(2)}Q)}{\partial v}
+(D_t^{(2)}Q)Q+Q(D_t^{(2)}Q)
+\Phi_t^2+\Phi_v.
\end{aligned}
\]
For any homogeneous $\mathcal{A}$-linear matrix $P$, the commutator with $d_\mathcal{A}$ is represented by
$t^2\frac{\partial P}{\partial v}$. Therefore the map $H$ represented by $D_t^{(2)}Q$ satisfies
\[
d_MH+Hd_M=\Phi_t^2+\Phi_v.
\]
It follows that $\Phi_t^2\sim\Phi_v$, and hence
\[
j_M^2=(\id_M+\Phi_t)^2
\sim\id_M+\Phi_v.
\]
Similarly, applying the second Hasse derivative with respect to $v$ shows that the map $H_v$ represented by $D_v^{(2)}Q$ satisfies
\[
d_MH_v+H_vd_M=\Phi_v^2.
\]
Thus $\Phi_v^2\sim0$.
\end{proof}

The proofs of \Cref{lem: Phi t and Phi v are well defined,lem: t squares to v and v squares to 0} also have intrinsic interpretations in terms of Hochschild cohomology. The maps $\Phi_t$ and $\Phi_v$ are the images of the degree $-1$ derivation classes
\[
\left[\frac{\partial}{\partial t}\right],
\left[\frac{\partial}{\partial v}\right]
\in H^{-1}\mathrm{Der}_{\mathbb{F}_2}(\mathcal{A})
\]
under the natural map\footnote{A closed derivation of degree $k$ determines a Hochschild cocycle of total degree $k+1$.}
$$
H^{-1}\mathrm{Der}_{\mathbb{F}_2}(\mathcal{A})
\xrightarrow{\hspace{.5em}\phantom{-\otimes^L_{\mathcal{A}}M}\hspace{.5em}}
HH^0(\mathcal{A},\mathcal{A})\cong
H^0\mathrm{RHom}_{\mathcal{A}\otimes_{\mathbb{F}_2}\mathcal{A}^{\mathrm{op}}}
(\mathcal{A},\mathcal{A})
\xrightarrow{\hspace{.5em}-\otimes^L_{\mathcal{A}}M\hspace{.5em}}
H^0\mathrm{RHom}_{\mathcal{A}}(M,M);
$$
note that this fact already recovers \Cref{lem: Phi t and Phi v are well defined}. Observe that the characteristic map
\[
HH^\ast(\mathcal{A},\mathcal{A})
\cong
H^\ast\mathrm{RHom}_{\mathcal{A}\otimes_{\mathbb{F}_2}\mathcal{A}^{\mathrm{op}}}
(\mathcal{A},\mathcal{A})
\xrightarrow{\hspace{.5em}-\otimes^L_\mathcal{A}M\hspace{.5em}}
H^\ast\mathrm{RHom}_\mathcal{A}(M,M)
\]
takes Hochschild cup products to compositions of derived endomorphisms. Since we have
\[
HH^\ast(C^\ast(BG,k),C^\ast(BG,k))\simeq HH^\ast(k[G],k[G])\simeq k[G]\otimes_k H^\ast(BG;k),
\]
for any prime $p$, finite abelian $p$-group $G$, and field $k$ of characteristic $p$ by \cite[Theorem 3.14]{briggs2017infinity} and \cite[Theorem 2.1]{cibils1997hochschild}, we see that
\[
HH^\ast(\mathcal{A},\mathcal{A})
\cong
\mathbb{F}_2[\Z_4]\otimes_{\mathbb{F}_2} H^\ast(B\Z_4;\mathbb{F}_2) 
\cong
\mathbb{F}_2[\epsilon,R,S]/(\epsilon^4,R^2),
\]
where
\[
\deg(\epsilon)=0,
\qquad
\deg(R)=1,
\qquad
\deg(S)=2.
\]
The derivation classes
$\left[\frac{\partial}{\partial t}\right]$ and
$\left[\frac{\partial}{\partial v}\right]$
map to $\epsilon$ and $\epsilon^2$, respectively. Their relations therefore recover
\Cref{lem: t squares to v and v squares to 0}.

\begin{lem}\label{lem: extracting j from eqv cochain}
Let $(C,d_C)$ be a bounded cochain complex of $\mathbb{F}_2[\Z_4]$-modules, and let
\[
M\coloneqq
\mathrm{RHom}_{\mathbb{F}_2[\Z_4]}(\mathbb{F}_2,C).
\]
Then the endomorphism $j^\ast$ of $M$ induced by the action of the generator $j\in\Z_4$ is homotopic to $j_M$.
\end{lem}

\begin{proof}
As above, we may use the semifree Borel model
\[
M\simeq
(C\otimes_{\mathbb{F}_2}\mathcal{A},d_C^{\Z_4}).
\]
In this model, differentiation of the differential gives
\[
\left[d_C^{\Z_4},\frac{\partial}{\partial t}\right]
=
(\id_C+j_C)\otimes\id_\mathcal{A}.
\]
Consequently, $j_M$ is represented by
\[
\id_{C\otimes_{\mathbb{F}_2}\mathcal{A}}
+
\left[d_C^{\Z_4},\frac{\partial}{\partial t}\right]
=
j_C\otimes\id_\mathcal{A},
\]
which is the endomorphism induced by the action of $j$. The lemma follows.
\end{proof}

Motivated by \Cref{lem: extracting j from eqv cochain}, for a dg $\mathcal{A}$-module $M$ we set
\[
j_M^{\Z_2}\coloneqq j_M\otimes\id
\in
H^0\mathrm{RHom}_{\mathbb{F}_2[v]}(M_{\Z_2},M_{\Z_2})
\]
and define
\[
\mathrm{AjCplx}(M)
\coloneqq
(M_{\Z_2},\widehat{j_M^{\Z_2}}).
\]
By \Cref{lem: t squares to v and v squares to 0}, if $M$ is SWF-like, then $\mathrm{AjCplx}(M)$ is an almost $j$-complex.

\begin{prop}\label{prop: tensor product formula}
Let $M$ and $N$ be SWF-like dg $C^\ast(B\Z_4;\mathbb{F}_2)$-modules. Then there is a homotopy equivalence of almost $j$-complexes
\[
\mathrm{AjCplx}(M)\otimes\mathrm{AjCplx}(N)
\simeq
\mathrm{AjCplx}
\left(
M\otimes^L_{C^\ast(B\Z_4;\mathbb{F}_2)}N
\right).
\]
\end{prop}

\begin{proof}
By the perfectness of $M$ and $M'$ and a thick-unit equivalence argument used in the proof of
\Cref{prop: tensor product for connected base}, there exist bounded complexes $C,C'$ of finite-dimensional $\mathbb{F}_2[\Z_4]$-modules such that
\[
M\simeq
\mathrm{RHom}_{\mathbb{F}_2[\Z_4]}(\mathbb{F}_2,C)
\qquad\text{ and }\qquad
N\simeq
\mathrm{RHom}_{\mathbb{F}_2[\Z_4]}(\mathbb{F}_2,C').
\]
Write $j_C$ and $j_{C'}$ for the respective actions of the generator. Using the Borel models, we may write
\[
\begin{aligned}
d_M
&=
\id_C\otimes d_\mathcal{A}
+d_C\otimes\id_\mathcal{A}
+t(\id_C+j_C)\otimes\id_\mathcal{A}
+v(\id_C+j_C^2)\otimes\id_\mathcal{A},\\
d_N
&=
\id_{C'}\otimes d_\mathcal{A}
+d_{C'}\otimes\id_\mathcal{A}
+t(\id_{C'}+j_{C'})\otimes\id_\mathcal{A}
+v(\id_{C'}+j_{C'}^2)\otimes\id_\mathcal{A}.
\end{aligned}
\]
After truncation to $\mathbb{F}_2[v]$, the two models are clean, since
\[
\Phi_{v,M_{\Z_2}}=(\id_C+j_C^2)\otimes\id,
\qquad
\Phi_{v,N_{\Z_2}}=(\id_{C'}+j_{C'}^2)\otimes\id,
\]
and both maps square to zero.

Thus $\mathrm{AjCplx}(M)\otimes\mathrm{AjCplx}(N)$ is represented by
\[
\left(
C\otimes_{\mathbb{F}_2}C'\otimes_{\mathbb{F}_2}\mathbb{F}_2[v],
d_\otimes,j_\otimes
\right),
\]
where
\[
\begin{aligned}
d_\otimes
&=
(d_C\otimes\id+\id\otimes d_{C'})\otimes\id\\
&\qquad
+v\bigl(
j_C^2\otimes\id+\id\otimes j_{C'}^2
+(\id+j_C^2)\otimes(\id+j_{C'}^2)
\bigr)\otimes\id\\
&=
(d_C\otimes\id+\id\otimes d_{C'})\otimes\id
+v(\id+j_C^2\otimes j_{C'}^2)\otimes\id.
\end{aligned}
\]
On the hat truncation, identified with $C\otimes_{\mathbb{F}_2}C'$, the action is
\[
j_\otimes=j_C\otimes j_{C'}.
\]

The same thick-unit argument as in the proof of
\Cref{prop: tensor product for connected base}
gives a natural equivalence
\[
M\otimes^L_{C^\ast(B\Z_4;\mathbb{F}_2)}N
\simeq
\mathrm{RHom}_{\mathbb{F}_2[\Z_4]}
(\mathbb{F}_2,C\otimes_{\mathbb{F}_2}C'),
\]
where $C\otimes_{\mathbb{F}_2}C'$ carries the diagonal $\Z_4$-action. Its Borel differential is
\[
\begin{aligned}
d_{C\otimes_{\mathbb{F}_2}C'}^{\Z_4}
&=
\id\otimes\id\otimes d_\mathcal{A}
+(d_C\otimes\id+\id\otimes d_{C'})\otimes\id\\
&\qquad
+t(\id+j_C\otimes j_{C'})\otimes\id
+v(\id+j_C^2\otimes j_{C'}^2)\otimes\id.
\end{aligned}
\]
Therefore, by \Cref{lem: extracting j from eqv cochain}, the almost $j$-complex associated with this module is represented by
\[
\left(
C\otimes_{\mathbb{F}_2}C'\otimes_{\mathbb{F}_2}\mathbb{F}_2[v],
d'_\otimes,j'_\otimes
\right),
\]
where
\[
d'_\otimes
=
(d_C\otimes\id+\id\otimes d_{C'})\otimes\id
+v(\id+j_C^2\otimes j_{C'}^2)\otimes\id
\]
and, on the hat truncation,
\[
j'_\otimes=j_C\otimes j_{C'}.
\]
Since $d_\otimes=d'_\otimes$ and $j_\otimes=j'_\otimes$, the proposition follows.
\end{proof}

\bibliographystyle{alpha}
\bibliography{tex}

\end{document}